\documentclass[reqno]{amsart}
\usepackage[usenames, dvipsnames]{color}
\usepackage{amsthm}
\usepackage{amsmath}
\usepackage{amssymb}
\usepackage{amscd}
\usepackage{graphics}
\usepackage{latexsym}
\usepackage{stmaryrd}
\usepackage{hyperref}
\usepackage{empheq}

\theoremstyle{plain}
\newtheorem{theorem}{Theorem}[section]
\newtheorem{thm}[theorem]{Theorem}

\newtheorem{lem}[theorem]{Lemma}
\newtheorem{prop}[theorem]{Proposition}

\newtheorem{conj}[theorem]{Conjecture}

\theoremstyle{definition}
\newtheorem{defn}[theorem]{Definition}

\newtheorem{rmk}[theorem]{Remark}

\theoremstyle{remark}

\definecolor{titlecol}{named}{BrickRed}
\definecolor{headcol}{named}{Violet}
\definecolor{seccol}{named}{Red}
\definecolor{sseccol}{named}{Bittersweet}
\definecolor{pbcol}{named}{Black}
\definecolor{sncol}{named}{Brown}
\definecolor{acol1}{named}{Red}
\definecolor{acol2}{named}{Apricot}

\newcommand\bR{{\mathbb R}}
\newcommand\bZ{{\mathbb Z}}

\def\bR{{\mathbb R}}
\def\bZ{{\mathbb Z}}

\begin{document}

\title{The Calabi flow with rough initial data}
\date{\today}

\author{Weiyong He}
\address{Department of Mathematics, University of Oregon, Eugene, OR 97403. }
\email{whe@uoregon.edu}

\author{Yu Zeng}
\address{Department of Mathematics, University of Rochester, Rochester, NY 14627.}
\email{yu.zeng@rochester.edu}

\begin{abstract}
In this paper, we prove that there exists a dimensional constant $\delta > 0$ such that given any background K\"ahler metric $\omega$, the Calabi flow with initial data $u_0$ satisfying
\begin{equation*}
\partial \bar \partial u_0 \in L^\infty (M) \text{ and } (1- \delta )\omega < \omega_{u_0} < (1+\delta )\omega, 
\end{equation*}
admits a unique short time solution and it becomes smooth immediately, where $\omega_{u_0} : = \omega +\sqrt{-1}\partial \bar\partial u_0$. The existence time depends on initial data $u_0$ and the metric $\omega$. As a corollary, we get that Calabi flow has short time existence for any initial data satisfying
\begin{equation*}
\partial \bar \partial u_0 \in C^0(M) \text{ and } \omega_{u_0} > 0, 
\end{equation*}  
which should be interpreted as a ``continuous K\"ahler metric". A main technical ingredient  is  Schauder-type estimates for biharmonic heat equation on Riemannian manifolds with time weighted H\"older norms. 
\end{abstract}

\maketitle
\tableofcontents

\section{Introduction}

In 1980s, Calabi(\cite{C1}) initiated a program to find a canonical representative of any given K\"ahler class on a closed K\"ahler manifold. To this end, he suggests to  seek  ``extremal" metrics as critical points of the $L^2$ norm of curvature. As a natural geometric flow approach to the existence of such metrics, he introduced the Calabi flow, given by the equation
\begin{equation}\label{eqn8001}
\frac{\partial \varphi }{\partial t} = R_\varphi - \underline R, 
\end{equation} 
where $R_\varphi$ is the scalar curvature of the metric $(g_\varphi)_{i\bar j} = g_{i\bar j} + \frac{\partial^2 \varphi}{\partial z_i \partial \bar z_j}$ and $\underline R$ is the integral average of $R_\varphi$, a topological constant. This equation is a fourth order nonlinear parabolic equation. Given any smooth initial K\"ahler potential, equation $(\ref{eqn8001})$ has short time solution by standard theory of parabolic equation. Chen has conjectured that the Calabi flow exists for all time for any smooth initial K\"ahler potential.
\begin{conj}[Chen] The Calabi flow exists for all time ($t>0$) for any smooth initial data.
\end{conj} 
Except for few special cases , the long time existence conjecture remains wild open and it is out of reach in the current status. A related problem  is the short time existence for rough initial datum, with rather weak regularity.  In this paper, we aim to seek an optimal regularity of initial datum for which the Calabi flow admits short time solution. Such problem was first studied by Chen-He(\cite{CH}) and they showed that the Calabi flow has short time solution for $C^{3,\alpha}$ initial datum. Later, He(\cite{He}) showed that if the initial data has  $C^{2,\alpha}$ norm smaller than some constant depending on (high order derivatives of) $g$, then the Calabi flow could start.  Chen has made the following conjecture,

\begin{conj}[Chen \cite{Ch}]
Given any initial $C^{1,1}$ K\"ahler potential, the Calabi flow admits short time solution and it immediately becomes smooth when $t>0$. 
\end{conj}

While we cannot fully confirm Chen's conjecture in this paper, given any K\"ahler metric $\omega$, we will prove the short time existence of Calabi flow for any initial data:
\begin{equation}
\partial \bar \partial u_0 \in L^\infty (M) \text{ and } (1-\delta) \omega < \omega_{u_0} < (1+\delta ) \omega, 
\end{equation}
for some dimensional constant $\delta > 0$, where $\omega_{u_0} : = \omega + \sqrt{-1}\partial \bar \partial u_0$. This initial condition should be understood as an ``$L^\infty$-K\"ahler metric" in a small $L^\infty$ neighborhood of a smooth K\"ahler metric. More precisely, we have the main theorem of the paper as follows. 

\begin{thm}[Theorem $\ref{thm802}$ and Theorem $\ref{thmsmooth}$]\label{techthm}
On a closed K\"ahler manifold $(M, \omega)$, there exists a dimensional constant $\delta > 0$ such that, for any $u_0$ satisfying 
$$\partial \bar \partial u_0 \in L^\infty(M) \text{ and } (1-\delta) \omega < \omega_{u_0} < (1+\delta ) \omega, $$ 
there exists a constant $T_{g, u_0}>0$ such that the Calabi flow with initial value $u_0$ admits a short time unique solution $\varphi \in C^\infty\big(M \times (0, T_{g, u_0})\big)$. Moreover, for any $\gamma \in (0,1)$
$$\lim_{t \rightarrow 0^+ }  \|\varphi(t) - u_0\|_{C^{1,\gamma}(M)}  = 0, $$  
and for any integers $k,l \geq 0$,
$$\limsup_{t \rightarrow 0^+ }  t^{\frac{k}{4}+ l}\big\||(\frac{\partial }{\partial t})^l \nabla_g^k\partial \bar \partial \varphi(t)|_g \big\|_{C^0(M)}  \leq C_{\frac{k}{4}+l , n} \big\||\partial \bar \partial u_0|_g\big\|_{L^\infty(M)}.$$
If we further assume that $\partial \bar \partial u_0 \in C^0(M)$, then
$$\lim_{t \rightarrow 0^+ }  \|\varphi(t) - u_0\|_{C^{1}(M)}  + \|\partial \bar \partial \varphi(t) - \partial \bar \partial u_0\|_{C^0(M)} = 0, $$ 
and for any $k, l \geq 0$ with $l+ k > 0$, 
$$\lim_{t \rightarrow 0^+ }  t^{\frac{k}{4}+ l}\|(\frac{\partial }{\partial t})^l \nabla_g^k\partial \bar \partial \varphi(t)\|_{C^0(M)}  =0.$$
\end{thm}

\begin{rmk}\label{rmkuniform}
This theorem asserts the existence of a uniform sized neighborhood around the center metric $g$ such that for any initial ``$L^\infty$-K\"ahler metric" in the neighborhood, the Calabi flow admits a unique short time solution. It has uniform size in the sense that it consists of all $\sup_{M } |\partial \bar \partial u_0|_{g} < \delta $ for some dimensional constant $\delta >0$ independent of $g$. This is the key for results like Theorem $\ref{mainthm}$ to hold. The existence time $T_{g, u_0}$ may vary depending on high order derivatives of $g$ and $ u_0$. 
\end{rmk}

Following the last remark,  these uniform sized neighborhoods, varying the center metric, cover all the metrics satisfying
\begin{equation}\label{asm}
\partial \bar \partial u_0 \in C^0(M) \text{ and } \omega_{u_0} > 0, 
\end{equation} 
which should be understood as a ``continuous K\"ahler metric". One can prove this fact by a standard smooth approximation argument. Therefore, we have the following theorem.

\begin{thm}\label{mainthm}
On a closed K\"ahler manifold $(M, \omega)$, given any initial data $u_0$ satisfying $(\ref{asm})$, Calabi flow admits short time unique solution $\varphi \in C^\infty(M \times (0,T))$ for some $T > 0$ depends on $g, u_0$. Moreover, 
$$\lim_{t \rightarrow 0^+ } \big( \|\varphi(t) - u_0\|_{C^1(M)} + \|\partial \bar \partial \varphi(t) - \partial \bar \partial u_0\|_{C^0(M)} \big) = 0, $$  
and for any integers $k,l \geq 0$ with $k+l >0$,
$$\lim_{t \rightarrow 0^+ }  t^{\frac{k}{4}+ l}\|(\frac{\partial }{\partial t})^l \nabla^k\partial \bar \partial \varphi(t)\|_{C^0(M)}  = 0.$$
\end{thm}

\begin{rmk}Recently Berman-Darvas-Lu \cite{BDL} has proposed to study the Calabi flow with initial datum in $\mathcal{E}^1$ (some energy class of $\omega$-plurisubharmoninc functions, which is a metric completion of the space of K\"ahler potentials), built upon J. Streets' work \cite{streets, streets1} on minimizing movement of Mabuchi's $\mathcal{K}$-energy; the later could be viewed as a weak Calabi flow and defines a long time path in $\mathcal{E}^1$ and exhibits nice large time  properties in metric topology, see \cite{streets, streets1, BDL}. In particular, when the Calabi flow has a smooth solution, it coincides such a path. However, such a long time minimizing path has very weak regularity and  its regularity is way beyond the authors' scope at the moment. 
\end{rmk}

By treating the Calabi flow equation as a small perturbation of the biharmonic heat equation (for short time),  we construct its short time solution near the solution of the initial value problem of the biharmonic heat equation via a fixed point argument. This type of argument for short time existence of parabolic equations is rather standard. For this argument to work, in particular for rough initial conditions, it is crucial to work on proper functional spaces adapted to the problem.  Koch-Lamm \cite{KL} considered parabolic equations and systems (geometric flows) with rough initial data, and they designed the functional spaces based on basic notions in harmonic analysis, certain weighted Sobolev spaces (in space time). 

Our observation is that one can actually work on classical H\"older norms, even for rough initial data. The essential point is to use H\"older norms with time weights and choose the weights correctly, according to the parabolic scaling. We introduce the \emph{weighted parabolic H\"older spaces}(see $Y_T$, $X_T$ in Section $\ref{sec3.3}$) and establish a Schauder estimate of the linear parabolic equation between these weighted parabolic H\"older spaces. These results are analogous to the standard parabolic Schauder theory between classical parabolic H\"older spaces ( $C^{\alpha, \frac{\alpha}{4}}(M \times [0,T])$ and $C^{4+\alpha, 1+\frac{\alpha}{4}}(M \times [0,T])$). The method is similar to the classical approach. But the results seem to be new for linear parabolic equations, see Theorem \ref{thm701}.

In these weighted spaces, instead of assuming ``parabolic H\"older norm at time $t$" bounded for all $t \in [0,T]$ as in the standard case, we put appropriate time weights on each component allowing them to blow up at certain rates when $t \rightarrow 0^+$. These prescribed blow up rates come naturally from the solution of initial value problem. For instance, if given initial value $u_0 \in C^2(M)$, then the solution of the initial value problem of biharmonic heat equations can be viewed as a smooth approximation of $u_0$ as $t \rightarrow 0^+$. Its derivatives up to second order approach the corresponding derivatives of $u_0$ and thus are bounded as $t \rightarrow 0^+$. Its derivatives of order $(k+2)$th blow up at rate $t^{-\frac{k}{4}}$ as $t \rightarrow 0^+$ for $k \geq 0$. For more details on the weighted spaces, see Section $\ref{sec3.3}$.

Another key feature of our approach is that these estimates are uniform (independent of center metric) for the initial value problem as well as the nonhomogeneous problem on time interval $(0,T)$, if $T$ sufficiently small.  This observation leads to the existence of a  neighborhood with uniform size around any center metric mentioned in Theorem $\ref{techthm}$ and Remark $\ref{rmkuniform}$. For the initial value problem, conceptually as a smooth approximation of the initial data as $t \rightarrow 0^+$, its solution should be controlled solely by the initial data for $t>0$ sufficiently small(see Theorem $\ref{thm001}$ and Theorem $\ref{thm3.9}$). While for the nonhomogeneous problem, surprisingly we find out that the parabolic Schauder estimate between weighted parabolic H\"older spaces becomes independent of $g$ when $T$ is sufficiently small with respect to the $C^{3,\alpha}$-harmonic radius of $g$ (see Theorem $\ref{thm702}$ and its proof in Section $\ref{sec3.4}$). The appropriate weights play an essential role here and the desired uniform estimates follow by the scaling invariance. 

The smoothness of the short time solution to the Calabi flow for $t>0$ follows from the standard parabolic Schauder estimate by restricting to time subinterval $[\epsilon, T]$ for any $\epsilon >0$. Uniqueness for nonlinear parabolic equations with rough initial data is always a very important problem and it can be subtle. It is essential for applications in particular. In our case the Calabi flow enjoys very nice properties and it always has a unique solution. This is essentially due to Calabi-Chen \cite{CC} that the Calabi flow decreases the distance between K\"ahler metrics in the space of K\"ahler metrics. 

Our proof is essentially based on weighted Schauder-type estimates for linear parabolic equations, developed in the paper. To construct a short time solution for the Calabi flow (a nonlinear parabolic equation), one applies a contracting mapping (fixed point) argument for the corresponding linearized equation.  The Calabi flow equation, or rather the structure of scalar curvature only plays a particular role in the construction of a desired contracting mapping. Such contracting mapping (fixed point) argument works for most geometric flows  (nonlinear parabolic equations), once the corresponding linear theory has been developed; for example see \cite{KL}. Hence our method should work in a much more general setting for geometric flows, in particular for higher order parabolic equations. For brevity of presentation and also for our main interest, we will limit our consideration to the Calabi flow. 

As an application, we study the regularity of a weak solution of its elliptic version, namely, the constant scalar curvature equation. As a direct consequence, we have the following, 

\begin{thm}\label{kminimizer}Suppose a $L^\infty$ metric $\omega_u$ minimizes K-energy, and if $\omega_u$ is in $\delta$-neighborhood of any smooth K\"ahler metric in $L^\infty$ norm, then it is a smooth K\"ahler metric with constant scalar curvature. In particular, if a K-energy minimizer is a continuous K\"ahler metric, then it is smooth K\"ahler metric with constant scalar curvature. 
\end{thm}

This confirms partially a conjecture of Chen, which asserts that a $C^{1, 1}$ minimizer of $K$-energy is a smooth constant scalar curvature K\"ahler metric. 

Our paper is organized as follows. In Section $\ref{sec2}$, we introduce some basic notations and the notion of \emph{harmonic radius} from Riemannian geometry. In Section $\ref{sec3}$, we study the linear equation, namely the biharmonic heat equation. We start in Section $\ref{sec3.1}$ with an overview of the \emph{biharmonic heat kernel} on closed manifolds and in Section $\ref{sec32}$ some uniform integral bounds on the biharmonic heat kernel. Based on estimates of the biharmonic heat kernel, we then study the initial value problem and the nonhomogeneous problem in Section $\ref{sec3.2}$ and Section $\ref{sec3.3}$ \& $\ref{sec3.4}$ respectively. In particular, in Section $\ref{sec3.3}$ we introduce the \emph{weighted parabolic H\"older spaces} and in Section $\ref{sec3.4}$ we derive a uniform Schauder estimate between weighted spaces on time interval $(0,T)$ for $T$ sufficiently small. In Section $\ref{sec4}$, we construct the short time solution to the Calabi flow via a fixed point argument on weighted parabolic H\"older space. In Section $\ref{sec5}$, we prove the solution in weighted space is actually smooth whenever $t>0$ and also we prove Theorem $\ref{kminimizer}$.

In Appendix $\ref{secA}$ \& $\ref{secB}$, we provide a self-contained proof of the existence and estimates of biharmonic heat kernel on closed manifolds stated in Section $\ref{sec3.1}$.  The fundamental solution of general high order parabolic operators on domains of $\bR^n$ has been studied intensely back in 1950's(see \cite[p335]{Fr} and the references therein), but none of them dealt with the case on closed manifolds. Even though it essentially follows the same technique, \emph{the parametrix method}, it's still quite tricky and involving to confirm the same expectation on closed manifolds. These results are classical and well-known to experts. But since we cannot find a direct reference of these results, we take this opportunity to write down the arguments in detail. Our argument essentially follows the books (\cite{Ei}, \cite{Fr}).

\subsection*{Acknowledgement} Both authors would like to thank their PhD adviser Prof. Chen, Xiuxiong for suggesting to work on this problem as well as his warm encouragements. Part of the work was done during this summer when the second author was a graduate student at Stony Brook University. He is very grateful for the summer support from Department of Mathematics at Stony Brook University. Especially, he wants to thank Prof. LeBrun, Claude for his efforts in coordinating various funds to make the summer support possible. The second author also wants to thank Yuanqi Wang for many helpful discussions on the parabolic Schauder theory as well as on the parametrix method, Song Sun and Gao Chen for discussions on harmonic radius. Both authors would like to thank Prof. Angenent, Sigurd and Prof. Chen, Jingyi for their interests on this work. The first author is supported in part by NSF, award no. 1611797.

\section{Preliminaries}\label{sec2}

\subsection{Notations}

Suppose $M$ is a closed $n$-dimensional Riemannian manifold with a fixed smooth background metric $g$. Denote $\nabla$ to be the Levi-Civita connection with respect to metric $g$. Locally in a coordinate chart $\{x_i\}_{i=1}^n$, we also have ordinary differentials with respect to the coordinate functions. We denote it as $D$. They differs by the Christoffel symbols which depends on (derivatives of) $g$ only. 

Given any multi index $\alpha= (\alpha_1, \alpha_2, \cdots, \alpha_n) \in \bZ_{\geq 0}^n$, in local coordinate $\{x_i\}_{i=1}^n$ we denote
\begin{equation}
D^\alpha := (\frac{\partial }{\partial x_1})^{\alpha_1}(\frac{\partial }{\partial x_2})^{\alpha_2} \cdots (\frac{\partial }{\partial x_n})^{\alpha_n}.
\end{equation}
We also denote $|\alpha| := \sum_{i=1}^n \alpha_i $ and sum of two multi indices $\alpha$ and $\beta$ by $\alpha+ \beta$ given by adding up each component.

For two points $x, y$ in a single local chart, we use $|x-y|$ to denote the Euclidean distance while we use $\rho(x, y)$ to denote the geodesic distance with respect to $g$. 

Given any tensor $T$ on $M$. We denote the $C^0$ norm of $T$ as
\begin{equation}
\|T\|_{C^0(M) } : = \sup_{M} |T|_g. 
\end{equation}
We denote $C^k$ norm of $T$ for integer $k >0$ as 
\begin{equation}
\|T\|_{C^k(M)} = \sum_{i=0}^k \|\nabla^i T\|_{C^0(M)}
\end{equation}
Fix a finite open cover $\{(U_\nu, p_\nu)\}_{\nu}$ of $M$,  where for each $\nu$, $U_\nu $ is a open subset of $M$ with point $p_\nu \in U_\nu$. Then we denote the $C^\alpha$ semi norm of $T$ as
\begin{equation}
[T]_{C^\alpha(M)} = \sup_{\nu} \sup_{x \neq y \in U_\nu} \frac{|T(x) - T(y)|_{g(p_\nu)}}{|x-y|^\alpha}.  
\end{equation} 

Without further notice, every constant $C>0$ in this paper is a dimensional constant unless specified.

\subsection{Normal coordinates and Harmonic radius}

Suppose $M$ is a smooth closed manifold and $\text{dim}_{\bR} M = n$ and $g$ is a smooth Riemannian metric on M. We denote $k_0 = \max_M |\text{Rm}(g)|_g$, $i_0 = $ the injectivity radius of $(M, g)$. Given any point $p \in M$, we know that the exponential map $\exp_p: \bR^n \rightarrow M$ maps the Euclidean ball $B_r(0)$ to the geodesic ball $B_r(p, g)$ diffeomorphically  for any $r < i_0$. As a result, the exponential map gives rise to a local coordinate called normal coordinate. In the rest of this paper, we will use $x \in B_r(0)$ to denote a point in the Euclidean ball as well as its image under the exponential map when there's no confusion.   

\begin{lem}\label{lem201}
There exists a constant $r_0 > 0 $ depending on $k_0$, $i_0$ and dimension $n$ such that at any point $p \in M$ under the normal coordinate 
\begin{equation}
\frac{1}{2}\delta_{i j} \leq g_{ij}(x) \leq 2 \delta_{ij} \text{ for any } x \in B_{r_0}(0) \cong B_{r_0}(p, g).
\end{equation}
Moreover, we have
\begin{equation}
\frac{1}{2}|x-y| \leq \rho(x, y) \leq 2 |x- y| \text{ for any } x, y \in B_{r_0} (0). 
\end{equation}
where $\rho(x, y)$ denotes the geodesic distance induced by the Riemannian metric $g$ between $x$ and $y$.  
\end{lem}

We will also need a uniform control on high order derivatives of metric $g$ in a fixed sized geodesic ball. For this reason, we introduce the notion of harmonic radius. There are very rich literature on $C^{1,\alpha}$-harmonic radius and we refer interested readers to \cite{JK},\cite{GW},\cite{P},\cite{K} and \cite{An}. Among these results, Anderson(\cite{An}) showed that the $C^{1,\alpha}$-harmonic radius has a lower bound depending on $\max_M |\text{Ric}(g)|_g$ and the injectivity radius $i_0$. In this paper, we will mainly use the $C^{3,\alpha}$-harmonic radius defined as below.
\begin{defn}
On a Riemannian manifold $(M, g)$, for any given $Q >1$, we define the \emph{$C^{3,\alpha}$-harmonic radius at $p \in M$} denoted as $r_h(p, Q)$ to be the supremum of all $r >0$ such that in geodesic ball $B_r(p, g)$ there exists a harmonic coordinates system $\{u_j\}_{j=1}^n$ with the corresponding metric tensor $g_{ij} = g(\frac{\partial }{\partial u_i}, \frac{\partial }{\partial u_j})$ satisfying
\begin{equation}
Q^{-1} \delta_{ij} \leq g_{ij} \leq Q \delta_{ij},
\end{equation}
and
\begin{equation}
\sum_{l=1}^3 r^l \|D^l g_{ij}\|_{C^0} + r^{3+\alpha} [D^3 g_{ij}]_{C^\alpha} \leq Q -1, 
\end{equation}
on $B_r(p, g)$, where the derivatives and norms are taken with respect to the coordinates $\{u_j\}_{j=1}^n$. For $Q >1$ as above, define the \emph{$C^{3, \alpha}$-harmonic radius of $(M, g)$} as $r_h(Q) := \inf_{p \in M} r_h(p, Q)$. 
\end{defn}

As a direct consequence to Anderson's result, we have the following lemma. 

\begin{lem}\label{lem202}
Suppose $(M,g)$ is a closed $n$-dimensional Riemannian manifold. Given any $Q>1$, there exists a constant $r_1 > 0$ depending on $Q$, $\|\text{Ric}(g)\|_{C^{1,\alpha}(M)}$, the injectivity radius $i_0$ and dimension $n$ such that the $C^{3,\alpha}$-harmonic radius on $(M, g)$, $r_h(Q) \geq r_1$. 

\end{lem}

\begin{rmk}
Note that we could choose constant $Q>1$ arbitrarily close to $1$ by possibly shrinking $r_h(Q)$. 
\end{rmk}

\section{Biharmonic heat equation on closed manifolds}\label{sec3}

\subsection{The biharmonic heat kernel}\label{sec3.1}

Suppose $M$ is a closed smooth manifold of dimension $n$. Given a smooth Riemannian metric $g$ on $M$, we define the \emph{biharmonic heat kernel} with respect to metric $g$ to be a smooth function $b \in C^{\infty}\big(M \times M \times (0, \infty)\big)$ such that $b_g (x, y; t)$ as a function of $(x, t) \in M\times (0,\infty)$ satisfies
\begin{equation}\label{eqn31}
( \frac{\partial }{\partial t} + \Delta_g^2 ) b_g = 0,    
\end{equation}
and for any continuous function $u$ on $M$, 
\begin{equation}
\lim_{t \rightarrow 0^+} \int_M b_g(x, y; t) u(y) \mathrm d V_g(y) = u(x), 
\end{equation}
uniformly for $x \in M$.  

The operator $(\frac{\partial }{\partial t} + \Delta_g^2)$ is parabolic in the sense of Petrowski and then one can construct the biharmonic heat kernel via the parametrix method for pseudo differential operators. We refer to \cite[Chapter 9]{Fr} and \cite[Chapter I]{Ei} for general construction procedures on heat kernels of domains in $\bR^n$. We summarize the properties of biharmonic heat kernel as the following theorems.

\begin{thm}\label{thm301}
Given a smooth Riemannian metric $g$ on M, there exists a unique biharmonic heat kernel with respect to $g$ denoted as $b \in C^\infty\big(M \times M \times (0, \infty)\big)$. Moreover, for any integers $k, p, q \geq 0 $, we have for any $(x,y, t) \in M \times M \times(0,T)$, 
\begin{equation}\label{eqn32}
|\partial_t^k \nabla_x^{p} \nabla_y^{q} b_g(x, y;  t)|_g \leq C t^{- \frac{n + 4k + p+ q}{4}} \exp\{- \delta \big(t^{-\frac{1}{4}} \rho(x, y)\big)^{\frac{4}{3}} \},  
\end{equation}
where $\nabla_x$ and $\nabla_y$ are covariant derivatives with respect to $g$. Constants $C, \delta >0$ depend on $T$, $g$ and $p + q +4k$.  
\end{thm}

\begin{thm}\label{thm302}
Suppose $b_g(x, y; t)$ is the biharmonic heat kernel introduced in Theorem $\ref{thm301}$. Then at any given point $p \in M$, there exists a local coordinate chart $U \subset \bR^n \rightarrow M$ near $p \in M$ such that for any $x, y \in U$ and any multi indices $\alpha, \beta \geq 0$, we have for any $(x,y, t) \in (x,y, t) \in M \times M \times(0,T)$,
\begin{equation}\label{eqn301}
|D_x^\beta (D_x + D_y)^\alpha b_g(x, y; t) | \leq  C t^{- \frac{n + |\beta|}{4}} \exp\{- \delta \big(t^{-\frac{1}{4}} \rho(x,y) \big)^{\frac{4}{3}} \}, 
\end{equation}
where $D_x$ and $D_y$ are ordinary derivatives in $\bR^n$. Constants $C>0$ and $\delta > 0$ depend on $T, g, |\alpha|+ |\beta|$. 
\end{thm}

\begin{rmk}
Estimate $(\ref{eqn301})$ is valid independent of choice of local coordinates but the constants $C>0$ and $\delta >0$ in $(\ref{eqn301})$ depend an actual choice of the local coordinate. For example, if $x =\psi (x')$ where $\psi: V \rightarrow U$ is a coordinate change near $p$ smooth up to the boundary, then we have
\begin{equation}
\begin{split}
|D_{x'}^\beta (D_{x'} + D_{y'})^\alpha \big(b_g(\psi(x'), \psi(y'); t)| \leq  C' t^{- \frac{n+ |\beta|}{4}} \exp\{- \delta' \big(t^{-\frac{1}{4}} \rho(x', y') \big)^{\frac{4}{3}} \}
\end{split}
\end{equation}
for any $x' ,y' \in V$ and any multiple indexes $\alpha, \beta \geq 0$. The constants $C' >0$ and $\delta' > 0$ depends on $C, \delta >0$ as well as on the transition function $\psi$. 
\end{rmk}

\begin{rmk}
Theorem $\ref{thm301}$ and Theorem $\ref{thm302}$ are general facts valid for any differential operator with smooth coefficients and parabolic in the Petrowski sense. We refer interested readers to \cite[Theorem 8 and 9 on p263]{Fr}, \cite[Chapter I.3]{Ei} for their proofs on domains of $\bR^n$. In the appendix, we gave a detailed proof of Theorem $\ref{thm301}$ and Theorem $\ref{thm302}$ for biharmonic heat kernel on closed manifold. This constructional method could be applied to (Petrowski) parabolic operators of any order.
\end{rmk}

\begin{thm}\label{thm303}
Suppose $b_g(x, y; t)$ is the biharmonic heat kernel introduced in Theorem $\ref{thm301}$. We have that for any $(x, t) \in M \times (0,T)$
\begin{equation}
\int_M b_g(x, y; t) \mathrm d V_g(y) = 1
\end{equation}
As a consequence, for any multi index $\alpha > 0$, 
\begin{equation}
\int_M D_x^{\alpha} b_g(x, y ; t) \mathrm d V_g(y)= 0. 
\end{equation}

\end{thm}

\begin{proof}
Proof of Theorem $\ref{thm303}$. Denote $u(x; t) = \int_M b_g(x, y; t) \mathrm d V_g (y)$. Thus by definition of the biharmonic heat kernel, we have $u \in C^{\infty}\big(M \times (0,T)\big)$ satisfies 
\begin{equation}
(\frac{\partial }{\partial t} + \Delta_g^2) u = 0,
\end{equation}
and 
\begin{equation}
\lim_{t \rightarrow 0^+} \|u(t) -1\|_{C^0(M)} = 0. 
\end{equation}
Consider for any $t >0$, 
\begin{equation}
\frac{\partial }{\partial t } \int_{M } |u(t)-1|^2 \mathrm d V_g = - 2 \int_{M}\big( \Delta_g^2 u(t)\big)^2 \mathrm d V_g  \leq 0. 
\end{equation}
Moreover, we have $\lim_{t \rightarrow 0^+} \int_M |u(t)-1|^2 \mathrm d V_g = 0$. Therefore we have $\int_M |u(t)-1|^2 \mathrm d V_g = 0$ for any $0< t <T$. It implies that $u \equiv 1$ and this ends the proof of Theorem $\ref{thm303}$. 

\end{proof}

\subsection{Uniform integral bounds of the biharmonic heat kernel}\label{sec32}

Given a smooth Riemannian metric $g$ on closed manifold $M$, we denote $b_g \in C^\infty\big(M \times M \times (0,\infty) \big) $ to be the unique biharmonic heat kernel with respect to metric $g$ introduced in last subsection. In this subsection, we study the uniform space integral bounds on $b_g$ for all metric $g$ for sufficiently small time.

We use $b_0 \in C^\infty\big(\bR^n \times \bR^n \times (0, \infty)\big)$ to denote the (unique) biharmonic heat kernel of the Euclidean metric on $\bR^n $. Then we have the following uniform bounds of $b_g$ for arbitrary metric $g$ for sufficiently small time.

\begin{thm}\label{thm2-1}
For any integer $k \geq 0$, any Riemannian metric $g$ and any $x\in M$, we have 
\begin{equation}\label{eqn201}
 \lim_{t \rightarrow 0^+}  \int_M t^{\frac{k}{4}} |\nabla_{g}^k b_g(x, y; t)|_{g(x)} \mathrm d V_g(y)  = \int_{\bR^n}  |D^k b_0(0, y; 1)| \mathrm d y, 
\end{equation}
where the covariant derivative $\nabla_g$ above is taken on variable $x$. 
\end{thm}

\begin{proof}
By Lemma $\ref{lem201}$, there exists a constant $r_0>0$ depending on the injectivity radius and the Riemannian curvature bounds such that under the local normal coordinate $B_{r_0}(0) \cong B_{r_0}(p, g)$ at any given $p \in M$, we have for any $x \in B_{r_0}(0)$, 
\begin{equation}
\frac{1}{2} \delta_{ij} \leq g_{ij}(x) \leq 2 \delta_{ij}, 
\end{equation}
and for any $x, y \in B_{r_0}(0)$,
\begin{equation}
\frac{1}{2}|x- y| \leq d_g(x, y) \leq 2|x-y|.  
\end{equation}

Next consider 
\begin{equation}
\begin{split}
&  \int_M t^{\frac{k}{4}}| \nabla_g^k b_{ g}(x, y; t)|_{g(x)}  \mathrm d V_g(y) \\
 = & \int_{B_{r_0}(x, g)} t^{\frac{k}{4}}| \nabla_g^k b_{ g}(x, y; t)|_{g(x)}  \mathrm d V_g(y) + \int_{M - B_{r_0}(x, g)} t^{\frac{k}{4}}| \nabla_g^k b_{ g}(x, y; t)|_{g(x)}  \mathrm d V_g(y).
 \end{split}
\end{equation}
The later integral goes to zero as $t \rightarrow 0^+$, since we have by Theorem $\ref{thm301}$ that 
\begin{equation}
\begin{split}
\int_{M - B_{r_0}(x, g)} t^{\frac{k}{4}}| \nabla_g^k b_{ g}(x, y; t)|_{g(x)}  \mathrm d V_g(y) & \leq C t^{-\frac{n}{4}} \exp\{- \delta (t^{-\frac{1}{4}} r_0)^{\frac{4}{3}}\} \text{Vol}(M, g), \\
 & \rightarrow 0, \text{ as } t \rightarrow 0^+, 
 \end{split}
\end{equation}
where $C, \delta > 0$ are constants depending on metric $g$ and $k$.

It remains to compute
\begin{equation}\label{eqn00026}
\lim_{t \rightarrow 0^+}\int_{B_{r_0}(x, g)} t^{\frac{k}{4}}| \nabla_g^k b_{ g}(x, y; t)|_{g(x)} \mathrm d V_g(y). 
\end{equation}
Restricted to geodesic ball $B_{r_0}(x, g)$, under the normal coordinate $ B_{r_0}(0) \cong B_{r_0}(x, g)$, we have $g_{ij} = g(\partial_i, \partial_j) \in C^\infty(B_{r_0}(0))$ and $b_g \in C^\infty \big(B_{r_0}(0) \times B_{r_0}(0) \times (0,\infty) \big)$

Given $\epsilon > 0$, on Euclidean ball $B_{\epsilon^{-\frac{1}{4}} r_0}(0)$, we define metric $(g_\epsilon)_{ij}(u) = g_{ij} (\epsilon^{\frac{1}{4}} u)$, where $g_{ij}(\cdot)$ is the metric $g$ written under the local normal coordinate. Thus $(g_\epsilon)_{ij} \rightarrow \delta_{ij}$ as $\epsilon \rightarrow 0^+$ smoothly and uniformly on any compact subset of $\bR^n$. We also define $b_\epsilon \in C^\infty\big(B_{\epsilon^{-\frac{1}{4}}r_0} (0) \times B_{\epsilon^{-\frac{1}{4}} r_0}(0)  \times (0, \infty) \big)$ as 
\begin{equation}
b_\epsilon(u, v; t) = \epsilon^{\frac{n}{4}} b_g \big(\epsilon^{\frac{1}{4}}u, \epsilon^{\frac{1}{4}} v; \epsilon t \big)
\end{equation}
Thus we have for any $u, v \in B_{\epsilon^{-\frac{1}{4}} r_0}(0)$ and $t> 0$
\begin{equation}\label{eqn2-11}
( \frac{\partial }{\partial t} + \Delta_{g_\epsilon}^2 ) b_\epsilon(u, v;t) = 0. 
\end{equation}

Recall we denote $b_0 \in C^\infty \big(\bR^n \times \bR^n \times (0,\infty) \big)$ to be the biharmonic heat kernel of the Euclidean metric on $\bR^n$. Then we have the following lemma.

\begin{lem}\label{lem2-2}
As $\epsilon \rightarrow 0^+$, 
\begin{equation}
  b_\epsilon  \rightarrow  b_0 , 
\end{equation}
smoothly and uniformly on any compact subset of $\bR^n \times \bR^n \times (0,\infty)$. 
\end{lem}

\begin{proof}
Proof of Lemma $\ref{lem2-2}$. By Theorem $\ref{thm301}$, we have for any $x, y\in M$ and $0<t <T$, 
\begin{equation}
|(\frac{\partial }{\partial t})^l \nabla_{g,x}^p \nabla_{g,y}^q b_g(x,y;t)|_g \leq C t^{-\frac{n+p+q+4l}{4}} \exp\{- \delta \big(t^{-\frac{1}{4}} d_{g}(x,y)\big)^{\frac{4}{3}} \}, 
\end{equation}
where constants $C, \delta > 0$ depend on $g, l, p, q$ and $T$. Therefore we have for any $u,v \in B_{\epsilon^{-\frac{1}{4}}r_0}(0)$ and $0< t < T$
\begin{equation}\label{eqn2-14}
\begin{split}
&|(\frac{\partial }{\partial t})^l \nabla_{g_\epsilon, u }^p \nabla_{g_\epsilon ,v}^q b_\epsilon(u,v;t)|_{g_\epsilon} \\
=& \big( |(\frac{\partial }{\partial t})^l \nabla_{g,x}^p \nabla_{g,y}^q b_g|_g \big) (\epsilon^{\frac{1}{4}}u,  \epsilon^{\frac{1}{4}}v;  \epsilon t) \times \epsilon^{\frac{n+p+q+4l}{4}}\\
\leq & C t^{-\frac{n+p+q+4l}{4}} \exp\{- \delta \big( (\epsilon t)^{-\frac{1}{4}} d_g(\epsilon^{\frac{1}{4}} u,  \epsilon^{\frac{1}{4}} v )\big)^{\frac{4}{3}}\} \\
\leq & C t^{-\frac{n+p+q+4l}{4}} \exp\{- 2^{- \frac{4}{3}} \delta \big( (\epsilon t)^{-\frac{1}{4}} |\epsilon^{\frac{1}{4}} u-   \epsilon^{\frac{1}{4}} v |\big)^{\frac{4}{3}}\} \\
\leq & C t^{-\frac{n+p+q+4l}{4}} \exp\{- 2^{- \frac{4}{3}} \delta \big(  t^{-\frac{1}{4}} |u-    v |\big)^{\frac{4}{3}}\}.
\end{split}
\end{equation}
Together with the fact $(g_\epsilon)_{ij} \rightarrow \delta_{ij}$ smoothly and uniformly on any compact set of $\bR^n$, we can conclude that $b_\epsilon$'s admit a convergent subsequence smoothly and uniformly converging to some limit function $b \in C^\infty\big(\bR^n \times \bR^n \times (0, \infty) \big)$ as $\epsilon \rightarrow 0^+$ on any compact subset of $\bR^n \times \bR^n \times (0, \infty)$. 

Next we will show all such sequential limit $b$ must be $b_0$. Since the biharmonic heat kernel on $\bR^n$ is unique, it suffices to check that $b \in C^\infty\big(\bR^n \times \bR^n \times (0,\infty)\big)$ meets the definition of biharmonic heat kernel of $\bR^n$. First, by letting $\epsilon \rightarrow 0^+$ in $(\ref{eqn2-11})$, we have
\begin{equation}\label{eqn2-15}
(\frac{\partial}{\partial t} + \Delta^2) b = 0 \text{ on } \bR^n \times \bR^n \times (0,\infty). 
\end{equation}
Second, we need to show that for any bounded continuous function $\varphi$ on $\bR^n$, 
\begin{equation}\label{eqn2-16}
\lim_{t \rightarrow 0^+} \int_{\bR^n} b_g(u, v;t ) \varphi(v) \mathrm d v = \varphi(u).
\end{equation}
Note by taking $\epsilon \rightarrow 0^+$ in $(\ref{eqn2-14})$ we have
\begin{equation}
 | b_g(u,v;t ) | \leq  C t^{-\frac{n}{4}} \exp\{- 2^{- \frac{4}{3}} \delta \big(  t^{-\frac{1}{4}} |u-    v |\big)^{\frac{4}{3}}\}
\end{equation} 
and thus 
\begin{equation}
\lim_{t \rightarrow 0^+} \int_{\bR^n} b_g(u, v;t ) \varphi(v) \mathrm d v = \varphi(u)  \lim_{t \rightarrow 0^+} \int_{\bR^n} b_g(u, v;t )  \mathrm d v.
\end{equation}
Thus it suffices to show 
\begin{equation}
\lim_{t \rightarrow 0^+} \int_{\bR^n} b_g(u, v;t )  \mathrm d v =1. 
\end{equation}
By dominated convergence theorem, we have
\begin{equation}
\begin{split}
\int_{\bR^n} b_g(u, v; t) \mathrm dv =  & \lim_{\epsilon \rightarrow 0^+} \int_{B_{\epsilon^{-\frac{1}{4}} r_0} (0)}b_\epsilon(u, v; t)  \mathrm d V_{g_{\epsilon}} (v) = \lim_{\epsilon \rightarrow 0^+} \int_{B_{r_0} (x, g) } b_g(\epsilon^{\frac{1}{4}} u, y; \epsilon t) \mathrm d V_g(y)\\
= & \lim_{\epsilon \rightarrow 0^+} \big( \int_M  b_g(\epsilon^{\frac{1}{4}} u, y; \epsilon t) \mathrm d V_g(y) - \int_{M -B_{r_0} (x, g)} b_g(\epsilon^{\frac{1}{4}} u, y; \epsilon t) \mathrm d V_g(y) \big) \\
\end{split}
\end{equation}
By Theorem $\ref{thm303}$, we have the first integral is always $1$ and as we discussed before, the later integral goes to zero. 
To summarize, we showed $b \in C^\infty \big(\bR^n \times \bR^n \times (0,\infty) \big)$ satisfies $(\ref{eqn2-15})$ and $(\ref{eqn2-16})$ and therefore $b = b_0$. This ends the proof of the lemma.  

\end{proof}

Continue $(\ref{eqn00026})$ in the proof of Theorem $\ref{thm2-1}$. We have
\begin{equation}
\begin{split}
\int_{B_{r_0}(x, g)} t^{\frac{k}{4}}| \nabla_g^k b_{ g}(x, y; t)|_{g(x)} \mathrm d V_g(y) & = \int_{B_{t^{-\frac{1}{4}}r_0}(0)} |\nabla^k_{g_{t}} b_{t} (0, v; 1)|_{g_{t}(0)} \mathrm d V_{g_{t}} (v)\\
& \rightarrow \int_{\bR^n} | D^k  b_0(0,v;1)| \mathrm d v, \text{ as } t \rightarrow 0^+,
\end{split}
\end{equation}
by Lemma $\ref{lem2-2}$ and dominated convergence theorem. Thus we have
\begin{equation}
 \lim_{t \rightarrow 0^+} \int_{M} t^{\frac{k}{4}}| \nabla_g^k b_{ g}(x, y; t)|_{g(x)} \mathrm d V_g(y) =  \int_{\bR^n} | D^k  b_0(0,v;1)| \mathrm d v,
\end{equation}
where the right hand side is a constant depending on $k$ and dimension $n$ only.

\end{proof}

\subsection{Initial value problem}\label{sec3.2}

Given the biharmonic heat kernel $b_g \in C^\infty\big(M \times M \times (0,\infty)\big)$ with respect to background metric $g$, and given a continuous function $u_0 \in C^0(M)$, we define for $(x, t) \in M \times (0,\infty)$
\begin{equation}\label{eqn4001}
S_g u_0 (x, t) = \int_M b_g(x, y; t) u_0(y) \mathrm d V_g(y). 
\end{equation} 
Since the biharmonic heat kernel is smooth when $t>0$, we get $S_g u_0 \in C^{\infty}\big(M \times (0,\infty)\big)$ and 
\begin{equation}\label{eqn4002}
(\frac{\partial }{\partial t} + \Delta_g^2) S_g u_0 (x, t) = 0. 
\end{equation}
Moreover, by the definition of the biharmonic heat kernel, we have that
\begin{equation}
\lim_{t \rightarrow 0^+} \|S_g u_0(t) - u_0\|_{C^0(M)} = 0. 
\end{equation}

We are mainly interested in the cases when the initial value $u_0 \in C^{1,1}(M)$ and $u_0 \in C^2(M)$. In both cases, $S_g u_0$ defined as in $(\ref{eqn4001})$ is a smooth function on $M \times (0,\infty)$ and satisfies equation $(\ref{eqn4002})$ smoothly. In this section, we aim to study the following two problems: i) How does the initial value $u_0 \in C^{1,1}(M)$ control the higher order derivatives of $S_g u_0(t)$ for $t>0$ and ii) Given $u_0 \in C^2(M)$, how does $S_g u_0(t)$ converge to $u_0$ as $t \rightarrow 0^+$. 

To answer these questions, we introduce the following theorem.

\begin{thm}\label{thm001}
Suppose $u_0 \in C^{1,1}(M)$. Then 
\begin{equation}
\lim_{t \rightarrow 0^+ } \|S_g u_0(t) - u_0\|_{C^1(M)} = 0.
\end{equation}
For any nonnegative integers $l$ and $k$, we have 
\begin{equation}
\sup_{t \in (0,T)}t^{l + \frac{k}{4}} \big\| | (\frac{\partial }{\partial t})^l \nabla_g^{k+2} S_g u_0(t) |_g \big\|_{C^0(M)} \leq C_{l+ \frac{k}{4}} \|u_0\|_{C^{1,1}(M)}, 
\end{equation}
for some constant $C_{l+ \frac{k}{4}} >0$ depending on $l + \frac{k}{4}$, $g$ and $T$. 
If we further assume that $u \in C^2(M)$, then
\begin{equation}
\lim_{t \rightarrow 0^+} \| S_g u_0(t) - u_0\|_{C^{2}(M)} = 0,
\end{equation}
and for any nonnegative integers $l$ and $k$ with $l+k>0$, 
\begin{equation}
\lim_{t \rightarrow 0^+} t^{l+ \frac{k}{4} } \big\| |(\frac{\partial }{\partial t})^l \nabla_g^{k+2} S_g u_0(t)|_g \big\|_{C^0(M)} = 0. 
\end{equation}

\end{thm}

\begin{proof}
Proof of Theorem $\ref{thm001}$. Suppose $u_0 \in C^{1,1}(M)$. By Theorem $\ref{thm301}$ we have $\|b_g(x, \cdot; t)\|_{L^1(M)} \leq C_0$ where $C_0 > 0$ depends on $g$ and $T$ only. Thus
\begin{equation}
\sup_{t \in (0,T)}\|S_g u_0(t)\|_{C^0(M)} \leq C_0 \|u_0\|_{C^0(M)}. 
\end{equation}
Moreover, $\lim_{t \rightarrow 0^+ } \|S_g u_0(t) - u_0\|_{C^0(M)} = 0$.

Next we examine $\nabla_g S_g u_0(t)$. By Theorem $\ref{thm302}$, $M$ admits a finite open cover $M = \bigcup V_i$ such that $\bar V_i \subset U_i$ and within each $U_i$ we have estimates $(\ref{eqn301})$. Denote $d_1 = \min_i \{\text{dist}(V_i, \partial U_i)\} >0$. By Lemma $\ref{lem201}$, without loss of generality, we can assume in local coordinate of each $U_i$, $\frac{1}{2}|x-y| \leq \rho(x, y) \leq 2 |x- y|$. Suppose $x \in V_i$ for some $i$. Then within local coordinates of $U_i$, we have 
\begin{equation}
\begin{split}
&\frac{\partial }{\partial x_l} S_g u_0(x, t) = \int_M  \frac{\partial }{\partial x_l}b_g(x, y; t) u_0(y) \mathrm d V_g(y) \\
=& \int_{U_i} (\frac{\partial}{\partial x_l} + \frac{\partial }{\partial y_l}) b_g(x, y; t) u_0(y) \mathrm d V_g(y) - \int_{U_i} \frac{\partial }{\partial y_l } b_g(x, y; t) u_0(y) \mathrm d V_g(y) \\
&+ \int_{M -U_i} \frac{\partial }{\partial x_l}b_g(x, y; t) u_0(y) \mathrm d V_g(y)\\
=&  \int_{U_i} (\frac{\partial}{\partial x_l} + \frac{\partial }{\partial y_l}) b_g(x, y; t) u_0(y) \mathrm d V_g(y) + \int_{U_i}  b_g(x, y; t) \frac{\partial }{\partial y_l }\big(u_0(y) \mathrm d V_g(y)\big) \\
&+ \int_{\partial U_i} b_g(x, y; t) u_0(y) \mathrm d \sigma_l(y)+ \int_{M -U_i} \frac{\partial }{\partial x_l}b_g(x, y; t) u_0(y) \mathrm d V_g(y), 
\end{split}
\end{equation}
where $\mathrm d \sigma_l(y)$ is a smooth $(n-1)$ form on submanifold $\partial U_i$ determined by $g,l$ only. Thus by Theorem $\ref{thm301}$ and $\ref{thm302}$ we have 
\begin{equation}
|\frac{\partial }{\partial x_l} S_g u_0(x, t)| \leq C ( \|u_0\|_{C^0(M)} + \|\nabla_g u_0\|_{C^0(M)}),
\end{equation}
where constant $C >0$ depends on $g, T$. Therefore $\sup_{t \in (0,T)}\|\nabla_g S_g u_0(t)\|_{C^0(M)} \leq C_1 \|u_0\|_{C^1(M)}$ for some constant $C_1 >0$ depending on $g,T$.

When $t \rightarrow 0^+$, clearly the later two integrals in $\frac{\partial }{\partial x_l} S_g u_0(t)$ go to zero, since as $t \rightarrow 0^+$ we have
\begin{equation}
|\int_{\partial U_i} b_g(x, y; t) u_0(y) \mathrm d \sigma_l(y)| \leq C\|u_0\|_{C^0(M)} \text{Vol}(\partial U_i) t^{- \frac{n}{4}} \exp\{- \delta (t^{-\frac{1}{4}} d_1)^{\frac{4}{3}}\} \rightarrow 0,
\end{equation}
and
\begin{equation}
|\int_{M -U_i} \frac{\partial }{\partial x_l}b_g(x, y; t) u_0(y) \mathrm d V_g(y)| \leq  C\|u_0\|_{C^0(M)} \text{Vol}(M)t^{- \frac{n+1}{4}} \exp\{- \delta (t^{-\frac{1}{4}} d_1)^{\frac{4}{3}}\} \rightarrow 0. 
\end{equation}
While for the first two integrals in $\frac{\partial }{\partial x_l} S_g u_0(t)$, we have for any $0<r \ll 1$
\begin{equation}
\begin{split}
&   |\int_{U_i} (\frac{\partial}{\partial x_l} + \frac{\partial }{\partial y_l}) b_g(x, y; t) u_0(y)   \mathrm d V_g(y) - u_0(x) \int_{U_i} (\frac{\partial}{\partial x_l} + \frac{\partial }{\partial y_l}) b_g(x, y; t) \mathrm d V_g(y)  |\\
= & |\int_{U_i} (\frac{\partial}{\partial x_l} + \frac{\partial }{\partial y_l}) b_g(x, y; t) \big( u_0(y) - u_0(x) \big)  \mathrm d V_g(y)|\\
\leq &  C \|u_0\|_{C^0(M)} \int_{|y-x| >  r, y \in U_i} t^{-\frac{n}{4}}e^{-\delta(t^{-\frac{1}{4}}|x-y|)^{\frac{4}{3}}} \mathrm d y \\
& + C \sup_{|x-y| \leq r} |u_0(x) - u_0(y)| \int_{|x-y| \leq r}  t^{-\frac{n}{4}}e^{-\delta(t^{-\frac{1}{4}}|x-y|)^{\frac{4}{3}}} \mathrm d y \\
\leq & C \|u_0\|_{C^0(M)} \int_{|\xi | > t^{-\frac{1}{4}} r} e^{-\delta |\xi|^{\frac{4}{3}}} \mathrm d \xi + C \sup_{\rho(x,y) \leq 2 r} |u_0(x) - u_0(y)| \int_{\bR^n} e^{-\delta |\xi|^{\frac{4}{3}}} \mathrm d \xi, 
\end{split}
\end{equation}
and 
\begin{equation}
\begin{split}
& | \int_{U_i}  b_g(x, y; t) \frac{\partial }{\partial y_l }\big(u_0(y) \mathrm d V_g(y)\big) - \frac{\partial }{\partial x_l }u_0(x) \int_{U_i}  b_g(x, y; t)  \mathrm d V_g(y) - u_0(x) \int_{U_i} b_g(x, y; t) \frac{\partial }{\partial y_l} \mathrm d V_g(y) \}|\\
= & |\int_{U_i}  b_g(x, y; t)\{ \big(\frac{\partial }{\partial y_l } u_0(y) - \frac{\partial }{\partial x_l }u_0(x) \big) \mathrm d V_g(y) + \big( u_0(y) - u_0(x) \big) \frac{\partial }{\partial y_l} \mathrm d V_g(y) \}|\\
\leq & C (\|u_0\|_{C^0(M)}+ \|\nabla u_0\|_{C^0(M)} ) \int_{|\xi | > t^{-\frac{1}{4}} r} e^{-\delta |\xi|^{\frac{4}{3}}} \mathrm d \xi \\
&+  C \sup_{\rho(x,y) \leq 2 r} (|u_0(x) - u_0(y)| + |\frac{\partial }{\partial y_l } u_0(y) - \frac{\partial }{\partial x_l }u_0(x) |) \int_{\bR^n} e^{-\delta |\xi|^{\frac{4}{3}}} \mathrm d \xi. 
\end{split}
\end{equation}
Choose $r = t^{\frac{1}{8}}$, then we have the sum of first two integral equals to 
\begin{equation}
u_0(x) I_1  \\
+  \frac{\partial }{\partial x_l }u_0(x) I_2 + o(1),
\end{equation}
as $t \rightarrow 0^+$, where 
\begin{equation}
\begin{split}
I_1 = & \int_{U_i}   (\frac{\partial}{\partial x_l} + \frac{\partial }{\partial y_l}) b_g(x, y; t) \mathrm d V_g(y)+ \int_{U_i} b_g(x, y; t) \frac{\partial }{\partial y_l} \mathrm d V_g(y)\\
  = & \int_{U_i}  \frac{\partial}{\partial x_l} b_g(x, y; t) \mathrm d V_g(y)+ \int_{U_i} \frac{\partial }{\partial y_l} \big( b_g(x,y;t)\mathrm d V_g(y) \big),\\
I_2 =  & \int_{U_i}  b_g(x, y; t)  \mathrm d V_g(y).
\end{split}
\end{equation}
We will show next $I_1 \rightarrow 0$ and $I_2 \rightarrow 1$ uniformly as $t \rightarrow 0^+$. By Theorem $\ref{thm303}$, we have
\begin{equation}
\begin{split}
|I_1| & = |\int_{M}  \frac{\partial}{\partial x_l} b_g(x, y; t) \mathrm d V_g(y)- \int_{M- U_i}  \frac{\partial}{\partial x_l} b_g(x, y; t) \mathrm d V_g(y)+ \int_{\partial U_i}  b_g(x,y;t)\mathrm d \sigma_l (y) |\\
& =  |- \int_{M- U_i}  \frac{\partial}{\partial x_l} b_g(x, y; t) \mathrm d V_g(y)+ \int_{\partial U_i}  b_g(x,y;t)\mathrm d \sigma_l (y) |\\
& \leq C \text{Vol}(M) t^{-\frac{n+1}{4}} \exp\{- \delta (t^{\frac{1}{4}}d_1)^{\frac{4}{3}}\} + C \text{Vol}(\partial U_i) t^{-\frac{n}{4}} \exp\{- \delta (t^{\frac{1}{4}}d_1)^{\frac{4}{3}}\}\\
& \rightarrow 0,
\end{split}
\end{equation}
and
\begin{equation}
\begin{split}
|I_2 -1 | & =  |  \int_{U_i} b_g(x, y; t) \mathrm d V_g(y) - \int_{M} b_g(x,y;t) \mathrm d V_g(y)  |\\
 & = |  - \int_{M- U_i} b_g(x, y; t) \mathrm d V_g(y)| \leq C \text{Vol}(M) t^{-\frac{n}{4}} \exp\{- \delta (t^{\frac{1}{4}}d_1)^{\frac{4}{3}}\} \rightarrow 0, 
\end{split}
\end{equation}
as $t \rightarrow 0^+$. Therefore, we have $\frac{\partial }{\partial x_l} S_g u_0(t) \rightarrow \frac{\partial }{\partial x_l} u_0$ uniformly as $t \rightarrow 0^+$. Thus
\begin{equation}
\lim_{t \rightarrow 0^+} \|\nabla_g S_g u_0(t) - \nabla_g u_0\|_{C^0(M)} = 0. 
\end{equation}

As for the second order derivative of $S_g u_0(t)$, we compute for $x\in V_i$
\begin{equation}\label{eqn601}
\begin{split}
&\frac{\partial^2 }{\partial x_k \partial x_l} S_g u_0(x, t) = \int_M  \frac{\partial^2 }{\partial x_l \partial x_k}b_g(x, y; t) u_0(y) \mathrm d V_g(y) \\
=&  \int_{U_i}  (\frac{\partial}{\partial x_l} + \frac{\partial }{\partial y_l} - \frac{\partial }{\partial y_l})(\frac{\partial}{\partial x_k} + \frac{\partial }{\partial y_k} -\frac{\partial }{\partial y_k} ) b_g(x, y; t) u_0(y) \mathrm d V_g(y) \\
& +   \int_{M -U_i} \frac{\partial^2 }{\partial x_l \partial x_k}b_g(x, y; t) u_0(y) \mathrm d V_g(y) \\
=& \int_{U_i} (\frac{\partial}{\partial x_l} + \frac{\partial }{\partial y_l})(\frac{\partial}{\partial x_k} + \frac{\partial }{\partial y_k}) b_g(x, y; t) u_0(y) \mathrm d V_g(y) \\
&- \int_{U_i} \frac{\partial }{\partial y_l }(\frac{\partial}{\partial x_k} + \frac{\partial }{\partial y_k}) b_g(x, y; t) u_0(y) \mathrm d V_g(y) \\
&-  \int_{U_i} \frac{\partial }{\partial y_k }(\frac{\partial}{\partial x_l} + \frac{\partial }{\partial y_l}) b_g(x, y; t) u_0(y) \mathrm d V_g(y) \\
& +\int_{U_i} \frac{\partial^2 }{\partial y_l \partial y_k } b_g(x, y; t) u_0(y) \mathrm d V_g(y) + \int_{M -U_i} \frac{\partial^2 }{\partial x_l \partial x_k}b_g(x, y; t) u_0(y) \mathrm d V_g(y)\\
=& \int_{U_i} (\frac{\partial}{\partial x_l} + \frac{\partial }{\partial y_l})(\frac{\partial}{\partial x_k} + \frac{\partial }{\partial y_k}) b_g(x, y; t) u_0(y) \mathrm d V_g(y) \\
&+\int_{U_i} (\frac{\partial}{\partial x_k} + \frac{\partial }{\partial y_k}) b_g(x, y; t)   \frac{\partial }{\partial y_l } \big( u_0(y) \mathrm d V_g(y) \big) + \int_{\partial U_i} (\frac{\partial}{\partial x_k} + \frac{\partial }{\partial y_k}) b_g(x, y; t)  u_0(y) \mathrm d \sigma_l(y)\\
&+\int_{U_i} (\frac{\partial}{\partial x_l} + \frac{\partial }{\partial y_l}) b_g(x, y; t)   \frac{\partial }{\partial y_k } \big( u_0(y) \mathrm d V_g(y) \big) + \int_{\partial U_i} (\frac{\partial}{\partial x_l} + \frac{\partial }{\partial y_l}) b_g(x, y; t)  u_0(y) \mathrm d \sigma_k(y)\\
& +\int_{U_i}  b_g(x, y; t) \frac{\partial^2 }{\partial y_l \partial y_k } \big( u_0(y) \mathrm d V_g(y) \big) + \int_{\partial U_i} \frac{\partial }{\partial y_k} b_g(x, y; t)  u_0(y) \mathrm d \sigma_l(y) \\
&+ \int_{\partial U_i}  b_g(x, y; t)  \frac{\partial }{\partial y_l} u_0(y) \mathrm d \sigma_k(y) + \int_{\partial U_i}  b_g(x, y; t)  u_0(y) \mathrm d \sigma_{kl}(y) \\
&+ \int_{M -U_i} \frac{\partial^2 }{\partial x_l \partial x_k}b_g(x, y; t) u_0(y) \mathrm d V_g(y).
\end{split}
\end{equation}
As we did before, we have that $\sup_{t \in (0,T)} \|\nabla_g^2 S_g u_0(t)\|_{C^0 (M)} \leq C_2  \|u_0\|_{C^{1,1}(M)}$ for some constant $C_2 >0$ depending on $g,T$. 

When $t \rightarrow 0^+$, we have that all the boundary integrals on $\partial U_i$ and the integral on $(M - U_i)$ go to zero by a similar argument as described above for $\nabla S_g u_0(t)$. If we further assume that $u_0 \in C^2(M)$, namely $\nabla_g^2 u_0 \in C^0(M)$, then similar to what we did above for $\nabla_g S_g u_0(t)$ when $t \rightarrow 0^+$, $\frac{\partial^2 }{\partial x_k \partial x_l} S_g u_0(t)$ equals to
\begin{equation}
u_0(x) I_4 + \frac{\partial }{\partial x_l}u_0(x) I_5 + \frac{\partial }{\partial x_k} u_0(x) I_6 + \frac{\partial^2 }{\partial x_k \partial x_l}u_0(x) I_2 + o(1),  
\end{equation}
where 
\begin{equation}
\begin{split}
I_4 =& \int_{U_i} (\frac{\partial}{\partial x_l} + \frac{\partial }{\partial y_l})(\frac{\partial}{\partial x_k} + \frac{\partial }{\partial y_k}) b_g(x, y; t) \mathrm d V_g(y) + \int_{U_i} (\frac{\partial}{\partial x_k} + \frac{\partial }{\partial y_k}) b_g(x, y; t)   \frac{\partial }{\partial y_l }\mathrm d V_g(y)\\
& + \int_{U_i} (\frac{\partial}{\partial x_l} + \frac{\partial }{\partial y_l}) b_g(x, y; t)   \frac{\partial }{\partial y_k }  \mathrm d V_g(y) +  \int_{U_i}  b_g(x, y; t) \frac{\partial^2 }{\partial y_l \partial y_k }  \mathrm d V_g(y) \\
= & \int_{U_i} \frac{\partial^2}{\partial x_k \partial x_l} b (x,y;t) \mathrm d V_g(y) +\int_{U_i} \frac{\partial }{\partial y_l} \big( \frac{\partial}{\partial x_k} b (x,y;t) \mathrm d V_g(y) \big)\\
&+  \int_{U_i} \frac{\partial }{\partial y_k}\big(\frac{\partial}{\partial x_l  } b (x,y;t) \mathrm d V_g(y) \big)+ \int_{U_i} \frac{\partial^2}{\partial y_k \partial y_l} \big( b (x,y;t) \mathrm d V_g(y) \big), \\
I_5 =& \int_{U_i} (\frac{\partial}{\partial x_k} + \frac{\partial }{\partial y_k}) b_g(x, y; t)  \mathrm d V_g(y) + \int_{U_i}  b_g(x, y; t) \frac{\partial }{\partial y_k }  \mathrm d V_g(y) \\ 
 = & \int_{U_i} \frac{\partial }{\partial x_k} b_g(x,y;t) \mathrm d V_g(y) + \int_{U_i}  \frac{\partial }{\partial y_k} \big(  b_g(x,y;t) \mathrm d V_g(y) \big),\\
I_6 = & \int_{U_i} \frac{\partial }{\partial x_l} b_g(x,y;t) \mathrm d V_g(y) + \int_{U_i}  \frac{\partial }{\partial y_l} \big(  b_g(x,y;t) \mathrm d V_g(y) \big), \\ 
\end{split}
\end{equation}
Following Theorem $\ref{thm303}$ and ideas of the proof developed for $\nabla S_g u_0(t)$, we have $I_4, I_5, I_6 \rightarrow 0$ and $I_2 \rightarrow 1$ uniformly as $t \rightarrow 0^+$. Therefore we have $$\lim_{t \rightarrow 0^+}\|\nabla_g^2 S_g u_0(t) - \nabla_g^2 u_0\|_{C^0(M)} = 0,$$
 for any $u_0 \in C^2(M)$. 

As for the higher order derivatives of $S_g u_0(t)$, for any multi index $\alpha \geq 0$, we take $D_x^\alpha$ on $(\ref{eqn601})$ and get
\begin{equation}\label{eqn000062}
\begin{split}
&D_x^\alpha \frac{\partial^2 }{\partial x_k \partial x_l} S_g u_0(t)  \\
=& \int_{U_i} D_x^\alpha (\frac{\partial}{\partial x_l} + \frac{\partial }{\partial y_l})(\frac{\partial}{\partial x_k} + \frac{\partial }{\partial y_k}) b_g(x, y; t) u_0(y) \mathrm d V_g(y) \\
&+\int_{U_i} D_x^\alpha (\frac{\partial}{\partial x_k} + \frac{\partial }{\partial y_k}) b_g(x, y; t)   \frac{\partial }{\partial y_l } \big( u_0(y) \mathrm d V_g(y) \big) + \int_{\partial U_i} D_x^\alpha(\frac{\partial}{\partial x_k} + \frac{\partial }{\partial y_k}) b_g(x, y; t)  u_0(y) \mathrm d \sigma_l(y)\\
&+\int_{U_i} D_x^\alpha(\frac{\partial}{\partial x_l} + \frac{\partial }{\partial y_l}) b_g(x, y; t)   \frac{\partial }{\partial y_k } \big( u_0(y) \mathrm d V_g(y) \big) + \int_{\partial U_i} D_x^\alpha(\frac{\partial}{\partial x_l} + \frac{\partial }{\partial y_l}) b_g(x, y; t)  u_0(y) \mathrm d \sigma_k(y)\\
& +\int_{U_i}  D_x^\alpha b_g(x, y; t) \frac{\partial^2 }{\partial y_l \partial y_k } \big( u_0(y) \mathrm d V_g(y) \big) + \int_{\partial U_i} D_x^\alpha \frac{\partial }{\partial y_k} b_g(x, y; t)  u_0(y) \mathrm d \sigma_l(y) \\
&+ \int_{\partial U_i}  D_x^\alpha b_g(x, y; t)  \frac{\partial }{\partial y_l} u_0(y) \mathrm d \sigma_k(y) + \int_{\partial U_i}  D_x^\alpha b_g(x, y; t)  u_0(y) \mathrm d \sigma_{kl}(y) \\
&+ \int_{M -U_i} D_x^\alpha \frac{\partial^2 }{\partial x_l \partial x_k}b_g(x, y; t) u_0(y) \mathrm d V_g(y).
\end{split}
\end{equation}
Thus we have that $\sup_{t \in (0,T)} t^{\frac{k-2}{4}}\|\nabla_g^{k} S_g u_0(t)\|_{C^0(M)} \leq C_k \|u_0\|_{C^{1,1}(M)}$ for $k>2$. 

When $t \rightarrow 0^+$, again we have all the boundary integrals and integral on $(M -U_i)$ going to zero. While for the integrals on $U_i$ in the expression of $D_x^\alpha \frac{\partial^2 }{\partial x_k \partial x_l} S_g u_0(t)$, we notice that $$|t^{\frac{|\alpha|}{4}} D_x^\alpha M(x,y;t)| \leq C t^{-\frac{n}{4}}\exp\{- \delta (t^{\frac{1}{4}} \rho(x,y))^{\frac{4}{3}} \}$$ where $M(x,y;t)$ could be any $(D_x + D_y)^\beta b_g(x, y; t) $ for $|\beta| \leq 2$. Given $u_0 \in C^2(M)$, this concentration condition above is exactly what we need to show that when $t \rightarrow 0^+$, $t^{\frac{|\alpha|}{4}} (D_x^\alpha \frac{\partial^2 }{\partial x_k \partial x_l} S_g u_0(t))$ equals to 
\begin{equation}
\begin{split}
 &u_0(x)( t^{\frac{|\alpha|}{4}} D_x^\alpha I_4) + \frac{\partial }{\partial x_l}u_0(x) (t^{\frac{|\alpha|}{4}}D_x^\alpha I_5) + \frac{\partial }{\partial x_k} u_0(x) (t^{\frac{|\alpha|}{4}} D_x^\alpha I_6)\\
 & +\frac{\partial^2 }{\partial x_k \partial x_l}u_0(x) (t^{\frac{|\alpha|}{4}}D_x^\alpha I_2 )+ o(1), 
 \end{split}
\end{equation}
where $I_4, I_5, I_6, I_2$ are as we defined before. When $\alpha > 0$, by Theorem $\ref{thm303}$ one can show easily that $D_x^\alpha I_4, D_x^\alpha I_5, D_x^\alpha I_6, D_x^\alpha I_2$ all go to zero uniformly as $t \rightarrow 0^+$. As a result we have that $\lim_{t \rightarrow 0^+} t^{\frac{k}{4} -\frac{1}{2}}\|\nabla^k S_g u_0(t)\|_{C^0(M)} = 0$ for any $u_0\in C^2(M)$ and $k>2$. Lastly, the estimates on derivative of $S_g u_0$ with respect to $t$ can be obtained by rewriting $\frac{\partial}{\partial t} S_g u_0 = - \Delta_g^2 S_g u_0$. Thus it ends the proof of Theorem $\ref{thm001}$

\end{proof}

When $g$ is K\"ahler, as an analog to Theorem $\ref{thm001}$, we have the following theorem.

\begin{thm}\label{thm3.9}
For any $\partial \bar \partial u_0 \in L^\infty(M)$, any smooth K\"ahler metric $g$ and any integer $k, l \geq 0$, we have 
\begin{equation}\label{eqn3-26}
\limsup_{t \rightarrow 0^+} \big( t^{l + \frac{k}{4}} \big\|  |(\frac{\partial }{\partial t})^l \nabla_g^k \partial \bar \partial \big( S_g u_0(t) \big) |_g \big\|_{C^0(M)} \big) \leq  C_{l + \frac{k}{4},n}  \big\| |\partial \bar\partial u_0|_g \big\|_{L^\infty (M)}, 
\end{equation}
for some constant $C_{l + \frac{k}{4}, n} > 0$ depending on $l + \frac{k}{4}$ and dimension $n$ only. If we assume further that $\partial \bar \partial u_0 \in C^0(M)$, then
\begin{equation}
\lim_{t \rightarrow 0^+} \big( \|S_g u_0(t) - u_0\|_{C^1(M)} + \| \partial \bar \partial S_g u_0(t) - \partial \bar \partial u_0 \|_{C^0(M)} \big) = 0, 
\end{equation}
and for any integers $k , l \geq 0$ with $l+k >0$, we have
\begin{equation}
\lim_{t \rightarrow 0^+} ( t^{l + \frac{k}{4}} \big\|  |(\frac{\partial }{\partial t})^l \nabla_g^k \partial \bar \partial \big( S_g u_0(t) \big) |_g \big\|_{C^0(M)}) = 0. 
\end{equation}
\end{thm}

\begin{proof}
Given any $z \in M$, we work under local holomorphic coordinate in a neighborhood of $z$, say $U \subset M$, such that $\frac{1}{2} \delta_{i\bar j} \leq g_{i\bar j} \leq 2 \delta_{i\bar j}$. Repeating the computation $(\ref{eqn000062})$, we get that as $t \rightarrow 0^+$, 
\begin{equation}
\begin{split}
& t^{\frac{k}{4}}\nabla_g^k \frac{\partial^2 }{\partial z_i \partial \bar z_j} \big(S_g u_0(z; t) \big) \\
 = & \int_{U} t^{\frac{k}{4}} \nabla_g^k b_g(z, w; t) \frac{\partial^2  u_0 }{\partial z_i \partial \bar z_j} (w) \mathrm d V_g(w)\\
 & + u_0 (z) \big(t^{\frac{k}{4}} \nabla_g^k I_1\big) + \frac{\partial u_0 }{\partial \bar z_j } (z) \big(t^{\frac{k}{4}} \nabla_g^k I_2 \big) + \frac{\partial u_0 }{\partial  z_i} (z) \big(t^{\frac{k}{4}} \nabla_g^k I_3 \big) + o(1), 
 \end{split}
\end{equation}
where
\begin{equation}
\begin{split}
I_1 = & \int_{U} \frac{\partial^2}{\partial z_i \partial \bar z_j} b_g (z,w;t) \mathrm d V_g(w) +\int_{U} \frac{\partial }{\partial \bar w_j} \big( \frac{\partial}{\partial z_i} b_g (z,w;t) \mathrm d V_g(w) \big)\\
&+  \int_{U} \frac{\partial }{\partial w_i}\big(\frac{\partial}{\partial \bar z_j } b_g (z,w;t) \mathrm d V_g(w) \big)+ \int_{U} \frac{\partial^2}{\partial w_i \partial \bar w_j} \big( b_g (z,w;t) \mathrm d V_g(w) \big), \\
I_2 = & \int_{U} \frac{\partial }{\partial z_i} b_g(z,w;t) \mathrm d V_g(w) + \int_{U}  \frac{\partial }{\partial w_i} \big(  b_g(z,w;t) \mathrm d V_g(w) \big),\\
I_3 = & \int_{U} \frac{\partial }{\partial \bar z_j} b_g(z,w;t) \mathrm d V_g(w) + \int_{U}  \frac{\partial }{\partial \bar w_j} \big(  b_g(z,w;t) \mathrm d V_g(w) \big). \\ 
\end{split}
\end{equation}
By the same argument above, we have that $t^{\frac{k}{4}} |\nabla_g^k I_1|_g, t^{\frac{k}{4}} |\nabla_g^k I_2|_g, t^{\frac{k}{4}} |\nabla_g^k I_3|_g \rightarrow 0$ as $t \rightarrow 0^+$. Therefore we have as $t \rightarrow 0^+$
\begin{equation}
t^{\frac{k}{4}} |\nabla_g^k \frac{\partial^2 }{\partial z_i \partial \bar z_j} \big(S_g u_0(z; t) \big)|_g \leq  \big[\int_M t^{\frac{k}{4}} |\nabla_g^k b_g|_g(z, w; t) \mathrm d V_g(w) \big] \big\||\partial \bar \partial u_0|_g \big\|_{C^0(M)} + o(1)
\end{equation}
By Theorem $\ref{thm2-1}$, we have 
\begin{equation}
\limsup_{t \rightarrow 0^+} \big( t^{\frac{k}{4}}  \big\| |\nabla_g^k\partial \bar \partial \big( S_g u_0 (t) \big) |_g \big\|_{C^0(M)} \big) \leq C_{k,n} \big\||\partial \bar \partial u_0|_g\big\|_{C^0(M)}. 
\end{equation}
for some constant $C_{k,n} > 0$ depending on $k$ and dimension $n$ only. The bounds on $t$-derivatives is obtained by the relation $\frac{\partial }{\partial t} S_g u_0  = - \Delta_g^2 S_g u_0 $. This ends the proof. 

\end{proof}

\subsection{Nonhomogeneous problem}\label{sec3.3}

In this section, we fix a smooth Riemannian metric $g$ and study the nonhomogeneous problem
\begin{equation}\label{eqn7001}
\begin{split}
&( \frac{\partial }{\partial t} + \Delta_g^2 ) u(x; t) = f(x;t) \text{ for } (x, t) \in M \times (0,T), \\
& u(x; 0) = 0 \text{ for } x \in M. 
\end{split}
\end{equation}
where $f $ is a function on $M \times (0,T)$. Given the biharmonic heat kernel $b_g \in C^\infty\big(M \times M \times (0,T)\big)$ with respect to metric $g$, in principle, the solution to the nonhomogeneous problem above, if exists, should be given by the following integral, often called the \emph{volume potential}.  
\begin{equation}
V [f] (x, t) := \int_{0}^t \int_M b_g(x, y; t-s) f(y; s) \mathrm d V_g(y) \mathrm d s. 
\end{equation}
It is in fact an improper integral, so first of all we have to answer if this definition even makes sense at all. 

The regularity of $V[f]$ depends on regularity of the function $f$ of course. For instance, if $f \in C^{\alpha, \frac{\alpha}{4}}(M \times [0,T])$, which is the standard parabolic H\"older space(we refer to \cite{Si} for its precise definition), then we have $V[f] \in C^{4+\alpha, 1+\frac{\alpha}{4}}(M \times [0,T])$. This is the well-known parabolic Schauder estimate. There are vast literatures about Schauder estimate for second order parabolic equations(\cite{LSU},\cite{ADN},\cite{Gi},\cite{Si},\cite{Kr},\cite{Wa} and etc.). While for the higher order case, for any differential operator parabolic in the Petrowski sense with uniform H\"older continuous coefficients on domains of $\bR^n$, its fundamental solution was constructed in 1950's by employing a heavy machinery, called the parametrix method, which goes back to Levi(\cite{Le}). Based on the estimates of the fundamental solutions, Solonnikov(\cite{So}) showed the (regularity) interior parabolic Schauder estimate on domains of $\bR^n$ by studying the volume potential. We refer interested readers to Eidelman(\cite{Ei}), Friedman(\cite{Fr}) and references therein for a complete exposition of the related techniques developed for general (high order) parabolic equations.

In this section, we follow essentially the ideas in Eidelman's book(\cite[Property 10 on p108]{Ei}). However, instead of working on the standard parabolic H\"older spaces ($C^{4+\alpha, 1+\frac{\alpha}{4}}$ and $C^{\alpha, \frac{\alpha}{4}}$),  we'll introduce the \emph{weighted parabolic H\"older spaces}(see respectively spaces $X_T$  and $Y_T$ below). We then establish a Schauder type estimate between the weighted parabolic H\"older spaces. 

Recall, in the definition of standard parabolic H\"older space, $f\in C^{\alpha, \frac{\alpha}{4}}(M \times [0,T])$ means that for any fixed $t\in [0,T]$, $f(\cdot; t) \in C^{\alpha}(M)$; for any fixed $x \in M$, $f(x; \cdot)$ is locally H\"older continuous of exponent $\frac{\alpha}{4}$; moreover, the H\"older norm of $f$ at time $t$, namely
$$\|f(t)\|_{C^0(M)}  + [f(t)]_{C^\alpha(M)}+ \sup_{x \in M} \sup_{0< h <T-t} \frac{|f(x; t+h) - f(x; t)| }{|h|^{\frac{\alpha}{4}}}, $$
is uniformly bounded for all $t \in [0,T]$. The main idea of constructing our weighted spaces is to put appropriate powers of $t$ as weights on each term of the quantity above which allows them to blow up at certain rate when $t \rightarrow 0^+$. More precisely, we have the following definitions.

Fix $0< T < \infty$ and $0< \alpha < 1$. We are mainly interested in functions $f$'s lying in the following space.
\begin{equation}\label{space1}
Y_T = \{f \in C^{0}\big(M \times (0,T) \big) \big| \|f\|_{Y_T} < \infty. \},
\end{equation}
where
\begin{equation}
\begin{split}
\|f\|_{Y_T} =&  \sup_{t\in(0,T) } \big( t^{\frac{1}{2}} \|f(t)\|_{C^0(M)} + t^{\frac{1}{2} + \frac{\alpha}{4}} [f(t)]_{C^{\alpha}(M)} \big) \\
&+ \sup_{(x,t) \in M \times (0,T)} \sup_{0< h < T-t} t^{\frac{1}{2} + \frac{\alpha}{4}}\frac{|f(x; t+h) - f(x;t)|}{|h|^{\frac{\alpha}{4}}}. 
\end{split}
\end{equation} 
One can check easily that $Y_T$ is a Banach space with the norm $\|\cdot\|_{Y_T}$. We also introduce the following space of functions $u$'s on $M \times (0,T)$ which admit continuous derivatives up to fourth order w.r.t. $x \in M$ and continuous first derivative w.r.t. $t \in (0,T)$,  
\begin{equation}\label{space2}
X_T = \{u \big| u \text{ has the regularities described above and } \|u\|_{X_T} < \infty. \},
\end{equation}
where 
\begin{equation}
\begin{split}
\|u\|_{X_T} =& \sup_{t\in (0,T)} \big( \sum_{k=0}^4  t^{- \frac{1}{2} + \frac{k}{4}} \|\nabla^k  u(t)\|_{C^0(M)}+  t^{\frac{1}{2}+ \frac{\alpha}{4}} [\nabla^4 u(t)]_{C^{\alpha}(M)} \\
& + t^{\frac{1}{2}} \|\frac{\partial }{\partial t} u(t)\|_{C^{0}(M)} +  t^{\frac{1}{2}+\frac{\alpha}{4}} [\frac{\partial }{\partial t} u(t)]_{C^{\alpha}(M)} \big) \\
&+ \sup_{(x,t) \in M \times (0,T)} \sup_{0< h < T-t} t^{\frac{1}{2}+\frac{\alpha}{4}}\frac{|\nabla^4u (x; t+h)- \nabla^4 u(x;t)|_g}{|h|^{\frac{\alpha}{4}}}\\
& + \sup_{(x,t) \in M \times (0,T)} \sup_{0< h < T-t} t^{\frac{1}{2}+\frac{\alpha}{4}}\frac{|\frac{\partial }{\partial t} u (x; t+h)- \frac{\partial }{\partial t} u (x;t)|}{|h|^{\frac{\alpha}{4}}}. 
 \end{split}
\end{equation}
Clearly $X_T$ is a Banach space with norm $\|\cdot \|_{X_T}$. 

\begin{lem}
$\|\cdot\|_{X_T}$ above is equivalent to the following norm 
\begin{equation}
\begin{split}
\|u\|_{X_T}^{'} : = & \|u\|_{X_T} + \sum_{k= 0}^3 \sup_{(x,t) \in M \times (0,T)} \sup_{0< h < T-t} t^{- \frac{1}{2} + \frac{k}{4}+\frac{\alpha}{4}}\frac{|\nabla^k u (x; t+h)- \nabla^k u(x;t)|_g}{|h|^{\frac{\alpha}{4}}}.
\end{split}
\end{equation}

\end{lem}

\begin{proof}
This follows directly from interpolations. So we omit its proof. 
\end{proof}

We have the following theorem about the volume potential $V[f]$ for $f \in Y_T$. One can also interpret the following theorem as the (regularity) Schauder estimate between weighted parabolic H\"older spaces for the biharmonic heat equation.

\begin{thm}\label{thm701}
For fixed $0 < T < \infty$, if $f \in Y_T$, then $V[f] \in X_T$ and we have for some constant $C_{g,T}>0$ depending on $g$ and $T$, 
\begin{equation}
\|V[f]\|_{X_T} \leq C_{g,T} \|f\|_{Y_T}. 
\end{equation}
Moreover equation $( \frac{\partial }{\partial t } + \Delta_g^2 ) V[f] = f$ holds in the classical sense on $M \times (0,T)$ and when $t \rightarrow 0^+$, $V[f](t) \rightarrow 0$ in any $C^{1,\gamma}(M)$ for $0< \gamma <1$. 
\end{thm}

\begin{rmk}
Theorem $\ref{thm701}$ asserts the existence of solution $u = V[f] \in X_T$ to the nonhomogeneous problem $(\ref{eqn7001})$ if the right hand side function $f \in Y_T$. In fact, such solution is unique in $X_T$. Namely, if there exists an another $v \in X_T$ such that $( \frac{\partial }{\partial t } + \Delta_g^2 ) v = f$, then $v = V[f]$. This is because for any $t >0$ 
\begin{equation}
\frac{\mathrm d}{\mathrm d t} \int_M \big(v(t) - V[f](t)\big)^2 \mathrm d V_g = - \int_M  [\Delta_g \big( v(t) - V[f](t) \big)]^2 \mathrm d V_{g} \leq 0, 
\end{equation}
and as $t \rightarrow 0^+$, $ \int_M \big(v(t) - V[f](t)\big)^2 \mathrm d V_g \rightarrow 0$. 
\end{rmk}

\begin{rmk}
The constant $C_{g,T}>0$ in Theorem $\ref{thm701}$ depends on metric $g$ and $T$. For fixed smooth metric $g$, $C_{g,T}$ is bounded whenever $T$ is bounded and $C_{g,T} \rightarrow \infty$ as $T \rightarrow \infty$. For fixed $T < \infty$, $C_{g,T}$ depends on $g$ in an unclear way(might depend on high order derivatives of $g$, $\text{Vol}(M)$, injectivity radius and etc.) since it comes from various bounds on the biharmonic heat kernel $b$. 

However, when $T$ is sufficiently small, we can choose the constant $C_{g,T}$ uniformly bounded independent of $g,T$.
\end{rmk}

To be more precise on the last remark, we introduce the following theorem.

\begin{thm}\label{thm702}
Suppose $M$ is a closed manifold of dimension $n$ and $g$ is a smooth Riemannian metric on $M$. There exist constants $T_g >0$ and $C>0$ where $T_g$ depends on $\|\text{Ric}(g)\|_{C^{1,\alpha}}$ and the injectivity radius $i_0$, and $C=C_n$ is a dimensional constant, such that for any $0< T < T_g $, we have 
\begin{equation}
\|V[f]\|_{X_T} \leq C_n \|f\|_{Y_T}.
\end{equation}
\end{thm}

Next we'll prove Theorem $\ref{thm701}$ first and then Theorem $\ref{thm702}$ in following subsections.  

\begin{proof}
The proof goes somewhat towards verifying the validation of the volume potential $V[f]$. This proof follows a similar idea in proving the standard Schauder estimate but instead we work on the weighted parabolic H\"older spaces. To begin with, we have by the estimates of $b$ in Theorem $\ref{thm301}$, for some appropriate constants $C,\delta >0$ depending on $g,T$, 
\begin{equation}
\begin{split}
& | V[f] (x; t) | \\
 \leq &C \|f\|_{Y_T}\int_0^t s^{-\frac{1}{2}} \big( \int_M (t-s)^{-\frac{n}{4}} \exp\{- \delta \big((t-s)^{-\frac{1}{4}} \rho(x,y)\big)^{\frac{4}{3}}\} \mathrm d V_g(y) \big)  \mathrm d s. \\
 \leq & C' t^{\frac{1}{2}} \|f\|_{Y_T}.
\end{split}
\end{equation}
Last step is because
\begin{equation}
\begin{split}
&  \int_M (t-s)^{-\frac{n}{4}} \exp\{- \delta  \big((t-s)^{- \frac{1}{4}} \rho(x, y)\big)^{\frac{4}{3}} \} \mathrm d V_g(y)\\
 \leq & C \big\{ \int_{ B_{r_0}(0) } (t-s)^{-\frac{n}{4}}  \exp\{-  2 ^{-\frac{4}{3}}\delta  \big((t-s)^{- \frac{1}{4}}|w|\big)^{\frac{4}{3}} \} \mathrm d w\\
  &+ ( t- s)^{-\frac{n}{4}} \exp\{- \delta \big((t-s)^{- \frac{1}{4}} r_0\big)^{\frac{4}{3}} \} (\int_{M - B_{r_0}(x, g)} \mathrm d V_g)   \big\} <\infty. 
\end{split}
\end{equation}
For similar reasons, we have for $0 \leq k \leq 3$,
\begin{equation}\label{eqn7002}
\nabla^k V[f] (x; t) = \int_0^t \int_M \nabla^k b_g(x, y; t-s) f(y; s) \mathrm d V_g(y) \mathrm d s,
\end{equation}
and
\begin{equation}
\|\nabla^{k} V[f](t)\|_{C^0(M)} \leq t^{\frac{1}{2} - \frac{k}{4}} C_k \|f\|_{Y_T}, 
\end{equation}
for some constant $C_k$ depending on $g, T$ and $k$. 

While for $k=4$, formula $(\ref{eqn7002})$ doesn't hold simply because $\nabla^4 b (x, \cdot; t- \cdot)$ may not be integrable on $M \times (0,t)$. Instead, for $h \ll 1$ we compute in local coordinates, 
\begin{equation}\label{eqn7003}
\begin{split}
& \nabla^3 V[f] (x + h e_l; t ) - \nabla^3 V[f](x; t) \\
= & \int_0^t \int_M \int_0^h \frac{\partial }{\partial x_l} \nabla^3 b_g(x+ \tau e_l , y; t-s)  f(y; s) \mathrm d \tau \mathrm d V_g(y) \mathrm ds, \\
= & \int_0^t \int_M \int_0^h \frac{\partial }{\partial x_l} \nabla^3 b_g(x+ \tau e_l , y; t-s) \big(  f(y; s) - f(x+ \tau e_l; s) \big) \mathrm d \tau \mathrm d V_g(y) \mathrm ds \\
& +  \int_0^t \int_M \int_0^h \frac{\partial }{\partial x_l} \nabla^3 b_g(x+ \tau e_l , y; t-s) f(x+ \tau e_l;s )  \mathrm d \tau \mathrm d V_g(y) \mathrm ds. 
\end{split}
\end{equation}
Regarding to each term above we have 
\begin{equation} \label{eqn7005}
\begin{split}
& \int_0^t \int_M  \big| \frac{\partial }{\partial x_l}  \nabla^3 b_g(x+ \tau e_l , y; t-s) \big(  f(y; s) - f(x+ \tau e_l; s) \big) \big| \mathrm d V_g(y) \mathrm ds \\
\leq & C \|f\|_{Y_T} \int_0^t \int_M (t-s)^{-\frac{n+4}{4}} \exp\{- \delta \big((t-s)^{-\frac{1}{4}} \rho(x + \tau e_l ,y) \big)\}  s^{-\frac{1}{2}- \frac{\alpha}{4}} \rho(x+ \tau e_l, y)^\alpha \mathrm d V_g(y) \mathrm d s,\\
\leq &  C'  \|f\|_{Y_T} \int_0^t s^{-\frac{1}{2}- \frac{\alpha}{4}}  \int_M (t-s)^{-\frac{n+4- \alpha}{4}} \exp\{- \delta'  \big((t-s)^{-\frac{1}{4}} \rho(x + \tau e_l ,y) \big)\}   \mathrm d V_g(y) \mathrm d s,\\
\leq & C''  \|f\|_{Y_T} \int_0^t s^{-\frac{1}{2}- \frac{\alpha}{4}} (t-s)^{-\frac{4-\alpha}{4}} \mathrm d s, \\
\leq & C''' \|f\|_{Y_T}  t^{-\frac{1}{2}}.
\end{split}
\end{equation}
Also, by Theorem $\ref{thm302}$, in local coordinates $x \in V $ such that $\bar V \subset U$, we have
\begin{equation}\label{eqn7005}
\begin{split}
&\int_M \frac{\partial }{\partial x_l} \nabla^3 b_g(x+ \tau e_l , y; t-s)  \mathrm d V_g(y) \\
 =& \int_U (\frac{\partial }{\partial x_l} + \frac{\partial }{\partial y_l }) \nabla^3 b_g(x+ \tau e_l, y; t-s)  \mathrm  d V_g(y) +  \int_U  \nabla^3 b_g(x+ \tau e_l, y; t-s) \frac{\partial }{\partial y_l } ( \mathrm  d V_g(y) )\\
 & + \int_{\partial U} \nabla^3 b_g(x+ \tau e_l , y; t-s)  \mathrm d \sigma_l(y) + \int_{M - U } \frac{\partial }{\partial x_l} \nabla^3 b_g(x+ \tau e_l , y; t-s)  \mathrm d V_g(y). 
 \end{split} 
\end{equation}
Thus by Theorem $\ref{thm302}$, 
\begin{equation}
| \int_M \frac{\partial }{\partial x_l} \nabla^3 b_g(x+ \tau e_l , y; t-s)  \mathrm d V_g(y) | \leq  C (t-s)^{-\frac{3}{4}} + C \leq C (t-s)^{- \frac{3}{4}},
\end{equation}
when $t \ll 1$. Therefore
\begin{equation}\label{eqn7006}
\begin{split}
&\int_0^t | \big( \int_M  \frac{\partial }{\partial x_l} \nabla^3 b_g(x+ \tau e_l , y; t-s)    \mathrm d V_g(y) \big) f(x+ \tau e_l;s )| \mathrm ds\\
\leq & C \|f\|_{Y_T} \int_0^t (t-s)^{-\frac{3}{4}} s^{- \frac{1}{2}} ds \leq C'  \|f\|_{Y_T} t^{-\frac{1}{4}}. 
\end{split}
\end{equation}
Thus each integral in $(\ref{eqn7003})$ is integrable. By Fubini's theorem, we can change the orders of integrations and then
\begin{equation}
\begin{split}
& \nabla^3 V[f] (x + h e_l; t ) - \nabla^3 V[f](x; t) \\
= &\int_0^h \{ \int_0^t \int_M  \frac{\partial }{\partial x_l} \nabla^3 b_g(x+ \tau e_l , y; t-s) \big(  f(y; s) - f(x+ \tau e_l; s) \big) \mathrm d V_g(y) \mathrm ds \\
& +  \int_0^t \big( \int_M  \frac{\partial }{\partial x_l} \nabla^3 b_g(x+ \tau e_l , y; t-s)  \mathrm d V_g(y) \big) f(x+ \tau e_l;s )  \mathrm ds\} \mathrm d \tau . 
\end{split}
\end{equation}

In order to show that $\nabla^3 V[f] (x; t)$ is differentiable at $x$, it suffices to show that the integrand in last integral denoted as $P(x+ \tau e_l ;  t) \rightarrow P(x ; t)$ as $\tau \rightarrow 0$. In fact we will show below for any $x, x' \in M$
\begin{equation}\label{eqn7004}
\begin{split}
& | P(x; t) | \leq C \|f\|_{Y_T} t^{-\frac{1}{2}}, \\
& | P(x ; t) - P(x' ; t) | \leq C\|f\|_{Y_T} t^{- \frac{1}{2} - \frac{\alpha}{4}}  \rho(x, x' )^\alpha, 
\end{split}
\end{equation} 
for some constant $C>0$ depending on $g, T$ only. The first estimate in $(\ref{eqn7004})$ follows from repeating arguments in $(\ref{eqn7005})$ and $(\ref{eqn7006})$. As for the second estimate in $(\ref{eqn7004})$, we need to do some extra work. Without loss of generality, we can assume that $\rho(x, x') \ll 1$ and thus we can work within a single local coordinate chart $V$ such that $\bar V \subset U$. Set $r = |x-x'|$ and $\xi = \frac{x+x'}{2}$. Compute
\begin{equation}
\begin{split}
 & P(x ; t) - P(x' ; t) \\
 = & \int_0^t \int_{B_{r}(\xi)}  \frac{\partial }{\partial x_l} \nabla^3 b_g(x , y; t-s) \big(  f(y; s) - f(x; s) \big) \mathrm d V_g(y) \mathrm ds \\
 &-  \int_0^t \int_{B_{r}(\xi)}  \frac{\partial }{\partial x_l} \nabla^3 b_g(x' , y; t-s) \big(  f(y; s) - f(x'; s) \big) \mathrm d V_g(y) \mathrm ds \\
 &+  \int_0^t \int_{M - B_{r}(\xi)}  \big[ \frac{\partial }{\partial x_l} \nabla^3 b_g(x , y; t-s) \big(  f(y; s) - f(x; s) \big) \\
 & - \frac{\partial }{\partial x_l} \nabla^3 b_g(x' , y; t-s) \big(  f(y; s) - f(x' ; s) \big)  \big] \mathrm d V_g(y) \mathrm ds\\
 &+ \int_0^t \big[ \big( \int_M  \frac{\partial }{\partial x_l} \nabla^3 b_g(x , y; t-s)  \mathrm d V_g(y) \big) f(x;s ) \\
 &- \big( \int_M  \frac{\partial }{\partial x_l} \nabla^3 b_g(x' , y; t-s)  \mathrm d V_g(y) \big) f(x' ;s ) \big]   \mathrm ds, \\
 := & P_1 + P_2 + P_3+ P_4.  
 \end{split}
\end{equation}

Next we estimate each of these $P_i$'s. For $P_1$, we break it into two integrals from $0$ to $\frac{t}{2}$ and from $\frac{t}{2}$ to $t$ and then
\begin{equation}
\begin{split}
|P_1| \leq & C\|f\|_{Y_T} t^{-\frac{1}{2} - \frac{\alpha}{4}}  \int_{B_{2r}(x)}  |x-y|^\alpha \int_{\frac{t}{2}}^{t}  (t-s)^{- \frac{n+4}{4}} \exp\{- \delta \big((t-s)^{-\frac{1}{4}} |x-y|\big)^{\frac{4}{3}}\}  \mathrm ds \mathrm d y \\
 & + C\|f\|_{Y_T} t^{-1} \int_0^{\frac{t}{2}} s^{-\frac{1}{2} -\frac{\alpha}{4}} \mathrm ds \int_{B_{2r}(x)}\big((t-s)^{-\frac{1}{4}}|x-y|\big)^n \exp\{- \delta \big((t-s)^{-\frac{1}{4}} |x-y|\big)^{\frac{4}{3}}\}  |x-y|^{\alpha- n}  \mathrm d y, \\
\leq &C\|f\|_{Y_T} t^{-\frac{1}{2} - \frac{\alpha}{4}}  \int_{B_{2r}(x)} |x-y|^{ \alpha - n} \mathrm d y \int_0^{+ \infty} \tau^{n-1} \exp\{- \delta \tau^{\frac{4}{3}}\}   \mathrm d \tau  \\
 & + C\|f\|_{Y_T} t^{-1} \int_0^{\frac{t}{2}} s^{-\frac{1}{2} -\frac{\alpha}{4}} \mathrm ds \int_{B_{2r}(x)}     |x-y|^{\alpha- n}  \mathrm d y. \\
\end{split}
\end{equation}
In the last step we used a change of variable from $s$ to $\tau$,  $\tau = (t-s)^{-\frac{1}{4}} |x-y|$ where $y \neq x$ is fixed when integrating w.r.t. $s$. Thus we have that 
\begin{equation}
|P_1| \leq C \|f\|_{Y_T} t^{-\frac{1}{2} -\frac{\alpha}{4}} r^\alpha. 
\end{equation}
$P_2$ satisfies the same estimate following exactly the same argument for $P_1$. 

For $P_3$, we can rewrite it as
\begin{equation}
\begin{split}
P_3 = & \int_0^t \int_{M - B_{r}(\xi)}  \big[  \big( \frac{\partial }{\partial x_l} \nabla^3 b_g(x , y; t-s) - \frac{\partial }{\partial x_l} \nabla^3 b_g(x' , y; t-s)\big) \big(  f(y; s) - f(x; s) \big) \\
 & +  \frac{\partial }{\partial x_l} \nabla^3 b_g(x' , y; t-s) \big(  f(x' ; s) - f(x ; s) \big)  \big] \mathrm d V_g(y) \mathrm ds, \\ 
 = & \int_0^t \int_{M - B_{r}(\xi)}  \big[  \big( \frac{\partial }{\partial x_l} \nabla^3 b_g(x , y; t-s) - \frac{\partial }{\partial x_l} \nabla^3 b_g(x' , y; t-s)\big) \big(  f(y; s) - f(x; s) \big) \mathrm d V_g(y) \mathrm ds\\
 & + \int_0^t  \big( \int_{M - B_{r}(\xi)} \frac{\partial }{\partial x_l} \nabla^3 b_g(x' , y; t-s) \mathrm d V_g(y) \big) \big(  f(x' ; s) - f(x ; s) \big)   \mathrm ds, \\
 :=& P_{3a} + P_{3b}. 
\end{split}
\end{equation} 
One can break each integral above into two integrals on $U - B_{r}(\xi)$ and on $M - U$ respectively. Since $y \in M- U$ stays a fixed distance away(namely, $\text{dist}(V,\partial U)$) from $x$ and $x'$, one can check easily that the integral on $M-U$ in $P_{3a}$ and $P_{3b}$ satisfies estimate $(\ref{eqn7004})$. While for the integral on $U - B_{r}(\xi)$ of $P_{3a}$, notice that since $y \notin B_r(\xi)$, we have for any $\eta \in B_{\frac{r}{2}}(\xi)$, $|y - \eta|$ and $|y -\xi|$ are comparable, and more precisely we have $\frac{1}{2}|y- \xi| \leq |y- \eta| \leq \frac{3}{2} |y -\xi| $. Thus the absolute value of the integral on $U - B_{r}(\xi) $ in $P_{3a}$ is less or equal to
\begin{equation}
\begin{split}
& C \|f\|_{Y_T} \int_0^t \int_{|y -\xi| \geq r} s^{-\frac{1}{2} - \frac{\alpha}{4}}  (t-s)^{-\frac{n+5}{4}} |x-x'| \exp\{-\delta \big( (t-s)^{- \frac{1}{4}} |\xi- y|\big)^{\frac{4}{3}} \} |\xi- y|^\alpha \mathrm d y \mathrm ds\\
\leq & C \|f\|_{Y_T} |x-x'| \big\{\int_0^{\frac{t}{2}} \int_{|y -\xi| \geq r} s^{-\frac{1}{2} - \frac{\alpha}{4}}  (t-s)^{-1} \big((t-s)^{-\frac{1}{4}} |\xi- y| \big)^{n+1} \\
&  \times \exp\{-\delta \big( (t-s)^{- \frac{1}{4}} |\xi- y|\big)^{\frac{4}{3}} \} |\xi- y|^{\alpha -(n+1)} \mathrm d y \mathrm ds \\
& + t^{-\frac{1}{2} - \frac{\alpha}{4}}  \int_{|y-\xi| \geq r} |\xi- y|^\alpha \int_{\frac{t}{2}}^{t} (t-s)^{-\frac{n+5}{4}}  \exp\{-\delta \big( (t-s)^{- \frac{1}{4}} |\xi- y|\big)^{\frac{4}{3}} \} \mathrm ds \mathrm d y  \big\}\\
\leq & C \|f\|_{Y_T}  r \big(  t^{-1} \int_0^{\frac{t}{2}} s^{-\frac{1}{2} - \frac{\alpha}{4}} \mathrm ds  \int_{|w| \geq r} |w |^{\alpha -(n+1)} \mathrm d w \\
& + t^{-\frac{1}{2} - \frac{\alpha}{4}} \int_{|w|\geq r} | w |^{\alpha - (n+1)} \mathrm d w \int_0^{+\infty} \tau^{n} \exp\{- \delta \tau^{\frac{4}{3}}\} \mathrm d \tau \big) \leq  C \|f\|_{Y_T} t^{-\frac{1}{2} -\frac{\alpha}{4}} r^{\alpha}. 
\end{split}
\end{equation}
The absolute value of the integral on $U - B_r(\xi)$ of $P_{3b}$ is less or equal to
\begin{equation}
\begin{split}
& C  r^\alpha \|f\|_{Y_T} \int_0^\frac{t}{2}  (t-s)^{-1} s^{-\frac{1}{2} -\frac{\alpha}{4}} \int_{U} (t-s)^{-\frac{n}{4}} \exp\{ - \delta ((t-s)^{-\frac{1}{4}} \rho(x', y))^{\frac{4}{3}}\} \mathrm d V_{g}(y)   \mathrm ds\\
&+ C r^{\alpha}\|f\|_{Y_T} t^{-\frac{1}{2} - \frac{\alpha}{4}} \int_{\frac{t}{2}}^t  \big| \int_{U - B_{r}(\xi)} \frac{\partial }{\partial x_l} \nabla^3 b_g(x' , y; t-s) \mathrm d V_g(y) \big|   \mathrm ds
\end{split}
\end{equation}
We focus on the later integral. Compute 
\begin{equation}
\begin{split}
& \int_{U - B_{r}(\xi)} \frac{\partial }{\partial x_l} \nabla^3 b_g(x' , y; t-s) \mathrm d V_g(y) \\
= & \int_{U - B_{r}(\xi)} ( \frac{\partial }{\partial x_l} + \frac{\partial }{\partial y_l} ) \nabla^3 b_g(x' , y; t-s) \mathrm d V_g(y) +  \int_{U - B_{r}(\xi)}  \nabla^3 b_g(x' , y; t-s)  \frac{\partial }{\partial y_l} \big(\mathrm d V_g(y) \big)\\
& + \int_{\partial U} \nabla^3 b_g(x' , y; t-s) \mathrm d \sigma_l(y) + \int_{\partial B_{r}(\xi)} \nabla^3 b_g(x' , y; t-s) \mathrm d \sigma_l(y). 
\end{split}
\end{equation}
Thus
\begin{equation}
\begin{split}
& \int_{\frac{t}{2}}^{t } \big| \int_{U - B_{r}(\xi)} \frac{\partial }{\partial x_l} \nabla^3 b_g(x' , y; t-s) \mathrm d V_g(y) \big| \mathrm ds \\
\leq & C  \int_{\frac{t}{2}}^t \int_{U}(t-s)^{-\frac{n+3}{4}}  \exp\{ - \delta ((t-s)^{-\frac{1}{4}} \rho(x', y))^{\frac{4}{3}}\} \mathrm d V_{g}(y)   \mathrm ds \\
& +C \int_{\frac{t}{2}}^t \int_{\partial U} (t-s)^{-\frac{n+3}{4}}  \exp\{ - \delta ((t-s)^{-\frac{1}{4}} \text{dist}(V, \partial U))^{\frac{4}{3}}\} \mathrm d \sigma_l(y)   \mathrm ds\\
& + C \int_{\partial B_r(\xi)} \big( \int_{\frac{t}{2}}^t  (t-s)^{-\frac{n+3}{4}}  \exp\{ - \delta ((t-s)^{-\frac{1}{4}} |\xi- y|)^{\frac{4}{3}}\}     \mathrm ds \big)  \mathrm d \sigma_l (y)\\
\leq & C t^{\frac{1}{4}} + C t + C \int_{\partial B_r(\xi)} |\xi- y |^{- (n-1) } \mathrm d \sigma_l (y) \int_0^{+\infty} \tau^{n-2} \exp\{ - \delta \tau^{\frac{4}{3}}\} \mathrm d \tau \leq C' . 
\end{split}
\end{equation}
Therefore, we have the integral on $U-B_{r}(\xi)$ of $P_{3b}$ less or equal to
\begin{equation}
\begin{split}
C \|f\|_{Y_T} t^{-1} r^\alpha \int_0^{\frac{t}{2}} s^{-\frac{1}{2} - \frac{\alpha}{4}} \mathrm ds + C\|f\|_{Y_T} t^{-\frac{1}{2} -\frac{\alpha}{4}} r^\alpha \leq C \|f\|_{Y_T} t^{-\frac{1}{2} -\frac{\alpha}{4}} r^\alpha. 
\end{split}
\end{equation}
Therefore combining results above we have $|P_3| \leq C \|f\|_{Y_T} t^{-\frac{1}{2} -\frac{\alpha}{4}} r^\alpha$ for some constant $C>0$ depending on $g, T$. 

For $P_4$, we have by virtue of computation $(\ref{eqn7005})$,
\begin{equation}\label{eqn7102}
\begin{split}
& | \int_M \frac{\partial }{\partial x_l} \nabla^3 b_g(x , y; t-s)  \mathrm d V_g(y) | \leq C (t-s)^{-\frac{3}{4}}, \\
& | \int_M \frac{\partial }{\partial x_l} \nabla^3 b_g(x , y; t-s)  \mathrm d V_g(y) -  \int_M \frac{\partial }{\partial x_l} \nabla^3 b_g(x'  , y; t-s)  \mathrm d V_g(y)|\\
& \leq  C (t-s)^{-\frac{3+\alpha}{4}} |x- x'|^\alpha. 
\end{split}
\end{equation}
Also we have $|f(x; s)| \leq s^{-\frac{1}{2}} \|f\|_{Y_T}$ and $|f(x;s) - f(x';s)| \leq s^{-\frac{1}{2} -\frac{\alpha}{4}} \|f\|_{Y_T} |x- x'|^\alpha$ for $0< s <T$. Putting these estimates together in $P_4$, we get that
\begin{equation}
|P_4| \leq C t^{- \frac{1}{4} - \frac{\alpha}{4}} \|f\|_{Y_T} r^\alpha \leq C  t^{- \frac{1}{2} - \frac{\alpha}{4}} \|f\|_{Y_T} r^\alpha. 
\end{equation}
Thus combining results of $P_1, P_2, P_3, P_4$, we have shown that for $x,x' \in V$ with $\rho(x,x') \ll 1$ 
\begin{equation}
|P(x;t) - P(x' ; t) | \leq C t^{-\frac{1}{2}  - \frac{\alpha}{4}}  \|f\|_{Y_T} r^\alpha. 
\end{equation}
Consequently, we have $V[f]$ is differentiable w.r.t. $x$ up to fourth order and 
\begin{equation}
\begin{split}
\nabla^4 V[f] (x; t) = & \int_0^t \int_M   \nabla^4 b_g(x , y; t-s) \big(  f(y; s) - f(x; s) \big) \mathrm d V_g(y) \mathrm ds \\
&+  \int_0^t \big( \int_M  \nabla^4 b_g(x, y; t-s)  \mathrm d V_g(y) \big) f(x;s )  \mathrm ds. 
\end{split}
\end{equation}
Moreover, we have that $ \sup_{t \in (0,T)} \big( t^{\frac{1}{2}} \|\nabla^4 V[f](t)\|_{C^0(M)}+ t^{\frac{1}{2} +\frac{\alpha}{4}}\big[\nabla^4 V[f](t)\big]_{C^\alpha(M)} \big) \leq C \|f\|_{Y_T}$ for some constant $C>0$ depending on $g,T$. 

Next we will prove that $\nabla^4 V[f](x, t)$ is H\"older continuous w.r.t. $t$. Consider for $0< h \ll t$ say $h < \min\{\frac{t}{2}, T - t\}$, 
\begin{equation}
\begin{split}
& \nabla^4 V[f] (x; t+h ) - \nabla^4 V[f] (x; t) \\
=  & \int_t^{t+h} \int_M \nabla^4 b_g(x,y; t+ h -s ) \big( f(y; s) - f(x;s) \big) \mathrm d V_g(y)   \mathrm ds  \\
& + \int_t^{t+h}  \big( \int_M \nabla^4 b_g(x,y; t+ h -s )  \mathrm d V_g(y) \big) f(x;s)    \mathrm ds \\
& + \int_0^h  \big\{ \int_0^t \int_M  \frac{\partial}{\partial t } \nabla^4 b_g(x , y; t+\tau-s) \big(  f(y; s) - f(x; s) \big) \mathrm d V_g(y) \mathrm ds \\
&+  \int_0^t \big( \int_M  \frac{\partial }{\partial t} \nabla^4 b_g(x, y; t+ \tau -s)  \mathrm d V_g(y) \big) f(x;s )  \mathrm ds \big\} \mathrm d \tau\\
:= & J_1 + J_2+J_3.
\end{split}
\end{equation}

We have 
\begin{equation}
\begin{split}
|J_1| &  = |\int_{0}^h \int_M \nabla^4 b_g(x,y; \tau ) \big( f(y; t+h - \tau) - f(x;t+h -\tau) \big) \mathrm d V_g(y)   \mathrm d\tau| \\
& \leq C \|f\|_{Y_T} \int_0^h \int_M \tau^{-\frac{n+4}{4}} \exp\{ - \delta (\tau^{-\frac{1}{4}} \rho(x, y))^{\frac{4}{3}}\} t^{-\frac{1}{2}- \frac{\alpha}{4}} \rho(x, y)^\alpha \mathrm dV_g(y) \mathrm d \tau\\
& \leq C\|f\|_{Y_T}  t^{-\frac{1}{2}- \frac{\alpha}{4}} \int_0^h \int_M \tau^{-\frac{n+4- \alpha}{4}} \exp\{ - \delta' (\tau^{-\frac{1}{4}} \rho(x, y))^{\frac{4}{3}}\}  \mathrm dV_g(y) \mathrm d \tau\\
& \leq C \|f\|_{Y_T}  t^{-\frac{1}{2}- \frac{\alpha}{4}} \int_0^h \tau^{-\frac{4- \alpha}{4}}  \mathrm d \tau \leq C \|f\|_{Y_T}  t^{-\frac{1}{2}- \frac{\alpha}{4}} h^{\frac{\alpha}{4}}. 
\end{split}
\end{equation}
By virtue of $(\ref{eqn7102})$, 
\begin{equation}
|J_2| \leq C \|f\|_{Y_T} \int_{t}^{t+h} (t+h-s)^{-\frac{3}{4}} s^{-\frac{1}{2}} \mathrm ds \leq C \|f\|_{Y_T} t^{- \frac{1}{2}} h^{\frac{1}{4}} \leq C\|f\|_{Y_T} t^{-\frac{1}{2} -\frac{\alpha}{4}} h^{\frac{\alpha}{4}}. 
\end{equation}
For $J_3$, we have
\begin{equation}
\begin{split}
|J_3| \leq & C\|f\|_{Y_T} \int_0^h \big\{ \int_{0}^{t} \int_M (t+\tau -s)^{- \frac{n+8}{4}} \exp\{ - \delta \big((t+\tau -s)^{-\frac{1}{4}} \rho(x, y)\big)^{\frac{4}{3}} \} s^{-\frac{1}{2}- \frac{\alpha}{4}} \rho(x, y)^\alpha \mathrm d V_g(y) \mathrm ds\\
&   + \int_{0}^t \int_M (t+\tau -s)^{-\frac{n+7}{4}} \exp\{ - \delta \big((t+\tau -s)^{-\frac{1}{4}} \rho(x, y)\big)^{\frac{4}{3}} \} s^{-\frac{1}{2}} \mathrm d V_g(y) \mathrm d s \big\} \mathrm d \tau \\
\leq & C\|f\|_{Y_T} \int_0^h \big\{ \int_{0}^{t} \int_M (t+\tau -s)^{- \frac{n+8-\alpha}{4}} \exp\{ - \delta'  \big((t+\tau -s)^{-\frac{1}{4}} \rho(x, y)\big)^{\frac{4}{3}} \} s^{-\frac{1}{2}- \frac{\alpha}{4}}  \mathrm d V_g(y) \mathrm ds\\
&   + \int_{0}^t \int_M (t+\tau -s)^{-\frac{n+7}{4}} \exp\{ - \delta \big((t+\tau -s)^{-\frac{1}{4}} \rho(x, y)\big)^{\frac{4}{3}} \} s^{-\frac{1}{2}} \mathrm d V_g(y) \mathrm d s \big\} \mathrm d \tau \\
\leq & C\|f\|_{Y_T} \int_0^h \big\{ \int_{0}^{t} (t+\tau -s)^{- \frac{8-\alpha}{4}} s^{-\frac{1}{2}- \frac{\alpha}{4}}  \mathrm ds   + \int_{0}^t  (t+\tau -s)^{-\frac{7}{4}} s^{-\frac{1}{2}} \mathrm d s \big\} \mathrm d \tau \\
\leq & C \|f\|_{Y_T} (t^{-\frac{8-\alpha}{4}} t^{\frac{1}{2} -\frac{\alpha}{4}} h + t^{-\frac{1}{2}- \frac{\alpha}{4}} h^{\frac{\alpha}{4}} + t^{-\frac{7}{4}} t^{\frac{1}{2}} h + t^{-\frac{1}{2}} h^{\frac{1}{4}})\\
\leq & C\|f\|_{Y_T} (t^{-\frac{3}{2}} h + t^{-\frac{1}{2} -\frac{\alpha}{4}} h^{\frac{\alpha}{4}}) \leq C \|f\|_{Y_T}  t^{-\frac{1}{2} -\frac{\alpha}{4}} h^{\frac{\alpha}{4}}
\end{split}
\end{equation}

Thus we can get that for $0<h \ll t$ say $h < \min\{\frac{t}{2}, T - t\}$
\begin{equation}
\begin{split}
& | \nabla^4 V[f] (x; t+h ) - \nabla^4 V[f] (x; t) | \\
\leq & |J_1| +|J_2|+|J_3| \leq C \|f\|_{Y_T} t^{-\frac{1}{2} - \frac{\alpha}{4}} h^{\frac{\alpha}{4}}. 
\end{split}
\end{equation}

Therefore 
\begin{equation}
\begin{split}
&\sup_{(x, t) \in M \times (0,T)}\sup_{0< h < T-t} t^{\frac{1}{2} + \frac{\alpha}{4}}\frac{|\nabla^4 V [f] (x; t+h) - \nabla^4 V[f] (x;t)|_g}{|h|^{\frac{\alpha}{4}}}\\
\leq & \sup_{(x, t) \in M \times (0,T)}\sup_{0< h < \min\{\frac{t}{2}, T-t\}} t^{\frac{1}{2} + \frac{\alpha}{4}}\frac{|\nabla^4 V [f] (x; t+h) - \nabla^4 V[f] (x;t)|_g}{|h|^{\frac{\alpha}{4}}} \\
&+ \sup_{(x, t) \in M \times (0,T)}\sup_{\min\{\frac{t}{2}, T-t\} < h < T-t} t^{\frac{1}{2} + \frac{\alpha}{4}}\frac{|\nabla^4 V [f] (x; t+h) - \nabla^4 V[f] (x;t)|_g}{|h|^{\frac{\alpha}{4}}} \\
\leq & C \|f\|_{Y_T} +  C \sup_{(x, t) \in M \times (0,T)}\sup_{\min\{\frac{t}{2}, T-t\} < h < T-t} t^{\frac{1}{2}} (|\nabla^4 V[f](x; t+h)|_g + |\nabla^4 V[f](x; t)|_g),\\
\leq& C \|f\|_{Y_T}
\end{split}
\end{equation}

Lastly, we will show that 
\begin{equation}\label{eqn7100}
\begin{split}
\frac{\partial}{\partial t} V[f](x; t)  = & f(x; t) + \int_0^t \int_M  \frac{\partial}{\partial t} b_g(x , y; t-s) \big(  f(y; s) - f(x; s) \big) \mathrm d V_g(y) \mathrm ds \\
&+  \int_0^t \big( \int_M  \frac{\partial}{\partial t} b_g(x, y; t-s)  \mathrm d V_g(y) \big) f(x;s )  \mathrm ds.
\end{split}
\end{equation}
Suppose the formula above is true. As a consequence, we have that $\frac{\partial}{\partial t} V[f](x; t)  + \Delta_g^2 V[f](x; t) = f(x;t)$ holds in the classical sense and then the estimates on H\"older norms of $\frac{\partial }{\partial t } V[f]$ w.r.t. $x$ and $t$ follow from estimates on $\nabla^4 V[f]$ above. Thus we have $V[f] \in X_T$ and $\|V[f]\|_{X_T} \leq C_{g,T} \|f\|_{Y_T}$. 

Thus, next we focus on proving the formula $(\ref{eqn7100})$ for $\frac{\partial}{\partial t} V[f] $. Consider for $0< h \ll t$,
\begin{equation}
\begin{split}
 & \frac{V[f](x; t+h) - V[f](x;t)}{h} \\
 = &  \frac{1}{h} \int_t^{t+h} \int_M b_g(x, y; t+h-s) f(y; s) \mathrm dV_g(y) \mathrm d s   \\
 &+ \int_0^t \int_M \frac{1}{h} \big(b_g(x , y; t+h-s) - b_g(x, y; t-s) \big) \big(  f(y; s) - f(x; s) \big) \mathrm d V_g(y) \mathrm ds \\
&+  \int_0^t \big( \int_M  \frac{1}{h} \big(b_g(x , y; t+h-s) - b_g(x, y; t-s) \big)  \mathrm d V_g(y) \big) f(x;s )  \mathrm ds.
\end{split}
\end{equation}
As $h \rightarrow 0^+$, the later two integrals above(from $0$ to $t$) go to the later two integrals in formula $(\ref{eqn7100})$ respectively. One can easily prove this using the mean value theorem and also the continuity of the resulting integrand w.r.t $t$ by rewriting $\frac{\partial }{\partial t} b = - \Delta_g^2 b$. While for the first integral in $(\ref{eqn7100})$ we shall see it goes to $f(x;t)$ as $h \rightarrow 0^+$. This is because by a change of variable it equals to
\begin{equation}
\begin{split}
& \frac{1}{h} \int_0^h \int_M b_g(x, y; \tau) f(y; t+h - \tau) \mathrm d V_g(y) \mathrm d \tau, \\
=  & \frac{1}{h} \int_0^h \int_M b_g(x, y; \tau) f(y; t) \mathrm d V_g(y) \mathrm d \tau \\
&+ \frac{1}{h} \int_0^h \int_M b_g(x, y; \tau) \big( f(y; t+h - \tau) - f(y; t) \big) \mathrm d V_g(y) \mathrm d \tau, \\
=& f(x; t) + o(1),
\end{split}
\end{equation}
as $h \rightarrow 0^+$. This ends the proof of Theorem $\ref{thm701}$. 

\end{proof}

\subsection{Uniform estimate on nonhomogeneous equations when $T$ is small.} \label{sec3.4}

In this subsection, we'll prove Theorem $\ref{thm702}$. The idea of its proof is to compare the solution to the biharmonic heat equation on manifolds locally with the solution to the biharmonic heat equation on $\bR^n$ via appropriate cutoffs . 

First of all, let's study the properties of the biharmonic heat equation on $\bR^n$. On $\bR^n$, we can introduce Banach spaces $X_T^0, Y_Y^0$ of functions on $\bR^n \times (0,T)$ similar to spaces $X_T, Y_T$ of functions on $M \times (0,T)$ defined in the previous subsection(One just replaces $M$ with $\bR^n$ and uses the Euclidean metric instead of metric $g$ on $M$ in the definition of $X_T, Y_T$.). We have the following proposition.

\begin{prop}\label{prop703}
Given $f \in Y_T^0$, there exists a unique solution $u \in X_T^0$ to the biharmonic heat equation on $\bR^n$, namely $(\frac{\partial }{\partial t} + \Delta^2 ) u = f$ where $\Delta$ is the standard Laplacian on $\bR^n$. Moreover, there exists a dimensional constant $C_n >0$ such that $$\|u\|_{X_T^0} \leq C_n \|f\|_{Y_T^0}. $$
\end{prop}

\begin{proof}
The existence and uniqueness of solution $u \in X_T^0$ follows the same argument we did in last subsection. Actually, in this simpler case on $\bR^n$, one can derive an explicit formula of the biharmonic heat kernel with respect to the Euclidean metric directly using the Fourier transformation.  

We want to point out that the constant $C_n >0$ is independent of $T$. This is because of the scaling invariance property of the biharmonic heat equation on $\bR^n$. Given any $\lambda >0$, we define $u_\lambda $ and $f_\lambda$ as
\begin{equation}
u_\lambda (x; t) = \lambda^{-\frac{1}{2}} u(\lambda^{\frac{1}{4}} x ; \lambda t), f_{\lambda} (x; t) = \lambda^{\frac{1}{2}}f(\lambda^{\frac{1}{4}}x; \lambda t). 
\end{equation}
One can check easily that 
\begin{equation}
\|u_{\lambda}\|_{X_{T}^0} = \|u \|_{X_{\lambda T}^0}, \|f_{\lambda}\|_{Y_{T}^0} = \|f\|_{Y_{\lambda T}^0}. 
\end{equation}
And moreover,
\begin{equation}
( \frac{\partial  }{\partial t} + \Delta^2 ) u_\lambda = f_\lambda \Leftrightarrow ( \frac{\partial  }{\partial t} + \Delta^2)  u = f. 
\end{equation}
As a consequence, for any $T>0$, $$\|u\|_{X_T^0 } = \|u_T\|_{X_1^0} \leq C_n \|f_T\|_{Y_1^0} = C_n \|f\|_{Y_T^0}. $$
This ends the proof of Proposition $\ref{prop703}$. 
\end{proof}

Next using Proposition $\ref{prop703}$ we will prove Theorem $\ref{thm702}$. 
\begin{proof}
Given $Q >1$, at any $p \in M$, we will work within the geodesic ball $B_{r_1}(p, g)$ under its local harmonic coordinate $\{x_i\}_{i=1}^n$ where $r_1 > 0$ is the lower bound of $C^{3,\alpha}$-harmonic radius of $(M, g)$ by Lemma $\ref{lem202}$. Suppose function $\eta \in C_0^\infty\big(B_{r_1}(p, g)\big)$ with $\eta =1$ on $B_{\frac{1}{2}r_1} (p, g)$. Denote $u = V[f] \in X_T$ to be the unique solution in $X_T$ of nonhomogeneous problem $(\ref{eqn7001})$ with right hand side $f \in Y_T$ by Theorem $\ref{thm701}$. Then we consider function $v(x; t) := \eta(x) u(x; t)$ in local harmonic coordinates $\{x_j\}_{j=1}^n$. By direct computation, we have
\begin{equation}\label{eqn7008}
( \frac{\partial }{\partial t} + \Delta^2 ) v= \eta f + \Delta^2 (\eta u) - \eta \Delta_g^2 u, 
\end{equation}
where $\Delta$ is the standard Laplacian given by $\sum_{i=1}^n \frac{\partial^2 }{\partial x_i^2}$. Denote $\nabla$ and subscript ``$,i$ " by covariant derivatives w.r.t. $x_i$ under metric $g$. Denote $D$ and subscript ``$i$" with no comma by ordinary derivatives in local coordinates $\{x_i\}_{i=1}^n$. Thus we can rewrite the right side of $(\ref{eqn7008})$ as
\begin{equation}
\begin{split}
F : =  & \eta f + \eta ( \delta^{ij}\delta^{kl} u_{ijkl} - g^{ij} g^{kl} u_{, ijkl} ) + \sum_{k=0}^3 D^{4-k} \eta \star D^{k} u, \\
= & \eta f + \eta(\delta^{ij} \delta^{kl} - g^{ij} g^{kl}) u_{,ijkl} + ( D g + D\eta) \star \nabla^3 u \\
&+ (D^2 g + Dg \star Dg+D^2 \eta + D \eta \star D g) \star \nabla^2 u \\
&+ (D^3 g + D^2 g\star D g + Dg \star D g \star D g + D^3 \eta  + D^2 \eta \star D g \\
&+ D\eta \star D^2 g + D\eta\star D g\star D g) \star \nabla u + D^4 \eta \star u,   
\end{split}
\end{equation}
where ``$\star$" abbreviates for multiplications of smooth functions with compact support and also for proper contractions. We have the bounds for any $1 \leq k \leq 3$
\begin{equation}
|D^k \eta| \leq C r_1^{-k}, |D^4 \eta| \leq C r_1^{-4}, |D^k g| \leq (Q-1) r_1^{-k}, 
\end{equation}
and
\begin{equation}
[D^4 \eta]_{C^\alpha} \leq C r_1^{-4 - \alpha}, [D^3 g]_{C^{\alpha}} \leq (Q-1) r_1^{-3 -\alpha}. 
\end{equation}
By direct computations, we have that 
\begin{equation}
\begin{split}
 \|F(t)\|_{C^0(\bR^n)}  \leq &  C_Q \big[\|f(t)\|_{C^0(M)} + (Q-1) \|\nabla^4 u(t)\|_{C^0(M)}+  \sum_{k=0}^3 r_{1}^{-(4-k)}\|\nabla^{k} u(t)\|_{C^0(M)}\big], \\
[F(t)]_{C^{\alpha}(\bR^n)}  \leq & C_Q \big[ r_{1}^{-\alpha} \|f(t)\|_{C^0(M)} + [f(t)]_{C^\alpha(M)} + (Q-1) [\nabla^4 u(t)]_{C^\alpha(M)} + r_1^{-\alpha}\|\nabla^4 u(t)\|_{C^0(M)}\\
&+ \sum_{k=0}^3( r_{1}^{-(4-k)- \alpha}\|\nabla^{k} u(t)\|_{C^0(M)} + r_{1}^{-(4-k) }[\nabla^{k} u(t)]_{C^\alpha(M)} )\big],\\
\end{split}
\end{equation}
and 
\begin{equation}
\begin{split}
& \sup_{(x,t) \in \bR^n \times (0,T) }\sup_{0<h<T-t} t^{\frac{1}{2}+ \frac{\alpha}{4}} \frac{|F(x; t+h) -F(x; t)|}{|h|^{\frac{\alpha}{4}}}\\
\leq & C_Q \big[\|f\|_{Y_T} + (Q-1) \|u\|_{X_T} +  \sum_{k=0}^3 r_1^{-(4-k)} T^{\frac{4-k}{4}} \|u\|_{X_T}\big].
\end{split}
\end{equation}
Thus we have
\begin{equation}\label{eqn7010}
\begin{split}
\|F\|_{Y_T^0} = & \sup_{t \in (0,T)} \big(t^{\frac{1}{2}} \|F(t)\|_{C^0(\bR^n )} + t^{\frac{1}{2}+ \frac{\alpha}{4}} [F(t)]_{C^\alpha(\bR^n)} \big) \\
& +  \sup_{(x,t) \in \bR^n \times (0,T) }\sup_{0<h<T-t} t^{\frac{1}{2}+ \frac{\alpha}{4}} \frac{|F(x; t+h) -F(x; t)|}{|h|^{\frac{\alpha}{4}}} \\
\leq & C_Q \big[ \big(1+ (r_1^{-1}T^{\frac{1}{4}})^{\alpha} \big)\|f\|_{Y_T} + (Q-1)\|u\|_{X_T} +(r_1^{-1} T^{\frac{1}{4}})^\alpha \|u\|_{X_T} \\
& + \big(\sum_{k=0}^3 (r_1^{-1}T^{\frac{1}{4}})^{4- k+\alpha} + (r_1^{-1} T^{\frac{1}{4}})^{4-k} \big)\|u\|_{X_T} \big] ,
\end{split}
\end{equation}
for some constant $C_Q>0$ bounded whenever $Q$ is bounded. Thus $F \in Y_{T}^0$.

Next we examine $v =\eta u$ closely. We have $v \in X_T^0$ since by computation, 
\begin{equation}
\begin{split}
&D v =( D \eta )u + \eta \star \nabla u,\\
& D^2 v  = (D^2 \eta) u +( D \eta + D g) \star  \nabla u + \eta \nabla^2 u,\\
& D^3 v = (D^3 \eta) u + (D^2 \eta + D^2 g + D \eta  \star D g + D  g \star D g) \star \nabla u \\
& + (D \eta + D g) \star \nabla^2 u +  \eta \nabla^3 u, \\
& D^4 v = (D^4 \eta) u  + ( D^3 \eta + D^2 \eta \star  Dg + D \eta \star D^2 g + D \eta \star Dg \star D g + D^3g \\
&+ D^2 g \star D g+ D g \star D g \star D g) \star  \nabla u + (D^2 \eta + D \eta \star D g) \star \nabla^2 u \\
&+ D \eta \star \nabla^3 u + \eta \nabla^4 u, \\
& \frac{\partial }{\partial t} v = \eta \frac{\partial  }{\partial t} u. 
\end{split}
\end{equation}
and consequently we have for some constant $C_Q >0$ depending on $Q$, 
\begin{equation}
\|v\|_{X_T^0} \leq C_Q \big(\sum_{k=0}^4 (r_1^{-1} T^{\frac{1}{4}})^k +( r_1^{-1}T^{\frac{1}{4}}  )^{k+\alpha} \big)\|u\|_{X_T} < \infty. 
\end{equation}
Therefore by Proposition $\ref{prop703}$, we have 
\begin{equation}\label{eqn7101}
\|v\|_{X_T^0} \leq C_n \|F\|_{Y_T^0}. 
\end{equation}

On the other hand, by our construction, $ v = u$ on $B_{\frac{1}{2}r_1} (p, g)$ since $\eta =1$ on $B_{\frac{1}{2}r_1} (p, g)$. Thus we have
\begin{equation}
\|u\|_{X_T\big(B_{\frac{1}{2}r_1}(p, g)\big)} = \|v\|_{X_T\big(B_{\frac{1}{2}r_1}(p, g)\big)}, 
\end{equation}
where the norm $\|\cdot\|_{X_T(U)}$ above is defined by simply replacing $M$ with $U \subset M$ in the definition of $\|\cdot \|_{X_T}$ on $M$. Since
\begin{equation}
\begin{split}
&\nabla v = Dv, \\
& \nabla^2 v = D^2 v + D g \star D v, \\
& \nabla^3 v = D^3 v + D g \star D^2 v + (D^2 g + D g \star D g) \star D v, \\
& \nabla^4 v = D^4 v + Dg \star D^3 v + ( D^2 g + Dg \star D g) \star D^2 v \\
& + ( D^3 g + D^2 g \star D g +  Dg \star Dg \star D g ) \star D v, 
\end{split}
\end{equation} 
we have that 
\begin{equation}\label{eqn7011}
\|u\|_{X_T\big(B_{\frac{1}{2}r_1}(p, g)\big)} = \|v\|_{X_T\big(B_{\frac{1}{2}r_1} (p, g)\big)} \leq C_Q(\sum_{k=0}^3 (r_1^{-1} T^{\frac{1}{4}})^k + (r_1^{-1} T^{\frac{1}{4}})^{k+\alpha} ) \|v\|_{X_T^0}
\end{equation}

Combining estimates $(\ref{eqn7010})$, $(\ref{eqn7101})$, $(\ref{eqn7011})$ and using a finite covering argument, we have 
\begin{equation}
\|u\|_{X_T} \leq C_Q (1+ \epsilon ) \big[ \|f\|_{Y_T} + (Q-1) \|u\|_{X_T} + \epsilon \|u\|_{X_T}\big]. 
\end{equation}
where $C_Q$ is a constant depending on $n$ and $Q >1$, bounded whenever $Q $ is bounded and $\epsilon = \epsilon(r_1^{-1}T^{\frac{1}{4}}) \rightarrow 0 $ as $r_1^{-1}T^{\frac{1}{4}} \rightarrow 0$. Therefore, we could simply choose $Q >1$ sufficiently close to $1$ such that 
$$ 2 C_Q (Q-1) < \frac{1}{4}.$$
Note after chosen $Q>1$, $r_1>0$ will be determined as in Lemma $\ref{lem202}$. Then we choose $T>0$ sufficiently small accordingly such that
\begin{equation}
\begin{split}
& \epsilon(r_1^{-1}T^{\frac{1}{4}}) < 1, \\
&2 C_Q\epsilon(r_1^{-1} T^{\frac{1}{4}}) < \frac{1}{4}. 
\end{split}
\end{equation} 
Thus $\|u\|_{X_T} \leq 4 C_Q \|f\|_{Y_T}$ by absorbing the right hand side term involving $\|u\|_{X_T}$ to the left. In fact, it is clear that $Q $ chosen as above will be a dimensional constant so does $C_Q$. Consequently we have $T = T_g$ chosen as above depends on the $C^{3,\alpha}$-harmonic radius of $(M, g)$, which eventually depends on geometry quantities $\|\text{Ric}(g)\|_{C^{1,\alpha}(M)}$, injectivity radius $i_0$ and dimension $n$ by Lemma $\ref{lem202}$. This ends the proof of Theorem $\ref{thm702}$. 

\end{proof}

\section{Short time existence of the Calabi flow.}\label{sec4}

Suppose $(M, g, J)$ is a smooth closed K\"ahler manifold of complex dimension $n$. In this section, we want to study the short time existence of the Calabi flow with initial data $u_0$ such that 
\begin{equation}\label{eqn8003}
\partial \bar \partial u_0 \in L^\infty(M) \text{ and } \big\||\partial \bar \partial u_0|_g\big\|_{L^\infty(M)} \ll 1.
\end{equation}
We'll prove the following theorem.

\begin{thm}\label{thm802}
There exists a dimensional constant $\delta > 0$ such that for any $u_0$ with $\big\||\partial \bar \partial u_0|_g \big\|_{L^\infty(M)} < \delta$ there exists a constant $T_{g, u_0}>0$ such that the Calabi flow with initial value $u_0$ admits a short time solution $\varphi$ on $M \times (0,T_{g, u_0})$. By short time solution with initial data $u_0$, we mean that $\varphi$ is continuously differentiable up to fourth order w.r.t. $x\in M$ and first order w.r.t. $t \in (0, T_{g, u_0})$ such that $\varphi$ satisfies $(\ref{eqn8001})$ and $\lim_{t \rightarrow 0^+ } \|\varphi(t) - u_0\|_{C^{1,\gamma}(M)}  = 0$ for any $\gamma \in (0,1)$. 

Moreover, we have $\varphi = S_g u_0 + \psi$ for some $\psi \in X_{T_{g, u_0}}$ and
$$\limsup_{t \rightarrow 0^+} \|\psi\|_{X_t} \leq C_n \big\||\partial \bar \partial u_0|_g\big\|_{L^\infty(M)}, $$ 
for some dimensional constant $C_n > 0$. If we further assume $\partial \bar \partial u_0 \in C^0(M)$, then 
$$\lim_{t \rightarrow 0^+} \|\psi\|_{X_t} = 0, $$
and consequently
$$\lim_{t \rightarrow 0^+} \big(\|\varphi(t)- u_0\|_{C^1(M)} + \|\partial \bar \partial \varphi(t) - \partial \bar \partial u_0\|_{C^0(M)} \big)= 0.$$ 

\end{thm}

Notice that we can write Calabi flow as a biharmonic heat equation with right hand side involving up to fourth derivatives of $\varphi$. More precisely, we rewrite equation $(\ref{eqn8001})$ as
\begin{equation}\label{eqn8002}
\frac{\partial \varphi }{\partial t} + \Delta_g^2 \varphi = R(\varphi), 
\end{equation}
where 
\begin{equation}\label{eqn8005}
\begin{split}
R (\varphi)  : = & R_\varphi - \underline R + \Delta_g^2 \varphi \\
= & - (g_\varphi^{i\bar j} g_\varphi^{k\bar l} - g^{i\bar j} g^{k \bar l}) \varphi_{, i\bar j k\bar l} + g_\varphi^{i\bar j} g_{\varphi}^{k\bar p} g_\varphi^{q \bar l} \varphi_{,\bar p q i} \varphi_{, k\bar l \bar j} +  g_{\varphi}^{i\bar j}  ( \text{Ric}_g )_{i\bar j } - \underline R.
\end{split}
\end{equation}

Thus in the rest of the section, we will work on equation $(\ref{eqn8002})$. Given initial data $u_0$ with $\partial \bar \partial u_0 \in L^\infty(M)$, we shall seek its solution in the following form 
\begin{equation}
\varphi = \psi + S_g u_0, 
\end{equation}
for some function $\psi \in X_T$ satisfying $\limsup_{t \rightarrow 0^+} \|\psi\|_{X_t}  \leq C_n \big\||\partial \bar \partial u_0|_g\big\|_{L^\infty(M)}$ for some dimensional constant $C_n > 0$. Note that by Theorem $\ref{thm3.9}$, $\varphi$ defined above satisfies that $\lim_{t \rightarrow 0^+} \|\varphi(t) - u_0\|_{C^{1,\gamma}(M)}  = 0$ for any $\gamma \in (0,1)$.
Moreover, if for some $\psi \in X_T$, $\varphi = \psi+  S_g u_0 $ solves equation $(\ref{eqn8002})$, then equivalently $\psi$ satisfies 
\begin{equation}\label{eqn8004}
\frac{\partial \psi }{\partial t} + \Delta_g^2 \psi = R(\psi + S_g u_0 ). 
\end{equation}
We have the following lemma on $R(\psi + S_g u_0)$. 

\begin{lem}\label{lem802}

\begin{enumerate}

\item[i)] There exists a constant $\delta_1 = \delta_1( n)>0$ such that given any initial data $u_0$ satisfying $\partial \bar \partial u_0 \in L^\infty(M)$ and $\big\||\partial \bar \partial u_0|_g\big\|_{L^\infty (M)} < \delta_1$, there exists a constant $T_1>0$ depending on $g$ and $u_0$ such that for any $0< T \leq T_1$ and any $\xi \in X_T$ with $\|\xi \|_{X_T} < \delta_1$, we have $R(\xi + S_gu_0) \in Y_T$.

\item[ii)] Given any $\epsilon > 0$, there exists a constant $\delta_2 = \delta_2( \epsilon, n)>0$ such that given any initial data $u_0$ satisfying $\partial \bar \partial u_0 \in L^\infty(M)$ and $\big\||\partial \bar \partial u_0|_g\big\|_{L^\infty (M)} < \delta_2$, there exists a constant $T_2>0$ depending on $\epsilon$, $g$ and $u_0$ such that for any $0< T \leq T_2$ and any $\xi_1, \xi_2 \in \{\|\xi\|_{X_T} < \delta_2 \}$ then we have, 
\begin{equation}
 \|R(\xi_1 + S_g u_0)  - R (\xi_2 + S_g u_0)\|_{Y_T} 
 \leq \epsilon \|\xi_1 - \xi_2 \|_{X_T}.
\end{equation} 
\end{enumerate}

\end{lem}

We postpone the the proof of Lemma $\ref{lem802}$ for now and continue to describe solution $\psi$ of equation $(\ref{eqn8004})$. For any $\||\partial \bar \partial u_0|_g\|_{L^\infty(M)} < \delta_1$, $T < T_1$ and $\|\xi\|_{X_T} < \delta_1$ as in Lemma $\ref{lem802}$, we have $R(\xi + S_g u_0) \in Y_T$ and thus we can define $\Psi_{u_0} (\xi):= V[R(\xi +S_g u_0)] $, which is a well defined map from $\{\|\xi\|_{X_T} < \delta_1\}$ to $X_T$. If $\|\psi\|_{X_T} < \delta_1$, then by the uniqueness of solution to the nonhomogeneuous problem $(\ref{eqn8004})$, we have $\psi = \Psi_{u_0}(\psi)$. In other words, in order to find a solution to $(\ref{eqn8004})$, it suffices to seek a fixed point of the map $\Psi_{u_0}: \{\|\xi\|_{X_T} \leq \delta_1 \} \rightarrow X_T$. 

Denote $C_n $ and $T_g$ the constants introduced in Theorem $\ref{thm702}$. Choose $\epsilon = \frac{1}{2C_n}$ in Lemma $\ref{lem802}$ ii) and denote $\delta_3 =\min\{ \delta_{2}(\frac{1}{2C_n}, n), \delta_1\}$, and then we have for any $\||\partial \bar \partial u_0|_g\|_{L^\infty(M)}<\delta_3$, $T \leq \min\{ T_g , T_2(\frac{1}{2C_n}, u_0, g), T_1(u_0, g)\}$ and $\xi_1, \xi_2 \in \{\|\xi\|_{X_T} < \delta_3\}$, 
\begin{equation}
\begin{split}
\| \Psi_{u_0}(\xi_1) -\Psi_{u_0}(\xi_2) \|_{X_T} \leq  C_n  \|R(\xi_1 + S_g u_0)  - R (\xi_2 + S_g u_0)\|_{Y_T}  
\leq \frac{1}{2} \|\xi_1- \xi_2\|_{X_T}.
\end{split}
\end{equation}

To summarize, we got a dimensional constant $\delta_3 > 0$ such that for any $\partial \bar \partial u_0 \in L^\infty(M)$ with $\||\partial \bar \partial u_0|_g\|_{L^\infty} < \delta_3$, there exists a constant $T_3 = T_3(g,u_0) > 0$ such that for any $T \leq T_3$ the map $\Psi_{u_0}:  \{\|\xi\|_{X_T} \leq \delta_3 \} \rightarrow X_T$ is a contraction. Once we have the local contracting property of $\Psi_{u_0}$, to find its fixed point, it all suffices to construct an iterating sequence with first two terms lying in the contracting region. Consequently the whole sequence will lie in the contracting region as well and converge to the fixed point. We tend to choose $\psi_0 = 0$ and $\psi_1 = \Psi_{u_0}(\psi_0)$ as the first two terms of the iterating sequence $\psi_k = \Psi_{u_0}(\psi_{k-1})$ for $k \geq 1$. In order for the sequence $\{\psi_{k}\}$ to have the property described above, it suffices to show that $\|\psi_1\|_{X_T} < \frac{1}{2} \delta_3$ for some $T>0$. In fact we have the following result.

\begin{lem}\label{lem803}
For $\| |\partial \bar \partial u_0|_g\|_{L^\infty(M)} < \delta_1$ as in Lemma $\ref{lem802}$, we have
\begin{equation}
\limsup_{t\rightarrow 0^+}\|R(S_g u_0)\|_{Y_{t}} \leq C\| |\partial \bar \partial u_0|_g\|_{L^\infty(M)}, 
\end{equation}
for some dimensional constant $C>0$ and consequently we have
\begin{equation}
\limsup_{t \rightarrow 0^+}\|\psi_1\|_{X_{t}} \leq C C_n\| |\partial \bar \partial u_0|_g\|_{L^\infty(M)} .
\end{equation}
If we further assume $\partial \bar \partial u_0 \in C^0(M)$, then 
\begin{equation}
\lim_{t \rightarrow 0^+} \|R(S_g u_0)\|_{Y_t} = 0, 
\end{equation}
and
\begin{equation}
\lim_{t \rightarrow 0^+} \|\psi_1 \|_{X_t} = 0. 
\end{equation}
\end{lem}
 
Choose $\delta = \frac{\delta_3 }{4CC_n}$, a dimensional constant, then by Lemma $\ref{lem803}$, for any $\partial \bar \partial u_0  \in L^\infty(M) $ with $\||\partial \bar \partial u_0|_g\|_{L^\infty(M)} < \delta$, there exists a constant $T = T_{g, u_0} \ll T_3 $ such that $\|\psi_1\|_{X_T} < \frac{1}{2} \delta_3$. Thus, all the subsequent $\psi_k$'s will lie in the contracting region $\{\|\xi\|_{X_T} < \delta_3\}$ simply because by induction $\|\psi_k\|_{X_T} \leq \sum_{i=0}^{k-1} \|\psi_{i+1} - \psi_i\|_{X_T} \leq (\sum_{i=0}^{k-1} \frac{1}{2^k})\|\psi_{1}\|_{X_T} < \delta_3$. Consequently, $\psi_k \rightarrow \psi_\infty$ in $X_T$ as $k \rightarrow \infty$ where $\psi_\infty$ is the desired fixed point of $\Psi_{u_0}$. Moreover, by Lemma $\ref{lem803}$, one can easily show that $\limsup_{t\rightarrow 0^+} \|\psi_\infty \|_{X_t}  \leq 2C C_n \||\partial \bar \partial u_0|_g\|_{L^\infty(M)}$ and when $\partial \bar \partial u_0 \in C^0(M)$ $\lim_{t\rightarrow 0^+} \|\psi_\infty \|_{X_t} = 0$, since for any $k \geq 1$, $\|\psi_k \|_{X_t} \leq (\sum_{i=0}^{k-1} \frac{1}{2^k})\|\psi_{1}\|_{X_t} \leq 2 \|\psi_1\|_{X_t}$. 

By our discussion up to this point, in order to prove Theorem $\ref{thm802}$, we just need to show Lemma $\ref{lem802}$ and Lemma $\ref{lem803}$.

\begin{proof}
Proof of Lemma $\ref{lem802}$. Given $\partial \bar \partial u_0 \in L^\infty(M)$, by Theorem $\ref{thm3.9}$ we have as $t \rightarrow 0^+$, for any nonnegative integers $k, l$ 
\begin{equation}
\begin{split}
&  t^{\frac{k}{4} + l }\| |(\frac{\partial}{\partial t} )^l \nabla_g^k \partial \bar \partial S_g u_0(t)|_g\|_{C^0(M)} \leq  C_{l+ \frac{k}{4} , n} \||\partial \bar \partial u_0|_g\|_{L^\infty(M)} + o(1). \\
 \end{split}
 \end{equation}
 By interpolations, we have as $t \rightarrow 0^+$, for any $k, l \geq 0$, 
\begin{equation}
\begin{split}
& t^{\frac{k + \alpha}{4} + l } [ (\frac{\partial}{\partial t} )^l  \nabla_g^k \partial \bar \partial S_g u_0(t)]_{C^\alpha (M)} \leq C_{l+ \frac{k+\alpha}{4} , n} \||\partial \bar \partial u_0|_g\|_{L^\infty(M)} + o(1). \\
\end{split}
\end{equation}
Moreover, 
\begin{equation}
\begin{split}
 &\sup_{t \in (0,T)} t^{\frac{k+\alpha}{4} + l } \sup_{0< h < T-t}  h^{-\frac{\alpha}{4}} |(\frac{\partial}{\partial t} )^l \nabla_g^k \partial \bar \partial S_g u_0(x; t+h)- (\frac{\partial}{\partial t} )^l \nabla_g^k \partial \bar \partial S_g u_0(x; t)|_g \\
  \leq & C_{l+ \frac{k+\alpha}{4} , n} \||\partial \bar \partial u_0|_g\|_{L^\infty(M)} +o(1), 
 \end{split}
\end{equation}
as $T \rightarrow 0^+. $

Thus, there exists a dimensional constant $0 < \delta_1 \ll 1$ such that for any $\||\partial \bar\partial u_0|_g\|_{L^\infty(M)} < \delta_1$ there exists $T_1 = T_1(g, u_0)$ such that for any $t <T_1$, $0 \leq k \leq 2 $ and $ 0 < h <  T_1 - t$, 
\begin{equation}
\begin{split}
 &\| |\nabla_g^k \partial \bar \partial S_g u_0(t)|_g \|_{C^0(M)}  \leq \frac{1}{2} t^{-\frac{k}{4}}, [\nabla_g^k \partial \bar \partial S_g u_0(t)]_{C^\alpha(M)} \leq \frac{1}{2}   t^{-\frac{k}{4} - \frac{\alpha}{4}}, \\
& \sup_{x \in M} |\nabla_g^k \partial \bar \partial S_g u_0(x; t+h)- \nabla_g^k \partial \bar \partial S_g u_0(x; t)|_g   \leq  \frac{1}{2} h^{\frac{\alpha}{4}}  t^{-\frac{k}{4} - \frac{\alpha}{4}}.
\end{split}
\end{equation}
 
Furthermore, for any $T \leq T_1$ and any $\|\xi\|_{X_{T}} < \delta_1 $, we have $\varphi = \xi + S_g u_0$ satisfying for $t < T$, $0\leq k \leq 2$ and $0 < h < T - t$, 
\begin{equation}
\begin{split}
& \| |\nabla_g^k \partial \bar \partial \varphi (t) |_g \|_{C^0(M)} \leq \frac{3}{4} t^{-\frac{k}{4}}, [\nabla_g^k \partial \bar \partial \varphi (t)]_{C^\alpha(M)} \leq \frac{3}{4}   t^{-\frac{k}{4} - \frac{\alpha}{4}}, \\
& \sup_{x \in M} |\nabla_g^k \partial \bar \partial \varphi (x; t+h)- \nabla_g^k \partial \bar \partial \varphi (x; t)|_g \leq  \frac{3}{4} h^{\frac{\alpha}{4}}  t^{-\frac{k}{4} - \frac{\alpha}{4}}.
\end{split}
\end{equation}

With these preparations on $\varphi$, by formula $(\ref{eqn8005})$ we have that for some dimensional constant $C>0$
\begin{equation}
\begin{split}
\|R(\varphi) (t)\|_{C^0(M)} \leq&  C \big( \| |\nabla_g^2 \partial \bar \partial  \varphi(t)|_g\|_{C^0(M)} + \||\nabla_g \partial \bar \partial \varphi(t)|_g\|^2_{C^0(M)} \\
&+ \|\text{Ric}_g\|_{C^0(M)} + |\underline R| \big)\\
\leq & C t^{-\frac{1}{2}} + C \|\text{Ric}_g\|_{C^0(M)}, 
\end{split}
\end{equation}
and
\begin{equation}
\begin{split}
[R(\varphi) (t)]_{C^\alpha(M)} \leq & C \big( [\nabla_g^2\partial \bar \partial  \varphi(t)]_{C^\alpha(M)} + ([\omega]_{C^\alpha(M)} + [\partial \bar \partial \varphi(t)]_{C^\alpha(M)}) \||\nabla_g^2 \partial \bar \partial \varphi(t)|_g\|_{C^0(M)} \\
&+ [\nabla_g \partial \bar \partial \varphi(t)]_{C^\alpha(M)} \||\nabla_g \partial \bar \partial \varphi(t)|_g\|_{C^0(M)} 
+ [\omega_{\varphi(t)}]_{C^\alpha(M)} \||\nabla_g \partial \bar \partial \varphi(t)|_g\|_{C^0(M)}^2 \\
&+ [\omega]_{C^\alpha(M)} \|\text{Ric}_g\|_{C^0(M)} + [\text{Ric}_g]_{C^\alpha(M)} 
+ \| \text{Ric}_g\|_{C^0(M)} [\partial \bar \partial \varphi(t)]_{C^\alpha(M)}  \big)\\
\leq & C t^{- \frac{1}{2}- \frac{\alpha}{4}}  + C t^{-\frac{1}{2}} [\omega]_{C^\alpha(M)}+C [\omega]_{C^\alpha(M)} \|\text{Ric}_g\|_{C^0(M)} + C [\text{Ric}_g]_{C^\alpha(M)}\\
&  + C t^{-\frac{\alpha}{4}} \|\text{Ric}_g\|_{C^0(M)}.
\end{split}
\end{equation}

Also we have for $0< h < T-t$
\begin{equation}
\begin{split}
& |R(\varphi)(x;t+h) - R(\varphi)(x; t)|\\
 \leq & C \big(|\nabla_g^2 \partial \bar \partial \varphi(x; t+h) - \nabla_g^2 \partial \bar \partial \varphi(x; t)|_g + |\partial \bar \partial \varphi(x; t+h) -\partial \bar \partial \varphi(x; t)  |_g |\nabla_g^2 \partial \bar \partial \varphi(x; t) |_g \\
 & + |\nabla_g \partial \bar \partial \varphi(x; t+h) - \nabla_g \partial \bar \partial \varphi(x; t)  |_g |\nabla_g \partial \bar \partial \varphi(x; t)|_g + |\nabla_g \partial \bar \partial \varphi(x; t+h) - \nabla_g \partial \bar \partial \varphi(x; t)  |_g |\nabla_g \partial \bar \partial \varphi(x; t+h)|_g \\
 & +  |\partial \bar \partial \varphi(x; t+h) -\partial \bar \partial \varphi(x; t)  |_g |\nabla_g \partial \bar \partial \varphi(x; t)|_g^2 + |\partial \bar \partial \varphi(x; t+h) -\partial \bar \partial \varphi(x; t)  |_g \|\text{Ric}_g\|_{C^0(M)} \big),\\
 \leq & C h^{\frac{\alpha}{4}} \big( t^{-\frac{1}{2} -\frac{\alpha}{4}} + t^{-\frac{\alpha}{4}} \|\text{Ric}_g\|_{C^0(M)}\big). 
\end{split}
\end{equation}
Therefore combing computations above, we have for $T \leq T_1$, 
\begin{equation}\label{eqn8010}
\begin{split}
\|R(\varphi)\|_{Y_T} = & \sup_{t\in (0,T)} ( t^{\frac{1}{2}} \|R(\varphi)(t)\|_{C^0(M)} + t^{\frac{1}{2} +\frac{\alpha}{4}} [R(\varphi)(t)]_{C^\alpha(M)}) \\
& + \sup_{(x, t) \in M \times (0,T)} \sup_{0<h<T-t} t^{\frac{1}{2} + \frac{\alpha}{4}} \frac{|R(\varphi)(x;t+h) - R(\varphi)(x; t)|}{h^\frac{\alpha}{4}}\\
\leq & C + C T^{\frac{\alpha}{4}} [\omega]_{C^\alpha(M)} + C T^{\frac{1}{2}}\|\text{Ric}_g\|_{C^0(M)} + C T^{\frac{1}{2}+ \frac{\alpha}{4}} [\omega]_{C^\alpha(M) } \|\text{Ric}_g\|_{C^0(M)}\\
&  + C T^{\frac{1}{2}+ \frac{\alpha}{4}} [\text{Ric}_g]_{C^\alpha(M)} < \infty.
\end{split}
\end{equation}

Therefore $R(\varphi) \in Y_T$ for $u_0, T$ and $\xi \in X_T$ satisfying assumptions described above. 

Following the discussion at the beginning of proof of Lemma $\ref{lem802}$, given any $0< \eta \ll 1$ such that there exists a constant $\delta_2 = \delta_2(\eta, n)> 0$ such that for any $\||\partial \bar \partial u_0|_g\|_{L^\infty(M)} < \delta_2$ there exists constant $T_2 = T_2(u_0, g)$ such that for any $T \leq T_2$ and any $\|\xi\|_{X_T} < \delta_2 $, we have $\varphi= \xi + S_g u_0$ satisfying for any $t < T$, $0 \leq k \leq 2$ and $0< h < T -t$, 
\begin{equation}\label{eqn7200}
\begin{split}
& \|\nabla_g^k \partial \bar \partial \varphi (t) \|_{C^0(M)} \leq \eta t^{-\frac{k}{4}}, [\nabla_g^k \partial \bar \partial \varphi (t)]_{C^\alpha(M)} \leq \eta   t^{-\frac{k}{4} - \frac{\alpha}{4}}, \\
& \sup_{x \in M} |\nabla_g^k \partial \bar \partial \varphi (x; t+h)- \nabla_g^k \partial \bar \partial \varphi (x; t)|_g \leq \eta h^{\frac{\alpha}{4}}  t^{-\frac{k}{4} - \frac{\alpha}{4}}.
\end{split}
\end{equation}
We'll specify $\eta> 0$ later. 

Suppose $T \leq T_2$ and $\| \xi_i \|_{X_T} < \delta_2 $($i=1,2$) and we denote $\varphi_i = \xi_i + S_g u_0$($i=1,2$). Thus $R(\varphi_i) \in Y_T$($i=1,2$) and we compute their difference as below
\begin{equation}\label{eqn7201}
\begin{split}
 & R(\varphi_1) - R(\varphi_2) \\
 = &- (g_{\varphi_1}^{i\bar j} g_{\varphi_1}^{k\bar l} - g^{i\bar j} g^{k \bar l}) (\xi_1 -\xi_2)_{, i\bar j k\bar l} - \varphi_{2, i\bar j k\bar l} (g_{\varphi_1}^{i\bar j} g_{\varphi_1}^{k\bar l} - g_{\varphi_2}^{i\bar j} g_{\varphi_2}^{k \bar l})  \\
& + g_{\varphi_1}^{i\bar j} g_{\varphi_1}^{k\bar p} g_{\varphi_1}^{q \bar l}  \varphi_{1,\bar p q i}( \xi_{1} -   \xi_{2})_{, k\bar l \bar j} + g_{\varphi_1}^{i\bar j} g_{\varphi_1}^{k\bar p} g_{\varphi_1}^{q \bar l}  \varphi_{2,\bar p q i}( \xi_{1} -   \xi_{2})_{, k\bar l \bar j}   \\
& + \varphi_{2,\bar p q i} \varphi_{2, k\bar l \bar j} (g_{\varphi_1}^{i\bar j} g_{\varphi_1}^{k\bar p} g_{\varphi_1}^{q \bar l} - g_{\varphi_2}^{i\bar j} g_{\varphi_2}^{k\bar p} g_{\varphi_2}^{q \bar l})+  ( \text{Ric}_g )_{i\bar j } (g_{\varphi_1}^{i\bar j}  -g_{\varphi_2}^{i\bar j} ) .
\end{split}
\end{equation} 
In the expression above, we arrange each term as a product of functions and the derivatives of $(\xi_1 - \xi_2)$, where we use the formula below to understand terms which are not written explicitly as derivatives of $(\xi_1- \xi_2)$, 
\begin{equation}
g^{-1}_{\varphi_1} - g^{-1}_{\varphi_2} = - \sqrt{-1} \int_0^1 g_{\varphi_\tau}^{-1} \big( \partial \bar \partial (\xi_1 -\xi_2)\big)g_{\varphi_\tau}^{-1} \mathrm d \tau, 
\end{equation}
for $\varphi_\tau = \tau \varphi_1 +(1 -\tau) \varphi_2$. Note $\varphi_\tau$ for all $\tau \in [0,1]$ satisfy $(\ref{eqn7200})$. Therefore we have for some dimensional constant $C>0$
\begin{equation}\label{eqn7300}
\begin{split}
 \|\big( g_{\varphi_1}^{-1} - g_{\varphi_2}^{-1}\big) (t)\|_{C^0(M)} \leq  & C \| \partial \bar \partial (\xi_1 - \xi_2) (t)\|_{C^0(M)} \\
 \leq & C \|\xi_1 -\xi_2\|_{X_T},\\
[\big( g_{\varphi_1}^{-1} - g_{\varphi_2}^{-1}\big) (t)]_{C^\alpha(M)} \leq & C [\partial \bar \partial (\xi_1 - \xi_2)(t)]_{C^\alpha(M)} \\
&+ C\| \partial \bar \partial (\xi_1 - \xi_2)(t)\|_{C^0(M)} ([\omega]_{C^\alpha(M)} + t^{-\frac{\alpha}{4}}), \\
\leq & C( t^{-\frac{\alpha}{4}} + [\omega]_{C^\alpha(M)} )\|\xi_1 - \xi_2 \|_{X_T}
\end{split}
\end{equation}
and for $0< h < T-t$, 
\begin{equation}\label{eqn7301}
\begin{split}
 &  \sup_{x\in M} | \big( g_{\varphi_1}^{-1} - g_{\varphi_2}^{-1}\big)(x; t+h) - \big( g_{\varphi_1}^{-1} - g_{\varphi_2}^{-1}\big)(x; t) |_g\\
 \leq & C \sup_{x\in M} | \partial \bar \partial ( \xi_1 - \xi_2 )(x; t+h) - \partial \bar \partial ( \xi_1 - \xi_2 )(x; t) |_g\\
 & + C h^{\frac{\alpha}{4}} t^{-\frac{\alpha}{4}} \|\partial \bar \partial ( \xi_1 - \xi_2 )(t)\|_{C^0(M)} , \\
 \leq & C h^{\frac{\alpha}{4}}t^{-\frac{\alpha}{4}} \|\xi_1 - \xi_2\|_{X_T}
\end{split}
\end{equation}
By induction, one can show estimates $(\ref{eqn7300})$ and $(\ref{eqn7301})$ holds for any $(g_{\varphi_1}^{-1})^{\otimes k} - (g_{\varphi_2}^{-1})^{\otimes k}$ for integer $k \geq 1$. 

By a similar computation on $\big((g_{\varphi_1}^{-1})^{\otimes 2} - (g^{-1})^{\otimes 2}\big)$, we have
\begin{equation}
\begin{split}
 \|\big((g_{\varphi_1}^{-1})^{\otimes 2} - (g^{-1})^{\otimes 2}\big) (t)\|_{C^0(M)} \leq  & C \eta,\\
[\big((g_{\varphi_1}^{-1})^{\otimes 2} - (g^{-1})^{\otimes 2}\big) (t)]_{C^\alpha(M)} \leq & C \eta (t^{-\frac{\alpha}{4}} + [\omega]_{C^\alpha(M)}). \\
\end{split}
\end{equation}
Also, we have for $0< h < T-t$, $i =1,2$ and $k \geq 1$
\begin{equation}
\begin{split}
\sup_{x\in M} | \big( (g_{\varphi_i}^{-1})^{\otimes k}\big)(x; t+h) - \big( (g_{\varphi_i}^{-1})^{\otimes k} \big)(x; t) |_g\leq  C \eta  h^{\frac{\alpha}{4}} t^{-\frac{\alpha}{4}}  . 
\end{split}
\end{equation}

With these preparations, by formula $(\ref{eqn7201})$ we have
\begin{equation}
\begin{split}
& \|\big( R(\varphi_1)- R(\varphi_2) \big) (t)\|_{C^0(M)}\\
 \leq & C \eta  \big( \|\nabla_g^4 (\xi_1 - \xi_2)(t)\|_{C^0(M)} +  t^{-\frac{1}{2}} \|\nabla_g^2 (\xi_1 - \xi_2)(t)\|_{C^0(M)} + t^{-\frac{1}{4}} \|\nabla_g^3(\xi_1 - \xi_2)(t)\|_{C^0(M)} \big) \\
 &+ C \|\text{Ric}_g\|_{C^0(M)} \|\nabla_g^2 (\xi_1 - \xi_2)(t)\|_{C^0(M)}, \\
 \leq &  C \eta t^{-\frac{1}{2}} \|\xi_1 - \xi_2 \|_{X_T} + C \|\text{Ric}_g\|_{C^0(M)} \| \xi_1 - \xi_2\|_{X_T}, 
\end{split}
\end{equation}
and 
\begin{equation}
\begin{split}
& [\big( R(\varphi_1)- R(\varphi_2) \big) (t)]_{C^\alpha(M)} \\
\leq & C \eta\big([\nabla_g^4(\xi_1 - \xi_2)(t)]_{C^\alpha(M)}+ ( [\omega]_{C^\alpha(M)} + t^{-\frac{\alpha}{4}} ) \|\nabla_g^4(\xi_1 - \xi_2)(t)\|_{C^0(M)}\\
&+ t^{-\frac{1}{2}}( t^{ - \frac{\alpha}{4}} + [\omega]_{C^\alpha(M)} )\| \xi_1 - \xi_2\|_{X_T}  + t^{-\frac{1}{4}} [\nabla_g^3 (\xi_1 - \xi_2)(t)]_{C^\alpha (M)} \\
&+ t^{-\frac{1}{4}}( [\omega]_{C^\alpha(M)} + t^{-\frac{\alpha}{4}} ) \|\nabla_g^3(\xi_1- \xi_2)(t)\|_{C^0(M)} + t^{-\frac{1}{2}} ( [\omega]_{C^\alpha(M)} + t^{-\frac{\alpha}{4}}) \| (\xi_1 - \xi_2)\|_{X_T}   \big) \\
& + C \big( ([\omega]_{C^\alpha(M)} + t^{-\frac{\alpha}{4}}) \|\text{Ric}_g\|_{C^0(M)} \| \xi_1 -\xi_2\|_{X_T}  + [\text{Ric}_g]_{C^\alpha(M)} \| \xi_1 -\xi_2\|_{X_T} \big),\\
\leq & C \eta t^{-\frac{1}{2}} \big(t^{ -\frac{\alpha}{4}} + [\omega]_{C^\alpha(M)}  \big) \|\xi_1 - \xi_2 \|_{X_T} \\
& + C \big(t^{-\frac{\alpha}{4}} \|\text{Ric}_g\|_{C^0(M)} + [\omega]_{C^\alpha(M)} \|\text{Ric}_g\|_{C^0(M)} + [\text{Ric}_g]_{C^\alpha(M)}\big) \|\xi_1 - \xi_2\|_{X_T}. 
\end{split}
\end{equation}

Also for $0 < h < T-t$, we have
\begin{equation}
\begin{split}
& |\big( R(\varphi_1)- R(\varphi_2) \big) (x; t+h) - \big( R(\varphi_1)- R(\varphi_2) \big)(x; t) | \\
\leq & C \eta \big( |\nabla_g^4 (\xi_1- \xi_2)(x;  t+h) - \nabla_g^4 (\xi_1- \xi_2)(x;  t)|_g + h^{\frac{\alpha}{4}} t^{-\frac{\alpha}{4}} \|\nabla_g^4 (\xi_1 - \xi_2)(t)\|_{C^0(M)} \\
&+ h^{\frac{\alpha}{4}} t^{-\frac{1}{2}-\frac{\alpha}{4}} \|\xi_1 - \xi_2\|_{X_T} + t^{-\frac{1}{4}} |\nabla_g^3 (\xi_1- \xi_2)(x;  t+h) - \nabla_g^3 (\xi_1- \xi_2)(x;  t)|_g \\
& +h^{\frac{\alpha}{4}} t^{-\frac{1}{4}-\frac{\alpha}{4}} \|\nabla_g^3 (\xi_1 - \xi_2)(t)\|_{C^0(M)} + h^{\frac{\alpha}{4}} t^{-\frac{1}{2}-\frac{\alpha}{4}} \|\xi_1 - \xi_2\|_{X_T} \big)\\
&+C \big( h^{\frac{\alpha}{4}} t^{-\frac{\alpha}{4}}\|\text{Ric}_g \|_{C^0(M)} \|\xi_1 - \xi_2\|_{X_T} \big), \\
\leq &C  h^{\frac{\alpha}{4}} \big(\eta  t^{-\frac{1}{2} - \frac{\alpha}{4}}+  \|\text{Ric}_g\|_{C^0(M)} t^{-\frac{\alpha}{4}} \big) \|\xi_1 - \xi_2\|_{X_T}.
\end{split}
\end{equation}

Thus combing computations above, we have for some dimensional constant $C>0$
\begin{equation}\label{eqn8200}
\begin{split}
& \|R(\varphi_1) - R(\varphi_2)\|_{Y_T}\\
 \leq & C \eta \|\xi_1 - \xi_2\|_{X_T}  + C  \big(T^{\frac{\alpha}{4}} [\omega]_{C^\alpha(M)}+ T^{\frac{1}{2}} \|\text{Ric}_g\|_{C^0(M)} + T^{\frac{1}{2}+\frac{\alpha}{4}} [\text{Ric}_g]_{C^\alpha(M)}\\
 & +T^{\frac{1}{2}+\frac{\alpha}{4}} [\omega]_{C^\alpha(M)} \|\text{Ric}_g\|_{C^0(M)}  \big)\|\xi_1- \xi_2\|_{X_T}. 
\end{split}
\end{equation}

In the formula above, it's clear that to make the constant in front of $\|\xi_1 - \xi_2\|_{X_T}$ smaller than given $\epsilon$, one can choose $\eta = \frac{\epsilon}{2C} $ in the beginning. Then the constants $\delta_2 = \delta_2(\epsilon, n)> 0$ will be determined accordingly. Also we choose $T_2$ even smaller such that 
$$T_2^{\frac{\alpha}{4}} [\omega]_{C^\alpha(M)}+ T_2^{\frac{1}{2}} \|\text{Ric}_g\|_{C^0(M)} + T_2^{\frac{1}{2}+\frac{\alpha}{4}} [\text{Ric}_g]_{C^\alpha(M)} +T_2^{\frac{1}{2}+\frac{\alpha}{4}} [\omega]_{C^\alpha(M)} \|\text{Ric}_g\|_{C^0(M)}< \frac{\epsilon}{2}.$$ 
Thus it ends the proof of Lemma $\ref{lem802}$.

\end{proof}

\begin{proof}
Proof of Lemma $\ref{lem803}$. By virtue of the computation $(\ref{eqn8010})$ for $\|R(\varphi)\|_{Y_T}$ above, we plug in $\varphi = S_g u_0$. By Theorem $\ref{thm3.9}$, one can check that 
$$\limsup_{t\rightarrow 0^+}\|R(S_g u_0)\|_{Y_{t}} \leq C\| |\partial \bar \partial u_0|_g\|_{L^\infty(M)}, $$
for some dimensional constant $C> 0$ and if $\partial \bar \partial u_0 \in C^0(M)$ then 
$$\lim_{t \rightarrow 0^+} \|R(S_g u_0)\|_{Y_t} = 0.$$ 
The estimate on $\psi_1 = V [R(S_g u_0)]$ follows from Theorem $\ref{thm702}$. Thus it ends the proof of the lemma.

\end{proof}

\section{Smoothness and uniqueness  of the Calabi flow when $t>0$.}\label{sec5}

In this section, we will show that the short time solution of the Calabi flow obtained in last section is smooth whenever $t > 0$. Given the short time solution $\varphi$ in Theorem $\ref{thm802}$ of Calabi flow, notice that $\varphi$ restricted to any subinterval $[\epsilon, T]$($\epsilon>0$) belongs to the standard parabolic H\"older space. Moreover, restricting to $[\epsilon, T]$, the leading term of scalar curvature function $R_\varphi$, namely $g_{\varphi}^{i\bar j} g_{\varphi}^{k\bar l} \nabla_{i\bar j k\bar l} \varphi$, now becomes an elliptic operator on $\varphi$ with uniformly H\"older continuous coefficients $g_{\varphi}^{i\bar j} g_{\varphi}^{k\bar l} $. Thus by rewriting Calabi flow equation as
\begin{equation}
(\frac{\partial }{\partial t} + g_{\varphi}^{i\bar j} g_{\varphi}^{k\bar l} \nabla_{i\bar j k\bar l} ) \varphi = \text{ lower order terms of } \varphi, 
\end{equation}
we can use the standard parabolic Schauder estimate to improve the regularity of $\varphi $. More precisely, we will prove the following theorem.

\begin{thm}\label{thmsmooth}
The short time solution of the Calabi flow with initial data $u_0$ obtained in Theorem $\ref{thm802}$, denoted as function $\varphi$ on $M \times (0,T)$, is smooth whenever $t>0$. Namely we have $\varphi \in C^{\infty}\big(M \times (0,T) \big)$. If $\partial \bar \partial u_0 \in L^\infty(M)$ then we have for any $k, l \geq 0$, 
\begin{equation}\label{eqn9002}
\limsup_{t \rightarrow 0^+} t^{l+ \frac{k}{4}}\| |(\frac{\partial }{\partial t})^l \nabla_g^k \partial \bar \partial \varphi(t)|_g\|_{C^0(M)} \leq C_{l + \frac{k}{4}, n} \||\partial \bar \partial u_0|_g\|_{L^\infty(M)} ,
\end{equation}
for some constant $C_{l + \frac{k}{4}, n} >0 $ depending on $l +\frac{k}{4}$ and dimensional $n$ only. If $\partial \bar \partial u_0 \in C^0(M)$,  we have for any $k, l \geq 0$ with $l + k > 0$ 
\begin{equation}\label{eqn9003}
\lim_{t \rightarrow 0^+} t^{l+ \frac{k}{4}}\||(\frac{\partial }{\partial t})^l \nabla^k \partial \bar \partial \varphi(t)|_g\|_{C^0(M)} = 0. 
\end{equation}
\end{thm}

\begin{proof}

By our discussion above, the short time solution $\varphi \in C^\infty \big(M \times (0,T) \big)$. Next we focus on proving estimates $(\ref{eqn9002})$ and $(\ref{eqn9003})$. 
 
Given any $z \in M$, we work under local holomorphic coordinate in a neighborhood of $z$, say $B_{r_0}(0) \cong U \subset M$, such that $\frac{1}{2} \delta_{i\bar j} \leq g_{i\bar j} \leq 2 \delta_{i\bar j}$. Given $\lambda \ll 1$, consider the following parabolic rescaling,

\begin{equation}
g_{\lambda} (z) = g(\lambda^{\frac{1}{4}}z), 
\varphi_\lambda (z, t) = \lambda^{-\frac{1}{2}} \varphi(\lambda^{\frac{1}{4}} z, \lambda t).
\end{equation}

Thus for $\varphi$ restricted to $B_{\lambda^{\frac{1}{4}}}(0) \times [\frac{\lambda}{2}, \lambda]$, we have $\varphi_\lambda \in C^\infty(B_{1}(0) \times [\frac{1}{2}, 1])$ satisfying
\begin{equation}
\frac{\partial }{\partial t} \varphi_\lambda = R_{g_{ \varphi_\lambda}} - \underline R, 
\end{equation}
where we denote $g_{\varphi_\lambda} : = g_\lambda + \sqrt{-1} \partial \bar \partial \varphi_\lambda$.

By Theorem $\ref{thm802}$, we have $\varphi = S_g u_0 + \psi$. Denote $\psi_\lambda$ and $(S_g u_0)_\lambda$ by corresponding smooth functions after rescaling, then we can rewrite the calabi flow equation as the following linear equation on $\psi_\lambda \in C^\infty(B_1(0) \times [\frac{1}{2},1])$
\begin{equation}
\frac{\partial}{\partial t} \psi_\lambda + g_{\varphi_\lambda}^{i \bar j} g_{\varphi_\lambda}^{k\bar l} \psi_{\lambda, i\bar j k \bar l }= f, 
\end{equation}
where 
\begin{equation}
f = g_{\varphi_\lambda}^{i\bar j} g_{\varphi_\lambda}^{k\bar p} g_{\varphi_\lambda}^{q \bar l} \varphi_{\lambda,\bar p q i} \varphi_{\lambda, k\bar l \bar j} +  g_{\varphi_\lambda}^{i\bar j}  ( \text{Ric}_{g_\lambda} )_{i\bar j } - \underline R + (\frac{\partial}{\partial t}  + g_{\varphi_\lambda}^{i \bar j} g_{\varphi_\lambda}^{k\bar l} ) (S_g u_0)_{\lambda, i\bar j k \bar l }
\end{equation}
Taking $\nabla_{g_\lambda}$ on the equation above, we have
\begin{equation}
\begin{split}
\frac{\partial}{\partial t} \nabla_{g_\lambda} \psi_{\lambda} + g_{\varphi_\lambda}^{i \bar j} g_{\varphi_\lambda}^{k\bar l} (\nabla_{g_\lambda} \psi_{\lambda})_{, i\bar j k \bar l } = &\nabla_{g_\lambda} f - (\nabla_{g_\lambda}g_{\varphi_\lambda}^{i \bar j}) g_{\varphi_\lambda}^{k\bar l} \psi_{\lambda, i\bar j k \bar l } - g_{\varphi_\lambda}^{i \bar j} (\nabla_{g_\lambda} g_{\varphi_\lambda}^{k\bar l} )\psi_{\lambda, i\bar j k \bar l } \\
&+ \sum_{k=1}^4 g_{\varphi_\lambda}^2 \star \nabla_{g_\lambda}^k \psi_\lambda.
\end{split}
\end{equation}

Denote the right hand side above by $F_\lambda$, by Theorem $\ref{thm3.9}$ and Theorem $\ref{thm802}$ one can check that when $\partial \bar \partial u_0 \in L^\infty(M)$
$$\|F_\lambda\|_{C^{\alpha, \frac{\alpha}{4}}\big(B_1(0) \times [\frac{1}{2},1]\big)} \leq C \||\partial \bar \partial u_0|_g\|_{L^\infty(M)} + o(1), $$
for some dimensional constant $C>0$ and when $\partial \bar \partial u_0 \in C^0(M)$,
$$\|F_\lambda\|_{C^{\alpha, \frac{\alpha}{4}}\big(B_1(0) \times [\frac{1}{2},1]\big)} =  o(1), $$
as $\lambda \rightarrow 0$. Also the elliptic coefficients $g_{\varphi_\lambda}^{i \bar j} g_{\varphi_\lambda}^{k\bar l} \in C^{\alpha, \frac{\alpha}{4}}(\big(B_1(0) \times [\frac{1}{2},1]\big))$ and as $\lambda \rightarrow 0$, we have
$$\|g_{\varphi_\lambda}^{i \bar j} g_{\varphi_\lambda}^{k\bar l} - \delta^{i\bar j} \delta^{k\bar l}\|_{C^{\alpha, \frac{\alpha}{4}}\big(B_1(0) \times [\frac{1}{2},1]\big)} \leq C \||\partial \bar \partial u_0|_g\|_{L^\infty(M)} + o(1) \ll 1. $$ 
Thus by the standard interior Schauder estimate, we have
\begin{equation}
\begin{split}
& \|\nabla_{g_\lambda} \psi_\lambda\|_{C^{4+\alpha, 1+\frac{\alpha}{4}}\big(B_{\frac{1}{2}}(0) \times [\frac{1}{4},1]\big)}\\
\leq & C \|F_\lambda\|_{C^{\alpha, \frac{\alpha}{4}} \big(B_1(0) \times [\frac{1}{2},1]\big)} + C \|\nabla_{g_\lambda}\psi_\lambda\|_{L^\infty\big(B_1(0) \times [\frac{1}{2},1]\big)}\\ 
\leq & C \|F_\lambda\|_{C^{\alpha, \frac{\alpha}{4}} \big(B_1(0) \times [\frac{1}{2},1]\big)} + C \|\psi\|_{X_\lambda}
\end{split}
\end{equation} 

Scale everything back and by a covering argument we have for $\partial \bar \partial u_0 \in L^\infty (M)$
\begin{equation}
\limsup_{t \rightarrow 0^+}t^{\frac{3}{4}}\| |\nabla_g^5 \psi (t) |_g\|_{C^0(M)}  \leq C \||\partial \bar \partial u_0|_g\|_{L^\infty(M)}, 
\end{equation}
and for $\partial \bar \partial u_0  \in C^0(M)$, we have
\begin{equation}
\lim_{t \rightarrow 0^+}t^{\frac{3}{4}}\| |\nabla_g^5 \psi (t) |_g\|_{C^0(M)}  = 0, 
\end{equation}

Together with bounds on $S_g u_0$ in Theorem $\ref{thm3.9}$, we proved the desired estimate for $\nabla_g^3 \partial \bar \partial \varphi(t)$ for $t $ sufficiently small. By induction, one can show $(\ref{eqn9002})$ and $(\ref{eqn9003})$.

\end{proof}

We mention the uniqueness of the Calabi flow. In general the nonlinear parabolic flow might not have uniqueness for rough initial data. But the Calabi flow enjoys a remarkable property, proved by Calabi-Chen, that it shrinks the distance of any (smooth) curves. This property will directly imply the uniqueness of the Calabi flow.  Such a uniqueness also holds for the weak Calabi flow in the sense of Streets. 

Now we give a proof of Theorem \ref{kminimizer}. In particular we need a notion of the weak Calabi flow, studied by J. Streets \cite{streets}. Streets' method is based on the convexity of $K$-energy \cite{BB} (see \cite{CLP} also) and some general results of Mayer \cite{Mayer} on the gradient flow on metric spaces with nonpositive curvature in the sense of Alexanderov. 

\begin{proof}[Proof of Theorem \ref{kminimizer}]
Streets proved \cite{streets, streets1}, among others,  that the weak Calabi flow coincides with the smooth one, if the later exists. His proof can be directly carried over to our solutions of the Calabi flow with initial metric in $L^\infty$. A general result by Mayer \cite{Mayer} implies that a $K$-energy minimizer is a stationary solution along the weak Calabi flow. By our result on the Calabi flow with rough initial data, if the initial metric is a $K$-energy minimizer $\omega_u$ in a $L^\infty$ $\delta$-neighborhood of $\omega$, then $\omega_{u(t)}$ is smooth for $t>0$. On the other hand, $\omega_{u(t)}$ is stationary, and hence $\omega_{u(t)}=\omega_{u}$ is a smooth minimizer; that is a constant scalar curvature metric. This completes the proof. 
\end{proof}

\appendix

\section{The global parametrix}\label{secA}

The problem of defining the global parametrix of $(\frac{\partial }{\partial t} + \Delta_g^2)$ on closed manifolds is a local matter. Given any point $p\in M$ and $r_0>0$ defined as in Lemma $\ref{lem201}$, we first restrict ourselves in a small geodesic ball $B_{r_0}(p, g)$ centered at $p$. Under the normal coordinate $B_{r_0}(0) \cong B_{r_0}(p, g)$, 
\begin{equation}
\Delta_g^2 u(x) = g^{ij}(x) g^{kl}(x) u_{ijkl} + \sum_{|\alpha| \leq 3} A_{\alpha}(x) D_x^\alpha u. 
\end{equation}
For any $\xi \in B_{r_0}(0)$, set
\begin{equation}\label{eqn3005}
\begin{split}
& P_0(\xi, D_x) u = g^{ij}(\xi) g^{kl}(\xi) u_{ijkl} , \\
& P_1(\xi, D_x) u  = \sum_{|\alpha| \leq 3} A_{\alpha}(\xi) D_x^\alpha u , 
\end{split}
\end{equation}
so that $\Delta_g^2 u (x)  = P_0(x, D_x) u + P_1(x, D_x) u $. 

Consider an auxiliary parabolic equation on $\bR^n $
\begin{equation}\label{eqn3001}
\frac{\partial u}{\partial t} + P_{0}(\xi, D_x) u = 0
\end{equation}
with constant coefficients $g^{ij}(\xi) g^{kl}(\xi)$ depending on a parameter $\xi \in B_{r_0}(0)$. The heat kernel of equation $(\ref{eqn3001})$ can be constructed explicitly using Fourier transform on $\bR^n$. We summarize the results from \cite{Fr} in the following proposition. 

\begin{prop}\label{prop301}
Fix $ 0<T < \infty$, for any $\xi \in B_{r_0}(0)$ there exists a smooth function $G(\cdot; \cdot; \xi) \in C^{\infty}\big(\bR^n \times (0,T) \big)$ such that 

\begin{enumerate}

\item[i)] For any integer $k \geq 0$ and multi index $\alpha \geq 0$, 
\begin{equation}
|\partial_t^k D_x^\alpha G(x;t; \xi) | \leq C t^{- \frac{n+4k +|\alpha|}{4}} \exp \{ - \delta (t^{- \frac{1}{4}}|x|)^{\frac{4}{3}}\},  
\end{equation}
where $C, \delta>0$ depend on $T$, $4k+|\alpha|$ and the upper and lower bounds of $g^{ij}(\xi)g^{kl}(\xi)$. 

\item[ii)] $G(\cdot; \cdot; \xi)$ satisfies equation $(\ref{eqn3001})$ with parameter $\xi$.

\item[iii)] For any bounded continuous function $u$ on $\bR^n$, 
\begin{equation}
\lim_{t \rightarrow 0^+} \int_{\bR^n} G(x- y; t; \xi) u(y) \mathrm d y  = u(x). 
\end{equation}
The convergence is uniform for all $x$ in bounded subset of $\bR^n$ and for all $\xi \in B_{r_0}(0)$. 

\item[iv)] $G(x;t; \xi)$ varies smoothly with respect to parameter $\xi \in B_{r_0}(0)$ and we have for any integer $k \geq 0$ and multi indices $\alpha, \beta \geq 0$, 
\begin{equation}
|\partial_t^k D_x^\alpha D_\xi^\beta G(x; t; \xi)| \leq C'  t^{- \frac{n+4k +|\alpha|}{4}} \exp \{ - \delta'  (t^{- \frac{1}{4}}|x|)^{\frac{4}{3}}\}, 
\end{equation}
where $C', \delta' >0$ depends on $T$, $4k+ |\alpha|$ and bounds on derivatives up to order $|\beta|$ of $g^{ij} g^{kl}$ on $B_{r_0}(0)$. 

\end{enumerate}
\end{prop}	

\begin{proof}
This proposition is a summary of results in \cite{Fr}(Page 241, Theorem 1 and Page 250, Lemma 4.). For the uniform convergence in iii), we refer to \cite{Ei} (Page 66).  Therefore we omit the proof here. 
\end{proof}

Next we patch these locally built heat kernels described in Proposition $\ref{prop301}$ together using a partition of unity and appropriate cutoff functions on $M$. Suppose $r_0 > 0$ is as in Lemma $\ref{lem201}$. Suppose $\{ B_{\frac{r_0}{2}}(p_\nu, g) \}$ is a finite open cover of $M$ and $\{\phi_\nu\}$ is a partition of unity subordinate to $\{ B_{\frac{r_0}{2}}(p_\nu, g) \}$, namely we assume that for each $\nu $, $\phi_\nu \in C^{\infty}_0 \big(B_{\frac{3 r_0}{4}}(p_\nu, g)\big)$ with $\phi_\nu >0 $ on $B_{\frac{r_0}{2}}(p_\nu, g)$ and that $\sum_\nu \phi_\nu = 1 $ on M. Suppose $\{\psi_\nu\}$ is a sequence of smooth functions on $M$ such that for each $\nu$, $\psi_\nu \in C^\infty_0\big(B_{\frac{7}{8}r_0}(p_\nu, g)\big)$ and $\psi_\nu = 1$ on $B_{\frac{3 r_0}{4}}(p_\nu, g)$. We define \emph{the global parametrix} of $(\frac{\partial }{\partial t} + \Delta_g^2)$ as following 
\begin{equation}\label{eqn3004}
Z(x, y; t) = \sum_\nu \psi_\nu(x) G_\nu(x-y; t; y) \phi_\nu(y) (\det g_\nu (y))^{-\frac{1}{2}},  
\end{equation}
where $G_\nu$ denotes the heat kernel introduced in Proposition $\ref{prop301}$ of the auxiliary equation $(\ref{eqn3001})$ with coefficients given by the leading term of operator $\Delta_g^2$ under the local normal coordinates $B_{r_0}(0) \cong B_{r_0}(p_\nu, g)$.

We summarize the properties of the global parametrix as following.

\begin{prop}\label{prop302}
Following the discussion above, we have

\begin{enumerate}

\item[i)] $Z \in C^{\infty}\big(M\times M \times (0,T)\big)$ and for any integers $k, p, q \geq 0$, 
\begin{equation}\label{eqn3003}
|\partial_t^k \nabla_x^p \nabla_y^q Z(x, y; t)|_g \leq C t^{-\frac{n+4k +p + q}{4}} \exp\{- \delta \big(t^{-\frac{1}{4}} \rho(x, y) \big)^{\frac{4}{3}}\}, 
\end{equation}
where $C,\delta > 0$ depends on $T$, $g$ and $4k+p+q$. 

\item[ii)] Denote 
\begin{equation}
K(x, y; t) = -( \frac{\partial }{\partial t} + \Delta_g^2) Z(x, y; t), 
\end{equation}
where $\Delta_g^2$ is taken on variable $x$. For any integers $k, p, q \geq 0$, we have
\begin{equation}\label{eqn3007}
|\partial_t^k \nabla_x^p \nabla_y^q K(x, y; t)|_g \leq C' t^{-\frac{n+3+4k+p+q}{4}}  \exp\{- \delta' \big(t^{-\frac{1}{4}} \rho(x, y) \big)^{\frac{4}{3}}\}, 
\end{equation} 
where $C',\delta' > 0$ depends on $T$, $g$ and $4k+p+q$.

\item[iii)] For any function $u \in C^0(M)$, we have
\begin{equation}
\lim_{t \rightarrow 0^+} \int_M Z(x, y; t) u(y) \mathrm d V_g(y) = u(x)
\end{equation}
uniformly for all $x \in M$. 

\item[iv)] For each $\nu$, under normal coordinate $B_{r_0}(0) \cong B_{r_0}(p_\nu, g)$, we have for any $x, y \in B_{r_0}(0)$ and multi indices $\alpha, \beta \geq 0$, 
\begin{equation}\label{eqn3012}
\begin{split}
& |D_x^\beta (D_x + D_y)^\alpha Z(x, y; t) | \leq C t^{-\frac{n+ |\beta|}{4}} \exp\{- \delta \big(t^{-\frac{1}{4}} \rho(x, y) \big)^{\frac{4}{3}} \}, 
 \end{split}
\end{equation}
and 
\begin{equation}\label{eqn3014}
\begin{split}
& |D_x^\beta (D_x + D_y)^\alpha K (x, y; t) | \leq C' t^{-\frac{n+ 3+ |\beta|}{4}} \exp\{- \delta'  \big(t^{-\frac{1}{4}} \rho(x, y) \big)^{\frac{4}{3}} \}, 
\end{split}
\end{equation}
where $C, C', \delta, \delta' >0$ are constants depending on $T, g$ and $|\alpha|+ |\beta|$. 

\end{enumerate}
\end{prop}

\begin{proof}

In this proof, we denote for each $\nu$, 
\begin{equation}
Z_\nu (x, y; t) = \psi_\nu(x) G_\nu(x-y; t; y) \phi_\nu(y) (\det g_\nu (y))^{-\frac{1}{2}}.
\end{equation}
Thus $Z = \sum_\nu Z_\nu$.

First we prove i). $Z \in C^{\infty}\big(M\times M \times (0,T)\big)$ since each term $Z_\nu$ in the finite sum $(\ref{eqn3004})$ is a smooth function on $M\times M \times (0,T)$. For each $\nu$, we have by Proposition $\ref{prop301}$ i) and iv), for any $x, y \in B_{r_0}(0) \cong B_{r_0}(p_\nu, g)$
\begin{equation}\label{eqn3002}
\begin{split}
 |\partial_t^k \nabla_x^p \nabla_y^q Z_\nu (x,y; t)|_g & \leq C_\nu t^{-\frac{n+4k +p + q}{4}} \exp\{- \delta_\nu \big(t^{-\frac{1}{4}} |x- y| \big)^{\frac{4}{3}}\}\\
&\leq C_\nu t^{-\frac{n+4k +p + q}{4}} \exp\{- 2^{-\frac{4}{3}}\delta_\nu \big(t^{-\frac{1}{4}} \rho(x, y) \big)^{\frac{4}{3}}\}. 
\end{split}
\end{equation}
Constants $C_\nu, \delta_\nu > 0$ depend on $g$, $T$ and $4k + p+ q$. If $x$ or $y$ $\notin  B_{r_0}(p_\nu, g)$, then inequality $(\ref{eqn3002})$ trivially holds. Thus for any $x, y \in M$, inequality $(\ref{eqn3002})$ is true. By taking $C = \sum_\nu C_\nu$ and $\delta  = \min_\nu 2^{-\frac{4}{3}} \delta_\nu$, we can get estimate $(\ref{eqn3003})$.

Secondly, we prove ii). Denote 
\begin{equation}
K_\nu(x, y;t) = - (\frac{\partial }{\partial t} + \Delta_g^2) Z_\nu(x, y; t)
\end{equation}
and thus $K = \sum_\nu K_\nu$. For each $\nu$, since $Z_\nu(\cdot, \cdot; t)$ is supported in $B_{r_0}(p_\nu, g) \times B_{r_0} (p_\nu, g)$, so is $K_\nu$. Then it suffices to compute $K_\nu $ in normal coordinate $B_{r_0}(0) \cong B_{r_0}(p_\nu, g)$. In local normal coordinate $B_{r_0}(0) \cong B_{r_0}(p_\nu,g)$, we have 
\begin{equation}
\Delta_g^2 u   = g_\nu^{ij}(x) g_\nu^{kl}(x) u_{ijkl} + \sum_{|\alpha| \leq 3} A_{\nu,\alpha} (x) D_x^\alpha u. 
\end{equation}
Similar to $(\ref{eqn3005})$ we set for $\xi \in B_{r_0}(0)$, 
\begin{equation}
\begin{split}
& P_{\nu,0} (\xi, D_x) u = g_\nu^{ij}(\xi) g_\nu^{kl}(\xi) u_{ijkl}, \\
& P_{\nu,1}(\xi, D_x) = \sum_{|\alpha| \leq 3} A_{\nu,\alpha} (\xi) D_x^\alpha u. 
\end{split}
\end{equation}
By direct computation, we have 
\begin{equation}\label{eqn3006}
\begin{split}
K_\nu (x, y; t) & = - \big( \det g_\nu (y)\big)^{-\frac{1}{2}} \phi_\nu(y ) \big[ \psi_\nu (x) (\frac{\partial }{\partial t} + \Delta_g^2) G_\nu (x-y; t; y) \\
& + \sum_{i=1}^3 \nabla_x^{i}\psi_\nu(x) \star \nabla_x^{4-i} G_\nu(x-y; t; y) \big]\\
& = - \big( \det g_\nu (y)\big)^{-\frac{1}{2}} \phi_\nu(y ) \big[  \psi_\nu(x) (P_{\nu, 0}(x, D_x) - P_{\nu, 0}(y, D_x)) G_\nu (x-y; t; y) \\
&+  \psi_\nu (x) P_{\nu,1}(x, D_x)  G_\nu (x-y; t; y) + \sum_{i=1}^3 \nabla_x^{4-i}\psi_\nu(x) \star \nabla_x^{i} G_\nu(x-y; t; y) \big]. 
\end{split}
\end{equation}
Last step is because $\partial_t G_\nu(x-y; t; y) = - P_{\nu,0} (y, D_x) G_\nu(x-y; t; y)$ by Proposition $\ref{prop301}$ ii). By computation $(\ref{eqn3006})$, we have that for any $x, y \in B_{r_0}(0) \cong B_{r_0}(p, g)$
\begin{equation}
|K_\nu(x, y; t)| \leq C_{\nu } \big[ |x- y| \cdot   t^{-\frac{n+4}{4}} \exp\{ -  \delta_{\nu} (t^{-\frac{1}{4}} |x -y |)^{\frac{4}{3}} \} + t^{-\frac{n+3}{4}} \exp\{ -  \delta_{\nu} (t^{-\frac{1}{4}} |x -y |)^{\frac{4}{3}} \} \big], 
\end{equation}
for appropriate constants $C_\nu, \delta_\nu >0$ depending on $g, T$. By choosing smaller $\delta'_\nu < \delta_\nu$ and appropriate $C'_\nu \gg C_\nu$, we get
\begin{equation}\label{eqn3008}
\begin{split}
|K_\nu(x, y; t)| & \leq C'_{\nu }  t^{-\frac{n+3}{4}} \exp\{ -  \delta'_{\nu} (t^{-\frac{1}{4}} |x -y |)^{\frac{4}{3}} \} \\
& \leq  C'_{\nu }  t^{-\frac{n+3}{4}} \exp\{ -  2^{-\frac{4}{3}}\delta'_{\nu} (t^{-\frac{1}{4}} \rho(x,y))^{\frac{4}{3}} \}.
\end{split}
\end{equation}
In fact, inequality $(\ref{eqn3008})$ holds for any $x, y \in M$ since it's trivially true outside the support of $K_\nu(\cdot, \cdot; t)$.
Similarly for each $\nu$, by taking derivatives on both hand sides of the equation $(\ref{eqn3006})$ and using Proposition $\ref{prop301}$ i) and iv), we can get the estimate $(\ref{eqn3007})$ of $\partial_t^k \nabla_x^p \nabla_y^q K_\nu(x,y;t)$ for any $(x, y; t) \in M\times M \times (0,T)$. Thus, $K = \sum_\nu K_\nu$ satisfies inequality $(\ref{eqn3007})$ for appropriate constants $C', \delta'> 0$ depending on $T$, $g$ and $4k+p+q$.

Thirdly, we prove iii). For $u \in C^{0}(M)$, it suffices to show that for each $\nu$, 
\begin{equation}
\lim_{t \rightarrow 0^+} \int_{M} Z_\nu(x, y; t) u(y)  \mathrm d  V_g(y) = \phi_{\nu}(x) u(x), 
\end{equation}
uniformly for all $x \in M$. Since $Z_\nu(\cdot, \cdot; t)$ is supported in $B_{r_0}(p_\nu, g) \times B_{r_0}(p_\nu, g)$, under local normal coordinate $x, y \in B_{r_0}(0) \cong B_{r_0}(p_\nu, g)$, it is equivalent to show that 
\begin{equation}
\lim_{t \rightarrow 0^+} \int_{\bR^n} Z_\nu(x, y; t) u(y)  (\det g_\nu(y))^{\frac{1}{2}} \mathrm d  y = \phi_{\nu}(x) u(x), 
\end{equation}
uniformly for all $x \in B_{r_0}(0)$. We compute
\begin{equation}
\begin{split}
& \int_{\bR^n} Z_\nu(x, y; t) u(y) \sqrt{\det g_\nu(y)} \mathrm d y \\
& = \psi_\nu(x) \big[ \int_{\bR^n} G_\nu(x-y; t; x) (\phi_\nu u) (y ) \mathrm d y 
 + \int_{\bR^n} \big( G_\nu(x-y; t; y) - G_\nu(x-y; t; x) \big) (\phi_\nu u) (y ) \mathrm d y \big] \\
& =\psi_\nu(x)( I_1 + I_2 ) . 
\end{split}
\end{equation}
By Proposition $\ref{prop301}$ iii), we know that $I_1 \rightarrow (\phi_\nu u)(x)$ uniformly for all $x \in B_{r_0}(0)$ as $t \rightarrow 0^+$. By Proposition $\ref{prop301}$ iv), we have that for appropriate constants $C ,\delta >0$
\begin{equation}
\begin{split}
|I_2| & \leq \int_{\bR^n} C t^{-\frac{n}{4}} \exp \{ - \delta (t^{- \frac{1}{4}} |x- y|)^{\frac{4}{3}} \} |x- y| (\phi_\nu u)(y) \mathrm d y\\
& \leq C t^{\frac{1}{4}}  \int_{\bR^n} \exp \{- \delta | w |^{\frac{4}{3}}\} |w | (\phi_\nu u)( x + t^{\frac{1}{4}} w) \mathrm d w\\
& \leq C t^{\frac{1}{4}} \sup_{B_{r_0}(0)} |\phi_\nu u|. 
\end{split}
\end{equation}
Notice that $\psi_\nu = 1$ on the support of $\phi_\nu$ and thus $\psi_\nu \phi_\nu = \phi_\nu$. Then we have that
\begin{equation}
\lim_{t \rightarrow 0^+} \int_{\bR^n} Z_\nu(x, y; t) u(y) \sqrt{\det g_\nu(y)} \mathrm d y = \psi_\nu(x) \phi_{\nu}(x) u(x) = (\phi_\nu u)(x). 
\end{equation}
The convergence is uniform for all $x\in B_{r_0}(0) \cong B_{r_0} (p_\nu ,g)$. Thus
\begin{equation}
\lim_{t \rightarrow 0^+} \int_{M} Z(x, y; t) u(y)  \mathrm d V_g(y) = \sum_\nu (\phi_\nu u)(x) = u(x), 
\end{equation}
and the convergence is uniform for all $x\in M$. 

Lastly, we prove iv). Fix $\nu$ and for $x, y \in B_{r_0}(p_\nu, g)$
\begin{equation}
Z(x, y; t)  =  Z_\nu(x, y; t)+ \sum_{\mu \neq \nu} Z_\mu (x, y; t)   . 
\end{equation}

We first consider $Z_\nu(x, y; t)$ in local normal coordinate $B_{r_0}(0) \cong B_{r_0}(p_\nu, g)$. Compute for any multi indices $\alpha, \beta \geq 0$
\begin{equation}\label{eqn3008}
\begin{split}
& D_x^\beta (D_x + D_y)^\alpha Z_\nu(x, y; t) \\
= & \sum_{\beta_1 + \beta_2 = \beta} \sum_{\alpha_1 + \alpha_2 + \alpha_3 = \alpha} \big(D_x^{\beta_1+ \alpha_1}  \psi_\nu(x) \big)\big(D_y^{\alpha_2} (\frac{\phi_\nu(y)}{\sqrt{\det g_\nu(y)}}) \big) D_x^{\beta_2} (D_x+ D_y)^{\alpha_3} G_\nu(x-y; t; y) \\
= &  \sum_{\alpha_1 + \alpha_2 + \alpha_3 = \alpha} \big(D_x^{\beta_1 + \alpha_1} \psi_\nu(x) \big)\big(D_y^{\alpha_2} (\frac{\phi_\nu(y)}{\sqrt{\det g_\nu(y)}}) \big) \big( D_x^{\beta_2} D_\xi^{\alpha_3} G_\nu(x-y; t; \xi) \big) |_{\xi = y}. 
\end{split} 
\end{equation}
Using Proposition $\ref{prop301}$ iv), we can get estimate for any multi indices $\alpha, \beta \geq 0$,
\begin{equation}\label{eqn3011}
|D_x^\beta (D_x + D_y)^\alpha Z_\nu(x, y; t)| \leq C t^{-\frac{n+ |\beta|}{4}} \exp\{- \delta (t^{-\frac{1}{4}} \rho(x, y))^{\frac{4}{3}}\}. 
\end{equation}
$C, \delta> 0$ are appropriate constants depending on $g$,$T$ and $|\alpha|+ |\beta|$. 

We then consider $Z_\mu (x, y; t)$ for $\mu \neq \nu$. By definition, $Z_\mu$'s are defined using local normal coordinate for $u, v \in B_{r_0}(0) \xrightarrow{\sigma_\mu} B_{r_0}(p_\mu ,g)$
\begin{equation}
Z_\mu(\sigma_\mu(u) , \sigma_\mu(v); t) = \psi_\mu(\sigma_\mu(u)) G_\mu(u - v;  t; v) \phi_\mu (\sigma_\mu(v)) \big(\det g_\mu(v)\big)^{-\frac{1}{2}}. 
\end{equation} 
By virtue of $(\ref{eqn3008})$ and $(\ref{eqn3011})$, for any $u, v \in B_{r_0}(0) \cong B_{r_0}(p_\mu, g)$ and any multi indices $\alpha, \beta \geq 0$, we have estimates
\begin{equation}
|D_u^\beta (D_u + D_v)^\alpha Z_\mu (\sigma_\mu(u), \sigma_\mu(v); t)| \leq C t^{-\frac{n+ |\beta|}{4}} \exp\{- \delta (t^{-\frac{1}{4}} \rho(u,v))^{\frac{4}{3}}\}, 
\end{equation}
for appropriate constants $C, \delta > 0$ depending on $g, T$ and $|\alpha|+ |\beta|$.

If $B_{r_0}(p_\mu, g)\cap B_{r_0} (p_\nu, g) \neq \varnothing$ for some $\mu \neq \nu$, then $Z_\mu(x, y; t)$ might have nontrivial contributions to $Z(x, y; t)$ for $x, y \in B_{r_0}(p_\mu, g)\cap B_{r_0} (p_\nu, g)$. We denote the normal coordinate in $B_{r_0}(p_\nu, g)$ as $B_{r_0}(0) \xrightarrow{\sigma_\nu} B_{r_0}(p_\nu, g)$ and denote the coordinate change $\sigma = \sigma_\mu^{-1} \sigma_{\nu} : \sigma_\nu^{-1}(B_{r_0}(p_\mu, g)\cap B_{r_0} (p_\nu, g)) \rightarrow \sigma_{\eta}^{-1} (B_{r_0}(p_\mu, g)\cap B_{r_0} (p_\nu, g))$. Suppose $u = \sigma(x)$ and $v =\sigma(y)$, then we have
\begin{equation}
\begin{split}
\frac{\partial }{\partial x_l} + \frac{\partial }{\partial y_l} &=  \frac{\partial \sigma^k}{\partial \xi_l}(x) \frac{\partial }{\partial u_k} + \frac{\partial \sigma^k}{\partial \xi_l}(y) \frac{\partial }{\partial v_k} \\
&= \big( \frac{\partial \sigma^k}{\partial \xi_l}(\sigma^{-1}(u)) - \frac{\partial \sigma^k}{\partial \xi_l}(\sigma^{-1}(v)) \big)  \frac{\partial }{\partial u_k} + \frac{\partial \sigma^k}{\partial \xi_l}(\sigma^{-1}(v)) \big( \frac{\partial }{\partial u_k}+ \frac{\partial }{\partial v_k} \big).
\end{split}
\end{equation}
By induction, we can prove that
\begin{equation}\label{eqn00003}
(D_x + D_y)^\alpha = \sum_{|\gamma|+ |\eta| \leq |\alpha|} \big(\Pi_{l=1}^{|\gamma|} (f_{\gamma,\eta,l} (u) - f_{\gamma,\eta, l}(v) )  \big) \star D_u^\gamma (D_u + D_v)^\eta, 
\end{equation}
where $f_{\gamma, \eta, l} $'s are smooth functions depending on the coordinate change $\sigma$ and the ``$\star$" abbreviate for product with smooth functions on $u, v$ as well as proper contractions. And consequently we have
\begin{equation}\label{eqn3015}
D_x^{\beta} (D_x + D_y)^\alpha = \sum_{|\gamma'|+|\eta'| \leq |\beta|} \sum_{|\gamma|+ |\eta| \leq |\alpha|} \big(D_u^{\eta'}\Pi_{l=1}^{|\gamma|} (f_{\gamma,\eta,l} (u) - f_{\gamma,\eta, l}(v) )  \big) \star D_u^{\gamma + \gamma'} (D_u + D_v)^\eta. 
\end{equation}
Thus we have for $x, y \in \sigma_\nu^{-1}(B_{r_0}(p_\mu, g)\cap B_{r_0} (p_\nu, g))$ and multi indices $\alpha, \beta \geq 0$, 
\begin{equation}
\begin{split}
& |D_x^\beta (D_x + D_y)^\alpha Z_\mu (\sigma_\nu(x), \sigma_\nu(y); t) | \\
& = |\sum_{|\gamma'|+|\eta'| \leq |\beta|} \sum_{|\gamma|+ |\eta| \leq |\alpha|} \big(D_u^{\eta'}\Pi_{l=1}^{|\gamma|} (f_{\gamma,\eta,l} (u) - f_{\gamma,\eta, l}(v) )  \big) \\
& \star D_u^{\gamma + \gamma'} (D_u + D_v)^\eta Z_\mu(\sigma_\mu(u), \sigma_\mu(v) ; t) |\\
& \leq C \sum_{k + l \leq |\beta|} \big( \sum_{i+ j \leq |\alpha|, i < l }  t^{-\frac{n+i+k}{4}} \exp\{- \delta (t^{-\frac{1}{4}} \rho(u,v))^{\frac{4}{3}}\} \\
& +  \sum_{i+ j \leq |\alpha|, i \geq l } \rho(u,v)^{i-l} t^{-\frac{n+i+k}{4}} \exp\{- \delta (t^{-\frac{1}{4}} \rho(u,v))^{\frac{4}{3}}\} \big)\\
& \leq  C' t^{-\frac{n+ |\beta|}{4}} \exp\{- \delta' (t^{-\frac{1}{4}} \rho(u, v))^{\frac{4}{3}}\}.
\end{split}
\end{equation}
Since $\rho(u, v) = \rho(\sigma(x), \sigma(y)) = \rho(x, y)$, we have
\begin{equation}\label{eqn3013}
|D_x^\beta (D_x + D_y)^\alpha Z_\mu (\sigma_\nu(x), \sigma_\nu(y); t) |  \leq  C' t^{-\frac{n+ |\beta|}{4}} \exp\{- \delta'  (t^{-\frac{1}{4}} \rho(x, y))^{\frac{4}{3}}\}.
\end{equation}
$C',\delta' >0$ are appropriate constants depending on $g, T$ and $|\alpha|+ |\beta|$. Then combining $(\ref{eqn3011})$ and $(\ref{eqn3013})$, we have estimate $(\ref{eqn3012})$ for any $x, y \in B_{r_0}(0) \cong B_{r_0}(p_\nu, g)$ and any multi indices $\alpha, \beta \geq 0$. 

For the estimate $(\ref{eqn3014})$ of $K$, we can similarly first consider $K_\nu$ for $x, y \in B_{r_0}(0) \cong B_{r_0}(p_\nu, g)$. By taking derivatives $D_x^\beta (D_x + D_y)^\alpha$ on $(\ref{eqn3006})$ and applying Proposition $\ref{prop301}$ iv), we can prove that
\begin{equation}
 |D_x^\beta (D_x + D_y)^\alpha K_\nu (x, y; t) | \leq C' t^{-\frac{n+ 3+ |\beta|}{4}} \exp\{- \delta'  \big(t^{-\frac{1}{4}} \rho(x, y) \big)^{\frac{4}{3}} \}
\end{equation}
For $\mu \neq \nu$ but $B_{r_0}(p_\mu, g)\cap B_{r_0} (p_\nu, g) \neq \varnothing$, by virtue of $(\ref{eqn3015})$, we can prove that for $x, y \in \sigma_\nu^{-1}(B_{r_0}(p_\mu, g)\cap B_{r_0} (p_\nu, g)) \subset B_{r_0}(0)$
\begin{equation}
 |D_x^\beta (D_x + D_y)^\alpha K_\mu (x, y; t) | \leq C' t^{-\frac{n+ 3+ |\beta|}{4}} \exp\{- \delta'  \big(t^{-\frac{1}{4}} \rho(x, y) \big)^{\frac{4}{3}} \}
\end{equation}
Thus combing results above, we have shown $(\ref{eqn3014})$ for for any $x, y \in B_{r_0}(0) \cong B_{r_0}(p_\nu, g)$ and any multi indices $\alpha, \beta \geq 0$.
\end{proof}

\section{Construction of biharmonic heat kernel on closed manifolds}\label{secB}

Denote the global parametrix constructed in the previous section as $Z(x, y; t)$. Using the method of Levi, in this section we will construct the biharmonic heat kernel in the form
\begin{equation}\label{eqn504}
b_g(x, y; t) = Z(x, y; t) + \int_{0}^t \int_M Z(x, \xi; t-s) \Psi (\xi, y; s) \mathrm d V_g(\xi) \mathrm d s
\end{equation}
for some function $\Psi \in C^0 \big(M \times M \times (0,T) \big)$ with $\Psi(\cdot, \cdot; t) \in C^\infty(M \times M)$ for any $0<t <T$ satisfying Apriori restrictions: 
\begin{enumerate}
\item[i)] For any $x, y \in M$ and integers $p, q \geq 0$, 
\begin{equation}\label{eqn502}
| \nabla_x^p \nabla_y^q \Psi (x, y; t)|_g \leq C t^{-\frac{n+3+p+q }{4}} \exp\{- \delta \big(t^{-\frac{1}{4}} \rho(x, y)\big)^{\frac{4}{3}}\}.
\end{equation}
$C, \delta > 0$ are appropriate constants depending on $g, T$ and $p+q$. 
\item[ii)] For each $\nu$, in local normal coordinate $x, y \in B_{r_0}(0) \cong B_{r_0}(p_\nu, g)$, we have for any multi indices $\alpha, \beta \geq 0$
\begin{equation}\label{eqn503}
|D_x^{\beta} (D_x + D_y)^\alpha \Psi(x,y;t)| \leq  C t^{-\frac{n+3+|\beta|}{4}} \exp\{- \delta \big(t^{-\frac{1}{4}}\rho(x,y)\big)^{\frac{4}{3}}\}. 
\end{equation}
$C, \delta > 0$ are appropriate constants depending on $g, T$ and $|\alpha|+ |\beta|$.
\end{enumerate}
The main purpose in this section is to seek an appropriate function $\Psi$ with Apriori restrictions i) and ii) such that $b_g(x, y; t)$ defined in $(\ref{eqn504})$ is the biharmonic heat kernel with estimates described in Theorem $\ref{thm301}$ and Theorem $\ref{thm302}$. We will specify the function $\Psi$ and prove that it satisfies the Apriori bounds i) and ii) later. But first of all, let's explain why we seek the biharmonic heat kernel $b_g(x, y; t)$ in the form of $(\ref{eqn504})$. It amounts to study improper integral of the following form,
\begin{equation}
(A * B)(x, y; t) := \int_0^t \int_M A(x, \xi ; t-s) b_g(\xi, y; s)\mathrm d V_g(\xi) \mathrm d s. 
\end{equation}
We summarize its property in the following lemma. 

\begin{lem}\label{lem503}
Suppose $A, B $ are functions on $M \times M \times (0,T)$ such that for some $\alpha, \beta > 0$
\begin{equation}
\begin{split}
|A(x, y; t)| \leq C_1 t^{-\frac{n+4 - \alpha}{4}} \exp\{ - \delta (t^{-\frac{1}{4}} \rho(x, y))^{\frac{4}{3}}\},  \\
|b_g(x, y; t)| \leq C_2 t^{-\frac{n+4 - \beta}{4}} \exp\{ - \delta_2 (t^{-\frac{1}{4}} \rho(x, y))^{\frac{4}{3}}\}. 
\end{split}
\end{equation}
Constants $\delta, \delta_2, C_1, C_2 >0$ are independent of $x, y, t$. Then $A * B$ is convergent and for any $0< \epsilon < 1$
\begin{equation}\label{eqn510}
|(A*B)(x, y; t)| \leq C(\epsilon) t^{-\frac{n+4 - (\alpha +\beta)}{4}}  \exp\{  -(1 - \epsilon) \delta (t^{-\frac{1}{4}} \rho(x, y))^{\frac{4}{3}}\}, 
\end{equation}
where $\delta = \min\{\delta, \delta_2\}$ and $C(\epsilon) >0 $ depends on $\epsilon, C_1, C_2, \delta, \delta_2, \alpha, \beta$.
\end{lem}

Given the parametrix $Z$ described in Proposition $\ref{prop302}$, then we can show the property of $(Z * \Psi)$ as the following. 

\begin{prop}\label{prop501}
Suppose $\Psi \in C^0 \big( M\times M\times (0,T) \big) $ such that 
\begin{equation}\label{eqn505}
\begin{split}
& |\Psi(x, y; t)| \leq C  t^{-\frac{n+3}{4}} \exp\{ - \delta (t^{-\frac{1}{4}} \rho(x, y))^{\frac{4}{3}}\},   \\
\end{split}
\end{equation}
and that for some $0< \gamma < 1$
\begin{equation}\label{eqn506}
\begin{split}
& |\Psi(x, y; t) - \Psi(x', y; t)| \\
&\leq C'  \big(\rho(x' , x)\big)^\gamma t^{-\frac{n+3 + \gamma}{4}} \max \big\{\exp\{ - \delta'  (t^{-\frac{1}{4}} \rho(x, y))^{\frac{4}{3}}\} , \exp\{ - \delta' (t^{-\frac{1}{4}} \rho(x', y))^{\frac{4}{3}}\}  \big\}.
\end{split}
\end{equation}
$ C, C',\delta, \delta' >0$ are constants independent of $x, x', y, t$. Then $(Z * \Psi)$ is fourth order continuously differentiable with respect to $x$ and first order continuously differentiable with respect to $t$. The $x$-derivatives of $(Z * \Psi)$ up to third order are obtained by formal differentiation under the integral signs and the rest of derivatives are given by

\begin{equation}\label{eqn00010}
\begin{split}
\nabla_x^4 (Z * \Psi) (x, y; t) = 
& \int_{0}^{\frac{t}{2}}  \int_M (\nabla_x^4 Z) (x, \xi; t-s) \Psi(\xi, y; s)  \mathrm d V_g(\xi) \mathrm d s\\
&+  \int_{\frac{t}{2}}^{t}  \int_M (\nabla_x^4 Z) (x, \xi; t-s) [\Psi(\xi, y; s) - \Psi(x, y;s)] \mathrm d V_g(\xi) \mathrm d s\\
& + \int_{\frac{t}{2}}^t \big(\int_M (\nabla_x^4 Z) (x, \xi; t-s) \mathrm d V_g(\xi) \big)  \Psi(x, y;s) \mathrm ds, \\
\frac{\partial }{\partial t} (Z * \Psi) (x, y; t) = & \Psi(x, y; t) +\int_{0}^{\frac{t}{2}}  \int_M (\frac{ \partial }{\partial t}Z) (x, \xi; t-s) \Psi(\xi, y; s)  \mathrm d V_g(\xi) \mathrm d s\\
&+ \int_{\frac{t}{2}}^t  \int_M (\frac{ \partial }{\partial t}Z) (x, \xi; t-s) [\Psi(\xi, y; s) - \Psi(x, y;s)] \mathrm d V_g(\xi) \mathrm d s\\
& + \int_{\frac{t}{2}}^t \big(\int_M (\frac{\partial }{\partial t}Z) (x, \xi; t-s) \mathrm d V_g(\xi) \big)  \Psi(x, y;s) \mathrm ds.
\end{split}
\end{equation}

\end{prop}

Note that if we assume $\Psi \in C^0\big(M \times M \times (0,T)\big)$ and $\Psi(\cdot, \cdot; t) \in C^\infty(M \times M)$ for any $0<t< T$, then Apriori restriction i) implies assumptions $(\ref{eqn505})$ and $(\ref{eqn506})$ in Proposition $\ref{prop501}$. However, using estimates $(\ref{eqn505})$ and $(\ref{eqn506})$ alone, we can show partial differentiability of $Z*\Psi$ as in Proposition $\ref{prop501}$ at most. If we take advantage of the full differentiability of $\Psi$ by assuming Apriori restrictions i) and ii), then we can indeed show the full differentiability of $Z*\Psi$. More precisely, we introduce the following proposition.

\begin{prop}\label{prop502}
Suppose that $\Psi $ is a function on $M \times M \times (0,T)$ such that $\Psi(\cdot, \cdot; t) \in C^\infty(M \times M)$ for any $0<t <T$ satisfies Apriori restrictions i) and ii). Then $(Z * \Psi)(\cdot, \cdot; t) \in C^\infty (M \times M )$ for any $0 < t<T$. Moreover, we have
\begin{enumerate}
\item[i)] For any $x, y \in M$ and integers $p, q \geq 0$, $\nabla_x^p \nabla_y^q (Z *\Psi) (x, y; \cdot) \in C^0\big((0,T)\big)$ and
\begin{equation}
| \nabla_x^p \nabla_y^q (Z *\Psi) (x, y; t)|_g \leq C t^{-\frac{n-1+p+q }{4}} \exp\{- \delta \big(t^{-\frac{1}{4}} \rho(x, y)\big)^{\frac{4}{3}}\}.
\end{equation}
$C, \delta > 0$ are appropriate constants depending on $g, T$ and $p+q$. 
\item[ii)] For each $\nu$, in local normal coordinate $x, y \in B_{r_0}(0) \cong B_{r_0}(p_\nu, g)$, we have for any multi indices $\alpha, \beta \geq 0$
\begin{equation}
|D_x^{\beta} (D_x + D_y)^\alpha (Z *\Psi)(x,y;t)| \leq  C t^{-\frac{n-1+|\beta|}{4}} \exp\{- \delta \big(t^{-\frac{1}{4}}\rho(x,y)\big)^{\frac{4}{3}}\}. 
\end{equation}
$C, \delta > 0$ are appropriate constants depending on $g, T$ and $|\alpha|+ |\beta|$.
\end{enumerate}
\end{prop}

We postpone the proofs of Lemma $\ref{lem503}$, Proposition $\ref{prop501}$ and Proposition $\ref{prop502}$ but use them to study the property of $b_g(x, y;t)$ defined in the form of $(\ref{eqn504})$ first. Suppose that $\Psi $ is a function on $M \times M \times (0,T)$ satisfying the assumption $(\ref{eqn505})$ in Proposition $\ref{prop501}$. For any $u \in C^0(M)$, we have
\begin{equation}
\begin{split}
& \int_M b_g(x, y; t) u(y) \mathrm d V_g(y) \\
 =&  \int_M Z(x, y; t) u(y) \mathrm d V_g(y) + \int_M (Z * \Psi) (x, y; t) u(y) \mathrm d V_g(y)\\
 = & I_1 +I_2. 
\end{split}
\end{equation}
By the Proposition $\ref{prop302}$ iii), we know that $I_1 \rightarrow u(x)$ as $t \rightarrow 0^+$ uniformly for all $x \in M$. For $I_2$, we have by Lemma $\ref{lem503}$
\begin{equation}
\begin{split}
|I_2| & \leq C t^{\frac{1}{4}}  \|u\|_{C^0(M)}\int_M    t^{-\frac{n}{4}} \exp\{ - \delta (t^{-\frac{1}{4}} \rho(x, y))^{\frac{4}{3}}\} \mathrm d V_g(y), \\
& \leq C t^{\frac{1}{4}}  \|u\|_{C^0(M)} \big( \int_{M - B_{r_0}(x, g)} t^{-\frac{n}{4}} \exp\{ - \delta r_0^{\frac{4}{3}} t^{-\frac{1}{3}}\} \mathrm d V_g \\
& + \int_{B_{r_0}(0)} t^{-\frac{n}{4}} \exp\{- 2^{-\frac{4}{3}}\delta (t^{-\frac{1}{4}} w )^{\frac{4}{3}}\} \mathrm dw \big), \\
& \leq C t^{\frac{1}{4}}  \|u\|_{C^0(M)}. 
\end{split}
\end{equation}
Thus $I_2 \rightarrow 0$ as $t \rightarrow 0^+$ uniformly for all $x \in M$ and then 
\begin{equation}
\lim_{t \rightarrow 0^+} \int_M b_g(x, y; t) u(y) \mathrm d V_g(y) = u(x)
\end{equation}
uniformly for all $x \in M$.

If in addition, we assume that $\Psi \in C^0 \big(M \times M \times (0,T)\big)$ satisfies assumption $(\ref{eqn506})$, then by Propostion $\ref{prop501}$, $b_g(x, y; t)$ has enough differentiability for derivatives in $( \frac{\partial}{\partial t} + \Delta_g^2)b $ to make sense. Apply $- (\frac{\partial}{\partial t} + \Delta_g^2)$ on $(\ref{eqn504})$ and we have
\begin{equation}
\begin{split}
- ( \frac{\partial}{\partial t} + \Delta_g^2) b_g(x, y; t) =&  K(x, y; t) - \Psi(x, y; t) \\
& + \int_{0}^{\frac{t}{2}} \int_M K(x, \xi; t-s) \Psi(\xi, y; s)  \mathrm d V_g(\xi) \mathrm ds\\
&+ \int_{\frac{t}{2}}^t \int_M K(x, \xi; t-s) [\Psi(\xi, y; s) - \Psi(x, y;s)] \mathrm d V_g(\xi) \mathrm ds \\
& + \int_{\frac{t}{2}}^t \big( \int_M K(x, \xi;  t-s)\mathrm d V_g(\xi)  \big)   \Psi(x, y;s) \mathrm ds\\
=& K(x, y; t) - \Psi(x, y; t) + \int_0^t \int_M K(x, \xi; t-s) \Psi(\xi, y; s)  \mathrm d V_g(\xi) \mathrm ds.
\end{split}
\end{equation}
Thus in order for $b_g(x,y; t)$ to satisfy equation $(\ref{eqn31})$, we shall seek $\Psi \in C^0\big(M \times M \times (0,T)\big)$ satisfying $(\ref{eqn505})$ and $(\ref{eqn506})$ such that it solves the integral equation
\begin{equation}\label{eqn507}
\Psi(x, y; t) = K(x, y; t) + \int_0^t \int_M K(x, \xi; t-s) \Psi(\xi, y; s)  \mathrm d V_g(\xi) \mathrm ds.
\end{equation}

Assume for now that $\Psi \in C^0\big(M \times M \times (0,T)\big)$ satisfying $(\ref{eqn505})$ and $(\ref{eqn506})$ is a solution to the integral equation above. Then $b$ has enough differentiability and it satisfy equation $(\ref{eqn31})$ in the classical sense. If we further assume that $\Psi(\cdot, \cdot; t) \in C^\infty(M \times M)$ for any $0<t<T$  satisfies Apriori restrictions i) and ii), then by Proposition $\ref{prop502}$ we can show that $b_g(\cdot, \cdot; t) \in C^\infty (M \times M )$ for any $0<t<T$ with estimates on its partial derivatives of $x,y \in M$ as stated in Theorem $\ref{thm301}$ and Theorem $\ref{thm302}$. To see that $b_g(x, y;t)$ is smooth in $t$, we need to use equation $(\ref{eqn31})$, namely $\frac{\partial b}{\partial t}(x,y;t) = - \Delta_g^2 b_g(x,y;t)$. Consider for $|h| \ll 1$
\begin{equation}
\begin{split}
&\frac{\partial b}{\partial t} (x, y; t+h) - \frac{\partial b}{\partial t} (x, y; t), \\
=&  -\Delta_g^2 \big( b_g(x,y; t+h) - b_g(x,y;t)\big), \\
=&  -\Delta_g^2 \int_0^h \frac{\partial b}{\partial t}(x, y; t+\tau)\mathrm d \tau, \\
= & \int_0^h (- \Delta_g^2)^2 b_g(x, y; t+\tau) \mathrm d \tau. 
\end{split}
\end{equation}
By Proposition $\ref{prop502}$, we have $(- \Delta_g^2)^2 b_g(x, y; \cdot)$ is continuous on $(0,T)$. Thus $\frac{\partial b}{\partial t}$ is differentiable with respect to $t$ for $0<t<T$. One can repeat the same argument above for $\frac{\partial^k b }{\partial t^k}$ inductively and conclude that $b$ is smooth in $t$.

Therefore, by our discussion up to this point, in order to prove Theorem $\ref{thm301}$ and Theorem $\ref{thm302}$, it suffices to seek $\Psi \in C^0 \big(M \times M \times (0,T)\big)$ with $\Psi(\cdot, \cdot; t) \in C^\infty(M\times M)$ for any $0 <t<T$ satisfying Apriori restrictions i) and ii) and moreover $\Psi$ solves the integral equation $(\ref{eqn507})$. 

Next, we construct such a solution $\Psi$ that fulfills all requirements above using Neumann series. We define
\begin{equation}
K_1 (x, y; t) : = K(x, y; t), 
\end{equation}
and inductively for all integer $m \geq 2$ we define
\begin{equation}
K_m(x, y; t) = (K * K_{m-1})(x, y; t). 
\end{equation}
We shall see
\begin{equation}\label{eqn508}
\Psi = \sum_{m=1}^\infty K_m, 
\end{equation}
is a solution to $(\ref{eqn507})$ that fulfills all the requirements above. Despite the convergence issues of the improper integrals and the infinite sum involved in the construction of $\Psi$, formally speaking, it should be a solution of the integral equation $(\ref{eqn507})$. We shall prove the following proposition.

\begin{prop}\label{prop503}
$\Psi$ in $(\ref{eqn508})$ is a continuous function on $M \times M \times (0,T)$ and for any $0 <t< T$, $\Psi(\cdot, \cdot; t) \in C^{\infty}(M \times M )$ satisfies Apriori restrictions i) and ii). Moreover, $\Psi$ satisfies the integral equation $(\ref{eqn507})$. 
\end{prop}

We then would start our proofs in the following order: Lemma $\ref{lem503}$, Proposition $\ref{prop503}$, Proposition $\ref{prop501}$ and lastly Proposition $\ref{prop502}$. 

\begin{proof}
Proof of Lemma $\ref{lem503}$. We have that
\begin{equation}
\begin{split}
& | (A * B)(x, y; t) |  \\
\leq & C_1 C_2  \int_0^t  \int_M (t-s)^{-\frac{ n + 4 -\alpha}{4}} s^{-\frac{n + 4 - \beta}{4}}  \exp\{- \delta \big[ ( \frac{ \rho(x, \xi)}{(t-s)^{\frac{1}{4}}})^{\frac{4}{3}} + (\frac{\rho(\xi, y)}{s^{\frac{1}{4}}})^{\frac{4}{3}} \big]\} \mathrm d V_g(\xi) \mathrm d s, 
 \end{split}
\end{equation}
where $\delta = \min\{\delta, \delta_2\}$. Since for any $x, y, \xi \in M$ and $0 < s < t$
\begin{equation}
\frac{\rho(x, \xi)^{\frac{4}{3}} }{(t-s)^{\frac{1}{3}}} + \frac{\rho(\xi, y)^{\frac{4}{3}}}{s^{\frac{1}{3}}} \geq \frac{\rho(x, y)^{\frac{4}{3}}}{t^{\frac{1}{3}}},  
\end{equation}
then we have for any $0 < \epsilon \ll 1$,
\begin{equation}
\begin{split}
& | (A * B)(x, y; t) |  \\
\leq & C \int_0^t  \int_M (t-s)^{-\frac{ n + 4 -\alpha}{4}} s^{-\frac{n + 4 - \beta}{4}}  \exp\{- \delta \epsilon \big[ ( \frac{ \rho(x, \xi)}{(t-s)^{\frac{1}{4}}})^{\frac{4}{3}} + (\frac{\rho(\xi, y)}{s^{\frac{1}{4}}})^{\frac{4}{3}} \big]\} \mathrm d V_g(\xi) \mathrm d s \\
& \times \exp\{ - (1- \epsilon)\delta (t^{-\frac{1}{4}} \rho(x, y))^{\frac{4}{3}} \}. 
 \end{split}
\end{equation}
Break the integral term in previous step into two integrals from $0$ to $\frac{t}{2}$ and from $\frac{t}{2}$ to $t$ and we have
\begin{equation}
\begin{split}
&\int_0^t  \int_M (t-s)^{-\frac{ n + 4 -\alpha}{4}} s^{-\frac{n + 4 - \beta}{4}}  \exp\{- \delta \epsilon \big[ ( \frac{ \rho(x, \xi)}{(t-s)^{\frac{1}{4}}})^{\frac{4}{3}} + (\frac{\rho(\xi, y)}{s^{\frac{1}{4}}})^{\frac{4}{3}} \big]\} \mathrm d V_g(\xi) \mathrm d s \\
\leq & C t^{-\frac{n+4- \alpha}{4}}  \int_0^{\frac{t}{2}} s^{-\frac{4- \beta}{4}} \int_M s^{-\frac{n}{4}} \exp\{- \delta \epsilon (s^{- \frac{1}{4}} \rho(\xi, y))^{\frac{4}{3}} \} \mathrm d V_g(\xi) \mathrm ds \\
+ &C t^{-\frac{n+4- \beta}{4}}  \int_{\frac{t}{2}}^{t} (t-s)^{-\frac{4- \alpha}{4}} \int_M (t-s)^{-\frac{n}{4}} \exp\{- \delta \epsilon \big( (t- s)^{- \frac{1}{4}} \rho(\xi, x)\big)^{\frac{4}{3}} \} \mathrm d V_g(\xi) \mathrm ds, \\
\leq & C(\epsilon) t^{-\frac{n+4 - (\alpha+ \beta)}{4}}. 
\end{split}
\end{equation}
Last step is because
\begin{equation}
\begin{split}
&  \int_M s^{-\frac{n}{4}} \exp\{- \delta \epsilon (s^{- \frac{1}{4}} \rho(\xi, y))^{\frac{4}{3}} \} \mathrm d V_g(\xi)\\
 \leq & C \big( \int_{ B_{r_0}(0) } s^{-\frac{n}{4}}  \exp\{-  2 ^{-\frac{4}{3}}\delta \epsilon (s^{- \frac{1}{4}}|w|)^{\frac{4}{3}} \} \mathrm d w\\
  &+ \int_{M - B_{r_0}(y, g)} s^{-\frac{n}{4}} \exp\{- \delta \epsilon (s^{- \frac{1}{4}} r_0)^{\frac{4}{3}} \} \mathrm d V_g(\xi) \big)
 \leq  C(\epsilon), 
\end{split}
\end{equation}
and similarly 
\begin{equation}
\int_M (t-s)^{-\frac{n}{4}} \exp\{- \delta \epsilon \big( (t- s)^{- \frac{1}{4}} \rho(\xi, x)\big)^{\frac{4}{3}} \} \mathrm d V_g(\xi) \leq C(\epsilon). 
\end{equation}
Thus we get estimate $(\ref{eqn510})$ for $(A * B)$ and it ends the proof of Lemma $\ref{lem503}$. 
\end{proof}

\begin{proof}
Proof of Proposition $\ref{prop503}$. We first show that $\Psi(x,y;t) = \sum_{m=1}^\infty K_m(x,y;t)$ is absolutely convergent. By Proposition $\ref{prop302}$ ii), we have
\begin{equation}
|K(x, y; t)| \leq C_1 t^{-\frac{n+3}{4}}  \exp\{- \delta \big(t^{-\frac{1}{4}} \rho(x, y) \big)^{\frac{4}{3}}\}
\end{equation}
for appropriate constants $C_1, \delta >0$ depending on $g, T$. Thus by Lemma $\ref{lem503}$, we have for any $1 \leq m \leq n+4$, 
\begin{equation}\label{eqn521}
|K_m (x, y; t)| \leq C_2 t^{-\frac{n+4 - m}{4}}  \exp\{- \delta_2 \big(t^{-\frac{1}{4}} \rho(x, y) \big)^{\frac{4}{3}}\}, 
\end{equation}
for appropriate constants $\delta_2 < \delta$ and $C_2 \gg C_1$ depending on $g, T$. For $m \geq n+4$, we can prove by induction that
\begin{equation}\label{eqn520}
|K_m (x, y;  t)| \leq C_2 \frac{C_3^{m - (n+4)}}{\Gamma(\frac{m-n}{4})} t^{-\frac{n+4 - m}{4}} \exp\{- \delta_2 \big(t^{-\frac{1}{4}} \rho(x, y) \big)^{\frac{4}{3}}\} 
\end{equation}
where $\Gamma(\cdot)$ denotes the Gamma function
\begin{equation}
\Gamma(t) : = \int_0^\infty x^{t-1} e^{-x} \mathrm d x. 
\end{equation}
and $C_3 > 0$ is a constant given by
\begin{equation}
C_3 = C_1 \Gamma(\frac{1}{4}) \sup_{x \in M, t \in (0,T)} \int_{M} t^{-\frac{n}{4}} \exp\{- (\delta -\delta_2) ( \frac{ \rho(x, \xi)}{t^{\frac{1}{4}}})^{\frac{4}{3}}\} \mathrm d V_g(\xi) < \infty.
\end{equation}
When $m = n+4$, we see that $(\ref{eqn520})$ holds. We assume for some $m \geq n+4$ estimate $(\ref{eqn520})$ holds, we next show for $m+1$ estimate $(\ref{eqn520})$ is also valid. By definition $K_{m+1} = K * K_m $ and thus we have
\begin{equation}
\begin{split}
& | K_{m+1}(x, y; t) |  \\
\leq & C_1 C_2 \frac{C_3^{m - (n+4)}}{\Gamma(\frac{m-n}{4})}  \int_0^t  \int_M (t-s)^{-\frac{ n + 3}{4}} s^{-\frac{n +4-m}{4}} \exp\{- \delta  ( \frac{ \rho(x, \xi)}{(t-s)^{\frac{1}{4}}})^{\frac{4}{3}} \}\\
&\times \exp\{ - \delta_2 (\frac{\rho(\xi, y)}{s^{\frac{1}{4}}})^{\frac{4}{3}}\} \mathrm d V_g(\xi) \mathrm d s, \\
\leq & C_1 C_2 \frac{C_3^{m - (n+4)}}{\Gamma(\frac{m-n}{4})} \exp\{- \delta_2 (\frac{\rho(x, y)}{t^{\frac{1}{4}}})^{\frac{4}{3}}\}
\int_0^t   (t-s)^{-\frac{3}{4}} s^{-\frac{n +4-m}{4}}  \\
& \times \int_M (t-s)^{-\frac{n}{4}}\exp\{- (\delta -\delta_2) ( \frac{ \rho(x, \xi)}{(t-s)^{\frac{1}{4}}})^{\frac{4}{3}} \} \mathrm d V_g(\xi) \mathrm d s,  \\
\leq & C_2 \frac{C_3^{m+1 - (n+4)}}{\Gamma(\frac{m-n}{4}) \Gamma(\frac{1}{4})} \exp\{- \delta_2 (\frac{\rho(x, y)}{t^{\frac{1}{4}}})^{\frac{4}{3}}\} t^{-\frac{n+4-(m+1)}{4}} \int_{0}^1 (1-s)^{\frac{1}{4}-1} s^{\frac{m-n}{4}-1} \mathrm ds,\\
\leq & C_2 \frac{C_3^{m+1 - (n+4)}}{\Gamma(\frac{m+1-n}{4}) } t^{-\frac{n+4-(m+1)}{4}} \exp\{- \delta_2 (\frac{\rho(x, y)}{t^{\frac{1}{4}}})^{\frac{4}{3}}\} . 
\end{split}
 \end{equation} 
Thus we have shown estimate $(\ref{eqn520})$ for any $m \geq n+4$. Combing estimates $(\ref{eqn521})$ and $(\ref{eqn520})$, we have
\begin{equation}
|\Psi(x, y; t) | \leq  C_2 t^{-\frac{n+3}{4}}  \exp\{- \delta_2 \big(t^{-\frac{1}{4}} \rho(x, y) \big)^{\frac{4}{3}}\} \big( \sum_{m=1}^{n+4} T^{\frac{m-1}{4}} + \sum_{m = n+5}^\infty \frac{C_3^{m- (n+4)} T^{\frac{m-1}{4}}}{\Gamma(\frac{m-n}{4})}  \big).
\end{equation}
Since $\Gamma (k ) = (k-1)!$ for integer $k \geq 1$, we have that the infinite sum in the last step is finite. Therefore, $\Psi(x, y;t)$ is well-defined and for appropriate constants $C,\delta >0$ only depending on $g, T$ we have
\begin{equation}
|\Psi(x, y; t) | \leq  C t^{-\frac{n+3}{4}}  \exp\{- \delta \big(t^{-\frac{1}{4}} \rho(x, y) \big)^{\frac{4}{3}}\}.
\end{equation}

We have for any integer $N \leq 1$,
\begin{equation}
\sum_{m=1}^{N+1} K_m = K  + K * \sum_{m=1}^N K_m.
\end{equation} 
Since $|\sum_{m=1}^N K_m(x,y; t)| \leq C t^{-\frac{n+3}{4}} \exp\{- \delta \big(t^{-\frac{1}{4}} \rho(x, y) \big)^{\frac{4}{3}}\}$ on $M \times M \times (0,T)$ for any integer $N$, by the dominate convergence theorem, we let $N \rightarrow \infty$ in the equation above and get that $\Psi$ satisfies the integral equation $(\ref{eqn507})$. 

In order to show that $\Psi(x,y; \cdot)$ is continuous on $(0,T)$, by integral equation $(\ref{eqn507})$ it suffices to show the integral $(K*\Psi)(x, y; \cdot)$ is continuous on $(0,T)$. Consider for $0< h \ll t$
\begin{equation}
\begin{split}
& (K*\Psi)(x, y; t+h) -(K*\Psi) (x, y; t)\\
=& \int_t^{t+h} \int_M K (x, \xi; t+h -s) \Psi(\xi, y; s) \mathrm d V_g(\xi) \mathrm d s \\
& + \int_0^t \int_M \big( K(x,\xi; t+h-s) - K(x,\xi; t-s) \big) \Psi(\xi, y; s) \mathrm d V_g(\xi) \mathrm d s.\\
\end{split}
\end{equation}
It's clear that the first integral is bounded by $C t^{-\frac{n+3}{4}} h^{\frac{1}{4}}$. While the second integral is bounded by
\begin{equation}
\begin{split}
& \int_0^h \int_0^t \int_M  |\frac{\partial }{\partial t} K(x, \xi; t+\tau-s) \Psi(\xi, y; s)|  \mathrm d V_g(\xi) \mathrm d s \mathrm d \tau.\\
\end{split}
\end{equation}
We can break the integral from $0$ to $t$ into $0$ to $\frac{t}{2}$ and from $\frac{t}{2}$ to $t$ and estimate them separately using ideas developed in Lemma $\ref{lem503}$. Then we can get that the second integral above is bounded by $C t^{-\frac{n+6}{4}} h + C t^{-\frac{n+3}{4}} h^{\frac{1}{4}}$. To sum up we get that 
\begin{equation}
|(K*\Psi)(x, y; t+h) -(K*\Psi) (x, y; t)| \leq C t^{-\frac{n+3}{4}} h^{\frac{1}{4}} + C t^{-\frac{n+6}{4}} h. 
\end{equation}
Therefore $(K*\Psi)(x, y; \cdot)$ and consequently $\Psi(x,y; \cdot) = K(x,y; \cdot)  + (K*\Psi)(x,y; \cdot)$ are continuous on $(0,T)$. In fact by our actual computation we have shown that they are H\"older continuous on $(0,T)$. 

Next we prove the differentiability of $\Psi$ with respect to space variables $x, y$. We first examine $K_2 = K * K$ closely. Set
\begin{equation}
\begin{split}
I(x, y; t) = \int_{0}^{\frac{t}{2}} \int_M K(x, \xi; t-s) K(\xi, y; s) \mathrm d V_g(\xi) \mathrm ds, \\
J(x, y; t) = \int_{\frac{t}{2}}^t \int_M K(x, \xi; t-s) K(\xi, y; s) \mathrm d V_g(\xi) \mathrm ds, 
\end{split}
\end{equation}
so that $K_2 = I + J$. It suffices to study $I$ since by a change of variable $\tau = t-s$
\begin{equation}
J(x, y; t) = \int_{0}^{\frac{t}{2}} \int_M  K(\xi, y; t-\tau) K(x, \xi ; \tau) \mathrm d V_g(\xi) \mathrm d \tau.
\end{equation}

Regarding to the differentiability of $I$, we have the following.

\begin{lem}\label{lemma1}
Suppose $x \in M$ and $y \in B_{\frac{1}{2}r_0} (p_\eta, g)$ for some $\eta$. Then for any integer $k\geq 0$ and multi index $\beta \geq 0$
\begin{equation}\label{eqn00001}
\begin{split}
&\nabla_x^k D_y^\beta I(x, y; t) \\
= &
\int_0^{\frac{t}{2}}  \int_M \sum_{\beta_1 +\beta_2 = \beta} D_\xi^{\beta_1}\big(\nabla_x^k K(x, \xi; t-s) \psi_\eta(\xi) \mathrm d V_g(\xi)\big) (D_y + D_\xi)^{\beta_2} K(\xi, y; s) \mathrm d s \\
& +  \int_{0}^{\frac{t}{2}} \int_M \nabla_x^k K(x, \xi; t-s) (1- \psi_\eta(\xi)) D_y^{\beta} K(\xi, y; s) \mathrm d V_g(\xi) \mathrm d s,
\end{split}
\end{equation}
where $\psi_\eta \in C^\infty_0(B_{\frac{7}{8}r_0} (p_\eta, g)) $ with $\psi_\eta = 1$ on $B_{\frac{3}{4}r_0}(p_\eta, g)$, $D_y$ and $D_\xi$ are ordinary derivatives in local coordinate $B_{r_0}(0) \cong B_{r_0}(p_\nu, g)$. Given that ${\bigcup}_{\nu}B_{\frac{1}{2}r_0}(p_\nu, g) = M$, we have $I(\cdot, \cdot; t) \in C^{\infty}(M \times M)$ for $0<t < T$. 

\end{lem}
\begin{proof}
Proof of lemma $\ref{lemma1}$. First of all, since $s \leq \frac{t}{2}$, 
\begin{equation}
|\nabla_x^m K(x, \xi; t-s)|_g \leq C_m (t-s)^{-\frac{n+3+m}{4}} \leq C_m t^{-\frac{n+3+m}{4}}.
\end{equation}
Thus for any integer $k \geq 0$, 
\begin{equation}
\nabla_x^k I(x, y; t) = \int_0^{\frac{t}{2}} \int_M \nabla_x^k K(x, \xi; t-s) K(\xi, y; s) \mathrm d V_g(\xi) \mathrm d s.
\end{equation}
Given $y \in B_{\frac{1}{2} r_0} (p_\eta, g)$, under local coordinate $B_{r_0} (0) \cong B_{r_0}(p_\eta, g)$, we consider the partial derivative of $\nabla_x^k I(x, y; t) $ with respect to $y_l$. Note that $\frac{\partial }{\partial y_l}$ doesn't commute with the integral sign simply because $\frac{\partial}{\partial y_l} K(\cdot, y; \cdot)$ may not be integrable on $M \times (0, \frac{t}{2})$. Alternatively, we get back to the original definition of derivative, for $h \ll 1$
\begin{equation}
\begin{split}
& \nabla_x^k I(x, y + h e_l; t) - \nabla_x^k I(x, y; t) \\
=& \int_0^{\frac{t}{2}} \int_M \nabla_x^k K(x, \xi; t-s) (\int_0^h\frac{\partial}{\partial y_l} K(\xi, y+ \tau  e_l; s) \mathrm d \tau) \mathrm d V_g(\xi) \mathrm d s, \\
= & \int_0^{\frac{t}{2}} \int_0^h \int_M \nabla_x^k K(x, \xi; t-s) \psi_{\eta}(\xi)\frac{\partial}{\partial y_l} K(\xi, y+ \tau  e_l; s) \mathrm d V_g(\xi) \mathrm d \tau \mathrm d s \\
& +\int_0^{\frac{t}{2}} \int_0^h \int_M \nabla_x^k K(x, \xi; t-s) \big(1-\psi_{\eta}(\xi)\big)\frac{\partial}{\partial y_l} K(\xi, y+ \tau  e_l; s) \mathrm d V_g(\xi) \mathrm d \tau \mathrm d s, \\
= & \int_0^{\frac{t}{2}} \int_0^h \int_M \nabla_x^k K(x, \xi; t-s) \psi_{\eta}(\xi) (\frac{\partial}{\partial y_l} + \frac{\partial}{\partial \xi_l}) K(\xi, y+ \tau  e_l; s) \mathrm d V_g(\xi) \mathrm d \tau \mathrm d s \\
& + \int_0^{\frac{t}{2}} \int_0^h  \int_M \frac{\partial }{\partial \xi_l}\big(\nabla_x^k K(x, \xi; t-s) \psi_{\eta}(\xi)  \mathrm d V_g(\xi) \big) K(\xi, y+ \tau  e_l; s)\mathrm d \tau \mathrm d s \\
& +\int_0^{\frac{t}{2}} \int_0^h \int_M \nabla_x^k K(x, \xi; t-s) \big(1-\psi_{\eta}(\xi)\big)\frac{\partial}{\partial y_l} K(\xi, y+ \tau  e_l; s) \mathrm d V_g(\xi) \mathrm d \tau \mathrm d s. \\
\end{split}
\end{equation}
One can check that each term in last expression is integrable by Proposition $\ref{prop302}$ ii) and iv). Therefore by Fubini's theorem, 
\begin{equation}
\begin{split}
& \nabla_x^k I(x, y + h e_l; t) - \nabla_x^k I(x, y; t) \\
= & \int_0^h \big\{  \int_0^{\frac{t}{2}} \int_M \nabla_x^k K(x, \xi; t-s) \psi_{\eta}(\xi) (\frac{\partial}{\partial y_l} + \frac{\partial}{\partial \xi_l}) K(\xi, y+ \tau  e_l; s) \mathrm d V_g(\xi) \mathrm d s \\
& + \int_0^{\frac{t}{2}}   \int_M \frac{\partial }{\partial \xi_l}\big(\nabla_x^k K(x, \xi; t-s) \psi_{\eta}(\xi)  \mathrm d V_g(\xi) \big) K(\xi, y+ \tau  e_l; s) \mathrm d s \\
& +\int_0^{\frac{t}{2}}\int_M \nabla_x^k K(x, \xi; t-s) \big(1-\psi_{\eta}(\xi)\big)\frac{\partial}{\partial y_l} K(\xi, y+ \tau  e_l; s) \mathrm d V_g(\xi)  \mathrm d s \big\} \mathrm d \tau. 
\end{split}
\end{equation}
Denote the integrand as $Q(x, y+ \tau e_l; t)$ so that the last integral can be written as $\int_0^h Q(x, y+ \tau e_l; t) \mathrm d \tau$. We have that for any $x\in M,  y, y' \in B_{\frac{1}{2}r_0}(p_\eta, g)$ and $0< t <T$
\begin{equation}
|Q(x, y'; t) -Q(x, y; t)| \leq C_t \rho(y, y')^\gamma, 
\end{equation}
where $\gamma \in (0,1)$ and $C_t$ depending on $t$, $k$, $\gamma$ but independent of $x, y, y'$. This is because by Proposition $\ref{prop302}$ ii) and iv) one can derive that for appropriate constant $\delta, C >0$
\begin{equation}
\begin{split}
& |K(\xi, y; s)- K(\xi, y';s)| \\
&\leq C s^{-\frac{n+3+ \gamma}{4}} \rho(y,y')^\gamma \max\{\exp \{- \delta  (s^{-\frac{1}{4}} \rho(\xi, y))^{\frac{4}{3}}\},  \exp \{- \delta  (s^{-\frac{1}{4}} \rho(\xi, y'))^{\frac{4}{3}}\}\},\\
& |(\frac{\partial}{\partial y_l} + \frac{\partial}{\partial \xi_l})K(\xi, y; s)- (\frac{\partial}{\partial y_l} + \frac{\partial}{\partial \xi_l}) K(\xi, y';s)| \\
&\leq C s^{-\frac{n+3+ \gamma}{4}} \rho(y,y')^\gamma \max\{\exp \{- \delta  (s^{-\frac{1}{4}} \rho(\xi, y))^{\frac{4}{3}}\},  \exp \{- \delta  (s^{-\frac{1}{4}} \rho(\xi, y'))^{\frac{4}{3}}\}\},\\
& |\frac{\partial}{\partial y_l} K(\xi, y; s)- \frac{\partial}{\partial y_l} K(\xi, y';s)| \\
&\leq C s^{-\frac{n+4+ \gamma}{4}} \rho(y,y')^\gamma \max\{\exp \{- \delta  (s^{-\frac{1}{4}} \rho(\xi, y))^{\frac{4}{3}}\},  \exp \{- \delta  (s^{-\frac{1}{4}} \rho(\xi, y'))^{\frac{4}{3}}\}\}.
\end{split}
\end{equation} 
Thus following the discussion above, we have $Q(x, y+\tau e_l; t) \rightarrow Q(x, y; t)$ as $\tau \rightarrow 0$. Therefore
\begin{equation}
\begin{split}
& \frac{\partial }{\partial y_l} \nabla_x^k I(x, y; t)  = Q(x, y; t),\\
& =    \int_0^{\frac{t}{2}} \int_M \nabla_x^k K(x, \xi; t-s) \psi_{\eta}(\xi) (\frac{\partial}{\partial y_l} + \frac{\partial}{\partial \xi_l}) K(\xi, y; s) \mathrm d V_g(\xi) \mathrm d s \\
& + \int_0^{\frac{t}{2}}   \int_M \frac{\partial }{\partial \xi_l}\big(\nabla_x^k K(x, \xi; t-s) \psi_{\eta}(\xi)  \mathrm d V_g(\xi) \big) K(\xi, y; s) \mathrm d s \\
& +\int_0^{\frac{t}{2}}\int_M \nabla_x^k K(x, \xi; t-s) \big(1-\psi_{\eta}(\xi)\big)\frac{\partial}{\partial y_l} K(\xi, y; s) \mathrm d V_g(\xi)  \mathrm d s .
\end{split}
\end{equation} 
One can finish the rest of the proof by induction on multi index $\beta \geq 0$ using similar arguments described above. This ends the proof of Lemma $\ref{lemma1}$. 

\end{proof}

By the lemma $\ref{lemma1}$, we get that $K_2(\cdot, \cdot; t) \in C^\infty(M \times M)$ for $0< t <T$. Moreover, we can further derive bounds on its derivatives as below. For $K_1 = K$ and an integer $r \geq 0$, by Proposition $\ref{prop302}$ ii) and iv), there exist constants $C_{1,r}, \delta_{1, r} >0$ such that:

\begin{enumerate}
\item[i)] For any $x, y \in M$ and integers $p, q \geq 0$ with $p+q \leq r$, 
\begin{equation}
| \nabla_x^p \nabla_y^q K_1 (x, y; t)|_g \leq C_{1, r} t^{-\frac{n+3+p+q }{4}} \exp\{- \delta_{1, r} \big(t^{-\frac{1}{4}} \rho(x, y)\big)^{\frac{4}{3}}\}.
\end{equation}
\item[ii)] For each $\nu$, in local normal coordinate $x, y \in B_{r_0}(0) \cong B_{r_0}(p_\nu, g)$, we have for any multi indices $\alpha, \beta \geq 0$ with $|\alpha|+|\beta| \leq r$
\begin{equation}
|D_x^{\beta} (D_x + D_y)^\alpha K_1 (x,y;t)| \leq  C_{1, r} t^{-\frac{n+3+|\beta|}{4}} \exp\{- \delta_{1, r} \big(t^{-\frac{1}{4}}\rho(x,y)\big)^{\frac{4}{3}}\}. 
\end{equation}
\end{enumerate}

\begin{lem}\label{lemma2}
Given $C_{1, r}, \delta_{1, r}>0$ as above, there exist constants $\delta_{2,r} < \delta_{1,r}$ and $C_{2,r} \gg C_{1, r}$ such that:
\begin{enumerate}
\item[i)] For any $x, y \in M$ and integers $p, q \geq 0$ with $p+q \leq r$, 
\begin{equation}
| \nabla_x^p \nabla_y^q K_2 (x, y; t)|_g \leq C_{2,r} t^{-\frac{n+2+p+q }{4}} \exp\{- \delta_{2,r} \big(t^{-\frac{1}{4}} \rho(x, y)\big)^{\frac{4}{3}}\}.
\end{equation}
\item[ii)] For each $\nu$, in local normal coordinate $x, y \in B_{r_0}(0) \cong B_{r_0}(p_\nu, g)$, we have for any multi indices $\alpha, \beta \geq 0$ with $|\alpha|+|\beta| \leq r$
\begin{equation}
|D_x^{\beta} (D_x + D_y)^\alpha K_2 (x,y;t)| \leq  C_{2,r} t^{-\frac{n+2+|\beta|}{4}} \exp\{- \delta_{2,r} \big(t^{-\frac{1}{4}}\rho(x,y)\big)^{\frac{4}{3}}\}. 
\end{equation}
\end{enumerate}
\end{lem}

\begin{proof}
Again, it suffices to work on $I$. Fix arbitrary $\delta < \delta_{1, r}$. Following equation $(\ref{eqn00001})$ and the ideas of the proof in Lemma $\ref{lem503}$, we have for $k+|\beta| \leq r$,
\begin{equation}
\begin{split}
& |\nabla_x^k D_y^\beta I(x, y; t)|_g \\
\leq & C t^{-\frac{n+3+k+|\beta|}{4}}\exp\{- \delta (t^{-\frac{1}{4} } \rho(x, y))^{\frac{4}{3}}\}\\
&\times \int_{B_{r_0}(p_\eta, g) \times (0, \frac{t}{2})} s^{-\frac{n+3}{4}} \exp\{- (\delta_{1,r} - \delta) (s^{-\frac{1}{4} } \rho(\xi, y))^{\frac{4}{3}}\} \mathrm d V_g(\xi) \mathrm d s\\
&+ C t^{-\frac{n+3+k}{4}}  \exp\{- \delta (t^{-\frac{1}{4} } \rho(x, y))^{\frac{4}{3}}\} \\
& \times \int_{(M - B_{\frac{3}{4}r_0}(p_\eta, g) )\times (0, \frac{t}{2})} s^{-\frac{n+3+|\beta|}{4}} \exp\{- (\delta_{1,r} - \delta) (s^{-\frac{1}{4} } \frac{1}{4} r_0)^{\frac{4}{3}}\} \mathrm d V_g(\xi) \mathrm d s, \\
\leq & C'  t^{-\frac{n+2+k+|\beta|}{4}}\exp\{- \delta (t^{-\frac{1}{4} } \rho(x, y))^{\frac{4}{3}}\}, 
\end{split}
\end{equation}
for $x \in M$ and $y \in B_{\frac{1}{2}r_0} (p_\eta, g)$ for some $\eta$. Constant $C' \gg C_{1, r}$ depends on $C_{1, r}, g, T, r$ and $(\delta_{1,r} -\delta)$. Since $\{B_{\frac{1}{2}r_0}(p_\eta, g)\} $ covers $M$, estimate i) for $K_2$ follows.

Next we focus on proving ii). For each $\nu$, we consider $x, y \in B_{r_0}(0) \cong B_{r_0}(p_\nu, g)$. If $y \in B_{\frac{r_0}{2} }(p_\nu, g)$, then by Lemma $\ref{lemma1}$, we have for any multi indices $\alpha, \beta \geq 0$ with $|\alpha|+|\beta| \leq r$, 
\begin{equation}
\begin{split}
&D_x^\beta (D_x + D_y)^{\alpha} I(x, y; t) \\
= &
\int_0^{\frac{t}{2}}  \int_M \sum_{\alpha_1+\alpha_2+ \alpha_3 =\alpha} D_x^\beta (D_x+D_\xi)^{\alpha_1} K(x, \xi; t-s) \times D_{\xi}^{\alpha_2}(\psi_\nu (\xi) \mathrm d V_g(\xi))\\
&\times (D_y + D_\xi)^{\alpha_3} K(\xi, y; s) \mathrm d s \\
& +  \int_{0}^{\frac{t}{2}} \int_M \sum_{\alpha_1+\alpha_2 = \alpha}D_{x}^{\alpha_1+\beta} K(x, \xi; t-s) (1- \psi_\nu(\xi)) D_y^{\alpha_2} K(\xi, y; s) \mathrm d V_g(\xi) \mathrm d s.
\end{split}
\end{equation}
Thus we have
\begin{equation}\label{eqn00002}
\begin{split}
& |D_x^\beta (D_x + D_y)^{\alpha} I(x, y; t) |\\
 \leq& C t^{-\frac{n+2 +|\beta|}{4}} \exp\{- \delta (t^{-\frac{1}{4}} \rho(x, y))^{\frac{4}{3}}\}+ C t^{-\frac{n+3+|\alpha|+|\beta|}{4}}  \exp\{- \delta (t^{-\frac{1}{4} } \rho(x, y))^{\frac{4}{3}}\} \\
& \times \int_{(M - B_{\frac{3}{4}r_0}(p_\eta, g) )\times (0, \frac{t}{2})} s^{\frac{|\alpha|-3}{4}} \big(s^{-\frac{n+2|\alpha|}{4}} \exp\{- (\delta_{1,r} - \delta) (s^{-\frac{1}{4} } \frac{1}{4} r_0)^{\frac{4}{3}}\} \big) \mathrm d V_g(\xi) \mathrm d s, \\
\leq & C'  t^{-\frac{n+2+|\beta|}{4}}\exp\{- \delta (t^{-\frac{1}{4} } \rho(x, y))^{\frac{4}{3}}\}.
\end{split}
\end{equation}
Constant $C' >0$ depends on $C_{1, r}, g, T, r$ and $(\delta_{1,r} -\delta)$. 

If $y \notin B_{\frac{1}{2}r_0}(p_\nu, g)$, say $y \in B_{\frac{1}{2}r_0}(p_\eta, g) \cap B_{r_0}(p_\nu, g)$ for some $\eta \neq \nu$ and moreover if $x \in B_{r_0}(p_\eta, g) \cap B_{r_0}(p_\nu, g)$, we denote $u,v \in B_{r_0}(0) \cong B_{r_0}(p_\eta, g)$ such that $u = \sigma(x)$ and $v = \sigma(y)$ where $\sigma$ denotes the coordinate transition function. Since $v = \sigma(y ) \in B_{\frac{1}{2} r_0}(p_\eta, g)$ by virtue of the estimate $(\ref{eqn00002})$, we have for $|\alpha|+ |\beta| \leq r$
\begin{equation}
|D_u^\beta (D_u + D_v)^{\alpha} I(u, v; t) | \leq C  t^{-\frac{n+2+|\beta|}{4}}\exp\{- \delta (t^{-\frac{1}{4} } \rho(u, v))^{\frac{4}{3}}\}. 
\end{equation}
Constant $C >0$ depends on $C_{1, r}, g, T, r$ and $(\delta_{1,r} -\delta)$. Using formula $(\ref{eqn00003})$, we have for $|\alpha| + |\beta| \leq r$, 
\begin{equation}
|D_x^\beta (D_x + D_y)^{\alpha} I(x, y; t) | \leq  C'  t^{-\frac{n+2+|\beta|}{4}}\exp\{- \delta' (t^{-\frac{1}{4} } \rho(x, y))^{\frac{4}{3}}\},
\end{equation}
where $\delta' < \delta$ and $C' \gg C$ depending on transition function $\sigma$. 

Suppose $y \in B_{\frac{1}{2}r_0}(p_\eta, g) \cap B_{r_0}(p_\nu, g)$ for some $\eta \neq \nu$ but $x \notin B_{r_0}(p_\eta, g)$. Since $v =\sigma (y)$ we have
\begin{equation}
\frac{\partial }{\partial y_l } = \sum_{k} \frac{\partial \sigma^k}{\partial y_l}(v) \frac{\partial }{\partial v_k},
\end{equation}
and by induction for some smooth functions $f_{\gamma}^\alpha$ on $v$, 
\begin{equation}
D_y^\alpha = \sum_{|\gamma| \leq |\alpha|} f^\alpha_\gamma(v) D_v^\gamma.
\end{equation}
Thus
\begin{equation}
\begin{split}
&D_x^\beta (D_x + D_y)^{\alpha} I(x, y; t) \\
= &
\int_0^{\frac{t}{2}}  \int_M \sum_{\alpha_1+ \alpha_2 =\alpha, |\gamma| \leq |\alpha_2|, \gamma_1 +\gamma_2 = \gamma}D_x^{\beta+\alpha_1} D_w^{\gamma_1}  (K(x, w; t-s) \psi_\eta (w) \mathrm d V_g(w))\\
&\times f_\gamma^{\alpha_2}(v) (D_v + D_w)^{\gamma_2} K(w, v; s) \mathrm d s \\
& +  \int_{0}^{\frac{t}{2}} \int_M \sum_{\alpha_1+\alpha_2 = \alpha}D_{x}^{\alpha_1+\beta} K(x, \xi; t-s) (1- \psi_\eta(\xi)) D_y^{\alpha_2} K(\xi, y; s) \mathrm d V_g(\xi) \mathrm d s.
\end{split}
\end{equation}
Consequently, we have
\begin{equation}
\begin{split}
& |D_x^\beta (D_x + D_y)^{\alpha} I(x, y; t) |\\
 \leq& C t^{-\frac{n+3 +|\beta|}{4}} \exp\{- \delta (t^{-\frac{1}{4}} \rho(x, y))^{\frac{4}{3}}\}\\
 & \times \int_{B_{\frac{7}{8}r_0}(p_\eta, g) \times (0, \frac{t}{2})} (t-s)^{-\frac{|\alpha|}{4}} \exp\{ - (\delta_{1,r} - \delta)\big((t-s)^{-\frac{1}{4}} \frac{1}{8} r_0 \big)^{\frac{4}{3}}\} \\
 &\times s^{-\frac{n+3}{4}} \exp\{-(\delta_{1, r} -\delta) (s^{-\frac{1}{4}} \rho(y, \xi))^{\frac{4}{3}}\} \mathrm d V_g(\xi) \mathrm d s \\
 &+ C t^{-\frac{n+3+|\alpha|+|\beta|}{4}}  \exp\{- \delta (t^{-\frac{1}{4} } \rho(x, y))^{\frac{4}{3}}\} \\
& \times \int_{(M - B_{\frac{3}{4}r_0}(p_\eta, g) )\times (0, \frac{t}{2})} s^{\frac{|\alpha|-3}{4}} \big(s^{-\frac{n+2|\alpha|}{4}} \exp\{- (\delta_{1,r} - \delta) (s^{-\frac{1}{4} } \frac{1}{4} r_0)^{\frac{4}{3}}\} \big) \mathrm d V_g(\xi) \mathrm d s, \\
\leq & C'  t^{-\frac{n+2+|\beta|}{4}}\exp\{- \delta (t^{-\frac{1}{4} } \rho(x, y))^{\frac{4}{3}}\}.
\end{split}
\end{equation}
Constant $C' >0$ depends on $C_{1, r}, g, T, r$ and $(\delta_{1,r} -\delta)$. To sum up, we can choose arbitrary $\delta_{2, r} < \delta_{1,r}$ and then constant $C_{2,r} \gg C_{1,r}$ will be determined by estimates above. This ends the proof of the lemma $\ref{lemma2}$. 
\end{proof}

Given $K_2(\cdot, \cdot; t) \in C^\infty(M \times M)$ for $0< t <T$ with estimates on its derivatives as in Lemma $\ref{lemma2}$, using the same argument one can show similar results for $K_3 = K * K_2$. In fact, by induction, one can prove that $K_m (\cdot, \cdot ; t) \in C^\infty(M\times M)$ for $0 < t < T$ and for any integer $r \geq 0$ there exists constants $\delta_{m, r}, C_{m, r} >0$ such that:
\begin{enumerate}
\item[i)] For any $x, y \in M$ and integers $p, q \geq 0$ with $p+q \leq r$, 
\begin{equation}
| \nabla_x^p \nabla_y^q K_m (x, y; t)|_g \leq C_{m,r} t^{-\frac{n+(4-m)+p+q }{4}} \exp\{- \delta_{m,r} \big(t^{-\frac{1}{4}} \rho(x, y)\big)^{\frac{4}{3}}\}.
\end{equation}
\item[ii)] For each $\nu$, in local normal coordinate $x, y \in B_{r_0}(0) \cong B_{r_0}(p_\nu, g)$, we have for any multi indices $\alpha, \beta \geq 0$ with $|\alpha|+|\beta| \leq r$
\begin{equation}
|D_x^{\beta} (D_x + D_y)^\alpha K_m (x,y;t)| \leq  C_{m,r} t^{-\frac{n+(4-m)+|\beta|}{4}} \exp\{- \delta_{m,r} \big(t^{-\frac{1}{4}}\rho(x,y)\big)^{\frac{4}{3}}\}. 
\end{equation}
\end{enumerate}

Lastly, we finish the proof of Proposition $\ref{prop503}$ by showing for fixed $0 < t <T$ and any integers $p,q \geq 0$, $\sum_{m=1}^\infty \nabla_{x}^p \nabla_y^q K_m(x,y ; t)$ is uniformly convergent for $x, y \in M$. Therefore one can interchange the orders of differentiation and summation and thus $\Psi(\cdot, \cdot ; t) = \sum_{m=1}^\infty K_m(\cdot, \cdot; t) \in C^\infty(M \times M)$.

However, the estimates on $K_m$'s that we have shown above are not sufficient to ensure uniform convergence of $\sum_{m=1}^\infty \nabla_{x}^p \nabla_y^q K_m(x,y ; t)$. It is because for fixed $r$, the sequence of positive constants $\delta_{m, r}$'s that we constructed inductively is strictly decreasing, namely
\begin{equation}
\delta_{1,r} > \delta_{2, r} > \cdots > \delta_{m, r} > \cdots. 
\end{equation}
Consequently constants $C_{m, r}$'s grow very fast with respect to $m$ under which we are unable to show the convergence of $\sum_{m=1}^\infty \nabla_{x}^p \nabla_y^q K_m(x,y ; t)$. When $r = 0$, we have shown a refined estimate for $K_m$'s at the beginning of the proof of Proposition $\ref{prop503}$, namely for $m \geq n+4$, we have $\delta_{m, 0} = \delta_{n+4, 0}$ and $C_{m,0} \leq C^m \Gamma(\frac{m-n}{4})^{-1}$. Similarly when $r>0$ we also need a refined estimate on derivatives up to $r$th-order of $K_m$'s for large $m$. More precisely, we introduce the following lemma.

\begin{lem}\label{lemma3}
Given any integer $r>0$, we set $N = n+4 + r$. Denote $\delta_3 = \delta_{r+4,r}$ and $C_3 = C_{r+4, r}$. Also denote $\delta_4 = \min\{\delta_{m, r} | m=1,2\cdots, N+2r+7\}$ and $C_4 = \max\{C_{m, r}| m=1,2\cdots, N+2r+7\}$. Clearly $\delta_4< \delta_3$ and $C_4 > C_3$. We have for any integers $p, q \geq 0$ with $p+q \leq r$, integer $l \geq 1$ and $a \in \{N, N+1, \cdots, N+r+3\}$, 
\begin{equation}
|\nabla_x^p \nabla_y^q K_{a+ (r+4) l } (x, y; t)|_g \leq \frac{C_4 D^{l-1}}{(l-1)!} t^{-\frac{n+4 + p+q - (a+(r+4)l) }{4}} \exp\{- \delta_4 (t^{-\frac{1}{4}} \rho(x, y))^{\frac{4}{3}}\}.
\end{equation}
Constant $D>0$ is given by
\begin{equation}
D = \frac{4C_3}{ r+4} \sup_{(x, t) \in M \times (0,\infty)} \int_{M} t^{-\frac{n}{4}} \exp\{- (\delta_3 -\delta_4) \big(t^{-\frac{1}{4}} \rho(x,\xi)\big)^{\frac{4}{3}}\} \mathrm d V_g(\xi) < \infty.
\end{equation}

\end{lem}

\begin{proof}
We prove this lemma by induction on $l \geq 1$.  When $l=1$, it is clearly true. We suppose for $l-1$ this lemma is true and next we show for $l$ it is also valid. Note $K_{a+(r+4)l} = K_{r+4} * K_{a+(r+4)(l-1)}$. Moreover, we have for any $p, q \geq 0$ with $p+q \leq r$,
\begin{equation}
\nabla_x^p \nabla_y^q K_{a+ (r+4)l}(x, y;t) =\int_0^t \int_M \nabla_{x}^p K_{r+4} (x, \xi; t-s) \nabla_y^q K_{a + (r+4)(l-1)}(\xi, y; s) \mathrm d V_g{\xi} \mathrm d s, 
 \end{equation}
since every term is integrable. Therefore,
\begin{equation}
\begin{split}
&|\nabla_x^p \nabla_y^q K_{a+ (r+4)l}(x, y;t) |_g \\
\leq&  C_3 C_4 D^{l-2} \frac{1}{(l-2)!}\int_0^t \int_M (t-s)^{-\frac{n}{4}} \exp\{- (\delta_3 -\delta_4) \big((t-s)^{-\frac{1}{4}} \rho(x,\xi)\big)^{\frac{4}{3}}\}\mathrm d V_g(\xi) \\ 
& \times (t-s)^{-\frac{4+p-(r+4)}{4}}s^{-\frac{n+4+q -\big(a + (r+4)(l-1)\big)}{4}}  \mathrm ds \times \exp\{- \delta_4 (t^{\frac{1}{4}} \rho(x, y))^{\frac{4}{3}}\}\\
\leq & C_4 \frac{D^{l-1}}{(l-2)!} \times t^{-\frac{n+4+p+q - \big(a+ (r+4) l\big)}{4}}  \exp\{- \delta_4 (t^{\frac{1}{4}} \rho(x, y))^{\frac{4}{3}}\}\\
& \times \frac{r+4}{4}\int_0^1 (1-s)^{\frac{r-p}{4} } s^{\frac{a- (n +4 + q) + (r+4)(l-1)}{4}}  \mathrm ds\\
 \leq &  \frac{C_4 D^{l-1}}{(l-1)!} t^{-\frac{n+4 + p+q - (a+(r+4)l) }{4}} \exp\{- \delta_4 (t^{-\frac{1}{4}} \rho(x, y))^{\frac{4}{3}}\}.
\end{split}
\end{equation}
This ends the proof of Lemma $\ref{lemma3}$.
\end{proof}

Given any integer $r>0$, we have that for any $p,q \geq 0$ with $p+q \leq r$, $\sum_{m=1}^\infty \nabla_x^p \nabla_y^q K_m(x, y; t)$ is uniformly convergent since by Lemma $\ref{lemma3}$
\begin{equation}
\begin{split}
& \sum_{m=1}^\infty |\nabla_x^p \nabla_y^q K_m(x, y; t)| \\
\leq & C_4 t^{-\frac{n+3 +p +q}{4}} \exp\{- \delta_4 (t^{-\frac{1}{4}} \rho(x, y))^{\frac{4}{3}}\} \big(\sum_{m=1}^{N+r+3} T^{\frac{m-1}{4}} +\sum_{a=N}^{N+r+3} T^{\frac{a+r+3}{4}} \sum_{l=1}^\infty \frac{(T^{r+4}D)^{l-1}}{(l-1)!}\big),\\
 \leq &  C t^{-\frac{n+3 +p +q}{4}} \exp \{- \delta(t^{-\frac{1}{4}} \rho(x, y))^{\frac{4}{3}} \},
\end{split}
\end{equation}
for appropriate constants $\delta =\delta_4$ and $C \gg C_4$ depending on $r, g,T$. Thus we have $\Psi = \sum_{m=1}^\infty K_m$ satisfies the Apriori restriction i). As for Apriori restriction ii), we have for any multi indices $\alpha,\beta \geq 0$ with $|\alpha|+|\beta| \leq r$ and for any $x, y \in B_{r_0}(0) \cong B_{r_0}(p_\nu, g)$, 
\begin{equation}
\begin{split}
& D_{x}^\beta (D_x+D_y)^\alpha \Psi (x, y; t) \\
= & \sum_{m=1}^{N+r+3} D_{x}^\beta (D_x+D_y)^\alpha K_m(x, y; t) + \sum_{\alpha_1+ \alpha_2 = \alpha} \sum_{a =N}^{N+r+3} \sum_{l=1}^\infty D_x^{\beta + \alpha_1} D_y^{\alpha_2} K_{a+ (r+4)l}(x, y; t).
\end{split}
\end{equation}
Thus
\begin{equation}
\begin{split}
 & |D_{x}^\beta (D_x+D_y)^\alpha \Psi (x, y; t)| \\
\leq & C_4 t^{-\frac{n+3+|\beta|}{4}} \exp\{- \delta_4 (t^{-\frac{1}{4}} \rho(x, y))^{\frac{4}{3}}\} \big(\sum_{m=1}^{N+r+3} T^{\frac{m-1}{4}} + C \sum_{a=N}^{N+r+3} T^{\frac{a+r+3 - |\alpha|}{4}} \sum_{l=1}^\infty \frac{(T^{r+4}D)^{l-1}}{(l-1)!}\big),\\
\leq & C' t^{-\frac{n+3+|\beta|}{4}} \exp\{- \delta (t^{-\frac{1}{4}} \rho(x, y))^{\frac{4}{3}}\}
\end{split}
\end{equation}
for constant $\delta= \delta_4$ and $C' \gg C_4$ depending on $g, T, r$. Thus it finishes the proof of Proposition $\ref{prop503}$. 

\end{proof}

We then start to prove Proposition $\ref{prop501}$. 

\begin{proof}
Proof of Proposition $\ref{prop501}$. Suppose $\Psi$ satisfies $(\ref{eqn505})$, then clearly we have for $k \leq 3$
\begin{equation}
\nabla_x^k (Z * \Psi)(x, y; t) = \int_0^t \int_M \nabla_x^k Z(x, \xi; t-s) \Psi(\xi, y; s) \mathrm d V_g(\xi) \mathrm ds, 
\end{equation}
since every term is integrable. While for the fouth derivatives($\nabla_x^4$ and $\frac{\partial }{\partial t}$) this formula no longer makes sense, since $\nabla_x^4 Z$ and $\frac{\partial }{\partial t} Z$ are not integrable. Therefore next we focus on deriving formula $(\ref{eqn00010})$ of fourth derivatives. Consider for $h \ll 1$ and $x \in B_{\frac{1}{2} r_0}(0)\cong B_{\frac{1}{2}r_0}(p_\nu, g)$ for some $\nu$, 
\begin{equation}
\begin{split}
& \nabla_x^3 (Z* \Psi) (x + h e_l, y; t) - \nabla_x^3 (Z* \Psi) (x, y; t) \\
=&  \int_0^t \int_M \int_0^h \frac{\partial }{\partial x_l}  \nabla_x^3 Z(x + \tau e_l , \xi; t-s) \Psi(\xi, y; s) \mathrm d \tau \mathrm d V_g(\xi) \mathrm ds, \\
=&  \int_0^{\frac{t}{2}} \int_M \int_0^h \frac{\partial }{\partial x_l}  \nabla_x^3 Z(x + \tau e_l , \xi; t-s) \Psi(\xi, y; s) \mathrm d \tau \mathrm d V_g(\xi) \mathrm ds\\
& + \int_{\frac{t}{2}}^t \int_M \int_0^h \frac{\partial }{\partial x_l}  \nabla_x^3 Z(x + \tau e_l , \xi; t-s) \big( \Psi(\xi, y; s) - \Psi(x + \tau e_l, y; s) \big) \mathrm d \tau \mathrm d V_g(\xi) \mathrm ds \\
& + \int_{\frac{t}{2}}^t \int_0^h \big( \int_M  \frac{\partial }{\partial x_l}  \nabla_x^3 Z(x + \tau e_l , \xi; t-s) \mathrm d V_g(\xi) \big)  \Psi(x + \tau e_l, y; s)  \mathrm d \tau \mathrm ds,  \\
= & \int_{0}^h \{ \int_{0}^{\frac{t}{2}} \int_M  \frac{\partial }{\partial x_l}  \nabla_x^3 Z(x + \tau e_l , \xi; t-s) \Psi(\xi, y; s)   \mathrm d V_g(\xi) \mathrm ds\\
& +\int_{\frac{t}{2}}^t \int_M  \frac{\partial }{\partial x_l}  \nabla_x^3 Z(x + \tau e_l , \xi; t-s) \big(\Psi(\xi, y; s) - \Psi(x+ \tau e_l, y; s) \big)  \mathrm d V_g(\xi) \mathrm ds \\
& + \int_{\frac{t}{2}}^t  \big( \int_M  \frac{\partial }{\partial x_l}  \nabla_x^3 Z(x + \tau e_l , \xi; t-s) \mathrm d V_g(\xi) \big)  \Psi(x + \tau e_l, y; s)  \mathrm ds \} \mathrm d \tau. 
\end{split}
\end{equation}
As we did before, it suffices to show that the integrand denoted by $Q_1(x+\tau e_l, y; t) \rightarrow Q_1(x, y; t)$ as $\tau \rightarrow 0$. In fact we have for any $x, x' \in B_{\frac{1}{2}r_0}(0)$, any $y \in M$ and $0<t<T$, 
\begin{equation}
| Q_1(x, y; t) - Q_1(x' , y; t) | \leq C t^{-\frac{n+3 + \gamma }{4}} \rho(x, x')^{\gamma}. 
\end{equation}
One can prove the above estimate by repeating the argument in the proof of Theorem $\ref{thm701}$ with appropriate modifications. Thus we have the formula of fourth derivatives of $(Z*\Psi)$ with respect to $x$ as in $(\ref{eqn00010})$.

As for the derivative of $(Z*\Psi)$ with respect to $t$, we consider $0< h \ll t$,
\begin{equation}
\begin{split}
& (Z*\Psi)(x, y; t+h) - (Z*\Psi)(x, y; t)\\
= & \int_t^{t+h} \int_M Z(x, \xi; t+h-s) \Psi(\xi, y; s) \mathrm d V_g(\xi) \mathrm ds \\
&+ \int_0^t \int_M \big(Z(x, \xi; t+h-s) - Z(x, \xi; t-s)\big) \Psi(\xi, y; s) \mathrm d V_g(\xi) \mathrm d s.
\end{split}
\end{equation}
By a change of variable $\tau = t+h-s$, the first integral can be written as
\begin{equation}
\begin{split}
&\int_0^h \int_M Z(x,\xi; \tau) \Psi(\xi, y; t+h-\tau) \mathrm d V_g(\xi) \mathrm d \tau\\
=& \int_0^h \{ \int_M Z(x,\xi; \tau) \Psi(\xi, y; t) \mathrm d V_g(\xi)\\
& + \int_M Z(x,\xi; \tau) \big( \Psi(\xi, y; t+h-\tau)  - \Psi(\xi, y; t) \big) \mathrm d V_g(\xi) \} \mathrm d \tau. 
\end{split}
\end{equation}
Since $\Psi \in C^0\big(M \times M \times (0,T)\big)$, by Proposition $\ref{prop302}$ iii) we have the integrand goes to $\Psi(x, y; t)$ uniformly for $\tau < h$ as $h \rightarrow 0$. Thus when $h \rightarrow 0$, the first integral equals to $h \big( \Psi(x, y; t) + o(1) \big) $. While the second integral equals to
\begin{equation}
\begin{split}
& \int_0^h \{ \int_0^{\frac{t}{2}} \int_M (\frac{\partial }{\partial t} Z) (x,\xi; t+\tau -s) \Psi(\xi, y; s) \mathrm dV_g(\xi) \mathrm d s \\
&+ \int_{\frac{t}{2}}^t  \int_M (\frac{ \partial }{\partial t}Z) (x, \xi; t+\tau-s) [\Psi(\xi, y; s) - \Psi(x, y;s)] \mathrm d V_g(\xi) \mathrm d s\\
& + \int_{\frac{t}{2}}^t \big(\int_M (\frac{\partial }{\partial t}Z) (x, \xi; t+\tau-s) \mathrm d V_g(\xi) \big)  \Psi(x, y;s) \mathrm ds \} \mathrm d \tau. 
\end{split}
\end{equation}
Again it suffices to show that the integrand $Q_2(x, y; t+ \tau) \rightarrow Q_2(x, y; t)$ as $\tau \rightarrow 0$. In fact we have for $x, y\in M$ and $\tau \ll t$
\begin{equation}
\begin{split}
& Q_2(x,y; t+\tau) - Q_2(x,y; t) \\
=  & \int_0^\tau\{ \int_0^{\frac{t}{2}} \int_M (\frac{\partial^2 }{\partial t^2} Z) (x,\xi; t+\epsilon -s) \Psi(\xi, y; s) \mathrm dV_g(\xi) \mathrm d s \\
&+ \int_{\frac{t}{2}}^t  \int_M (\frac{ \partial^2 }{\partial t^2}Z) (x, \xi; t+\epsilon-s) [\Psi(\xi, y; s) - \Psi(x, y;s)] \mathrm d V_g(\xi) \mathrm d s\\
& + \int_{\frac{t}{2}}^t \big(\int_M (\frac{\partial^2 }{\partial t^2}Z) (x, \xi; t+\epsilon-s) \mathrm d V_g(\xi) \big)  \Psi(x, y;s) \mathrm ds\} \mathrm d \epsilon.
\end{split}
\end{equation}
Following a similar argument we used to prove the continuity of $\Psi$ with respect to $t$, we can show that
\begin{equation}
| Q_2(x,y; t+\tau) - Q_2(x,y; t) | \leq C (t^{-\frac{n+7}{4}} |\tau| + t^{-\frac{n+3 + \gamma}{4}} |\tau|^{\frac{\gamma}{4}}). 
\end{equation}
 Combining results about first and second integral above, we have that $(Z*\Psi)$ is differentiable with respect to $t$ and its derivative is given by formula $(\ref{eqn00010})$. Thus it ends the proof of Proposition $\ref{prop501}$. 
\end{proof}

Lastly, we prove Proposition $\ref{prop502}$. 

\begin{proof}
The proof of Proposition $\ref{prop502}$. Since $Z$ satisfies i) iv) in Proposition $\ref{prop302}$ and $\Psi$ satisfies Apriori restrictions i) ii), clearly $(Z*\Psi) (\cdot, \cdot; t) \in C^{\infty}(M \times M)$ for any $0<t<T$. And estimates i) ii) on derivatives of $(Z*\Psi)$ with respect to $x, y$ in Proposition $\ref{prop502}$ hold. One can prove this using ideas developed in the proofs of Lemma $\ref{lemma1}$ and Lemma $\ref{lemma2}$. 

The only thing left to show here is that $\nabla_x^p \nabla_y^q (Z*\Psi)(x,y; \cdot) \in C^0 \big((0,T)\big)$. Suppose $x \in B_{\frac{r_0}{2}}(p_\nu, g)$ and $y \in B_{\frac{r_0}{2}}(p_\eta, g)$. Then we have for $0 <  h \ll t$
\begin{equation}
\begin{split}
 &\nabla_x^p \nabla_y^q (Z*\Psi)(x,y; t+h) - \nabla_x^p \nabla_y^q (Z*\Psi)(x,y; t)\\
=& \int_0^{\frac{t}{2}} \int_{B_{r_0}(p_{\eta}, g)} \big( \int_0^h \frac{\partial }{\partial t}E_1(x,\xi; t+\tau-s) \mathrm d \tau \big)\star F_1(\xi, y; s)\mathrm d V_g(\xi) \mathrm ds \\
&+ \int_0^{\frac{t}{2}} \int_{M - B_{\frac{3}{4} r_0} (p_\eta, g)} \big( \int_0^h \frac{\partial }{\partial t}E_2(x,\xi; t+\tau-s) \mathrm d \tau \big)\star F_2(\xi, y; s) \mathrm d V_g(\xi) \mathrm ds\\
&+ \int_{\frac{t}{2}}^{t} \int_{B_{r_0}(p_\nu, g)} \big( \int_0^h \frac{\partial }{\partial t}E_3(x,\xi; t+\tau-s) \mathrm d \tau \big)\star F_3 (\xi, y; s)\mathrm d V_g(\xi) \mathrm ds \\
&+  \int_{\frac{t}{2}}^{t} \int_{M - B_{\frac{3}{4} r_0} (p_\nu, g)} \big( \int_0^h \frac{\partial }{\partial t}E_2(x,\xi; t+\tau-s) \mathrm d \tau \big)\star F_2(\xi, y; s) \mathrm d V_g(\xi) \mathrm ds\\
&+ \int_{t}^{t+h} \int_{B_{r_0}(p_\nu, g)}  E_3(x, \xi; t+h -s) \star F_3 (\xi, y; s) \mathrm d V_g(\xi) \mathrm ds\\
&+  \int_{t}^{t+h} \int_{M - B_{\frac{3}{4} r_0} (p_\nu, g)}E_2(x, \xi; t+h-s) \star F_2(\xi, y; s)\mathrm d V_g(\xi) \mathrm ds.
\end{split}
\end{equation}
``$A \star B$" abbreviate for finite sum of the products of the form $ \sum _{A, B} A \times B \times$ (smooth functions in $\xi$). $E_i$'s and $F_i's$ are partial derivatives of $Z$ and $\Psi$ with respect to $x,y$ respectively. They satisfy estimates 
\begin{equation}
\begin{split}
&|\frac{\partial }{\partial t} E_1(x,y; t)| \leq C t^{-\frac{n+4+p+q}{4}} \exp\{- \delta (t^{-\frac{1}{4}} \rho(x,y))^{\frac{4}{3}} \},\\
&|\frac{\partial }{\partial t} E_2(x,y; t)| \leq C t^{-\frac{n+4+p}{4}} \exp\{- \delta (t^{-\frac{1}{4}} \rho(x,y))^{\frac{4}{3}} \},\\
&|\frac{\partial }{\partial t} E_3(x,y; t)| \leq C t^{-\frac{n+4}{4}} \exp\{- \delta (t^{-\frac{1}{4}} \rho(x,y))^{\frac{4}{3}} \},\\
&|E_2(x,y; t)| \leq C t^{-\frac{n+p}{4}} \exp\{- \delta (t^{-\frac{1}{4}} \rho(x,y))^{\frac{4}{3}} \},\\
&| E_3(x,y; t)| \leq C t^{-\frac{n}{4}} \exp\{- \delta (t^{-\frac{1}{4}} \rho(x,y))^{\frac{4}{3}} \},\\
&|F_1(x,y; t)| \leq C t^{-\frac{n+3}{4}} \exp\{- \delta (t^{-\frac{1}{4}} \rho(x,y))^{\frac{4}{3}} \},\\
&|F_2(x,y; t)| \leq C t^{-\frac{n+3+q}{4}} \exp\{- \delta (t^{-\frac{1}{4}} \rho(x,y))^{\frac{4}{3}} \},\\
&|F_3(x,y; t)| \leq C t^{-\frac{n+3+p+q}{4}} \exp\{- \delta (t^{-\frac{1}{4}} \rho(x,y))^{\frac{4}{3}} \}.\\
\end{split}
\end{equation}
One can estimate each integral using bounds above and get
\begin{equation}
\begin{split}
&|\nabla_x^p \nabla_y^q (Z*\Psi)(x,y; t+h) - \nabla_x^p \nabla_y^q (Z*\Psi)(x,y; t)| \\
\leq & C( t^{-\frac{n+3 +p+q}{4} } h + |\text{Ln}   t|  t^{-\frac{n+3 +p+q}{4} } h +   t^{-\frac{n+3 +p+q}{4} } h(1-\text{Ln} h) ). 
\end{split}
\end{equation}
Thus we see that $\nabla_x^p \nabla_y^q (Z*\Psi)(x,y; \cdot) \in C^0 \big((0,T)\big)$. This ends the proof of Proposition $\ref{prop502}$.
\end{proof}

\end{document}